\documentclass[10pt]{amsart}

\usepackage{amsmath}
\usepackage{amsfonts}
\usepackage{amscd}
\usepackage{amsthm}
\usepackage{amssymb}
\usepackage{mathrsfs}
\usepackage{enumerate}
\usepackage{bbm}
\usepackage{bm}
\usepackage{multirow} 
\usepackage{makecell}
\usepackage{booktabs, here}
\usepackage{indentfirst}
\usepackage{epsfig}
\usepackage{latexsym}
\usepackage{float}
\usepackage{epstopdf}
\usepackage{color}
\usepackage{enumitem}
\usepackage{cite}

\usepackage{extpfeil}

\usepackage{hyperref}

\numberwithin{equation}{section}

\title[Inertia perturbation of the Kuramoto model]{Inertia perturbation theory for the inertial Kuramoto model}

\author[Cho]{Hangjun Cho}
\address[Hangjun Cho]{\newline Department of Mechanical Engineering \newline University of Washington, Seattle, WA 98195, United States of America}
\email{hangjun@uw.edu}

\author[Dong]{Jiu-Gang Dong}
\address[Jiu-Gang Dong]{\newline School of Mathematical Sciences \newline Dalian University of Technology, Dalian 116024, People's Republic of China}
\email{jgdong@dlut.edu.cn}

\author[Ha]{Seung-Yeal Ha}
\address[Seung-Yeal Ha]{\newline Department of Mathematical Sciences and Research Institute of Mathematics \newline Seoul National University, Seoul 08826, Republic of Korea}
\email{syha@snu.ac.kr}

\author[Ryoo]{Seung-Yeon Ryoo}
\address[Seung-Yeon Ryoo]{\newline The Division of Physics, Mathematics and Astronomy \newline California Institute of Technology, Pasadena California 91125, United States of America}
\email{sryoo@caltech.edu}

\newtheorem{theorem}{Theorem}[section]
\newtheorem{lemma}{Lemma}[section]

\newtheorem{proposition}{Proposition}[section]
\newtheorem{remark}{Remark}[section]
\newtheorem{definition}{Definition}[section]

\newtheorem{question}{Question}[section]
\newtheorem{conjecture}{Conjecture}[section]

\newcommand{\bbr}{\mathbb{R}}

\allowdisplaybreaks
\begin{document}

\date{\today}


\thanks{\textbf{Acknowledgment.} 
The work of H. Cho was partially supported by the National Research Foundation (NRF) of Korea grant funded by the Korea government (MSIT) (RS-2023-00253171), the work of J.-G. Dong was supported by the National Natural Science Foundation of China through grant 12171069, the work of S.-Y. Ha was supported by NRF grant (NRF-RS2025-00514472), and the work of S.-Y. Ryoo was partially supported by the  Korea Foundation for Advanced Studies.}

\begin{abstract}
In this work, we study the inertial Kuramoto model, which is a second-order extension of the classical first-order Kuramoto model, as an inertial perturbation of the first-order Kuramoto model. We develop a quantitative Tikhonov theorem, from which we derive a new synchronization statement in the small inertia regime, with strong bounds on the limiting order parameter. We also explore the determinability of phase velocities from phase positions, which shows that the perturbation viewpoint must be limited to the small inertia regime.
This paper complements our recent work (2025), where we established asymptotic phase-locking of inertial Kuramoto oscillators under generic initial conditions in the low inertia-high coupling regime.
\end{abstract}

\maketitle

\tableofcontents

%
%
%
%
\section{Introduction} \label{sec:1}
\setcounter{equation}{0}

\subsection{Formulation of the models and previous work}
We fix a positive integer $N$ representing the number of oscillators, and for each $i\in[N]\coloneqq \{1,\cdots,N\}$, let $\theta_i=\theta_i(t)\in\mathbb{R}$ denote the phase and $\omega_i=\dot\theta_i$ the (instantaneous) frequency of the $i$-th oscillator, where both are defined as real-valued functions of time $t\geq0$. The evolution of the phase variables $\{\theta_i\}_{i=1}^N$ under the \emph{inertial Kuramoto dynamics} is described by the following Cauchy problem:
\begin{equation}
\begin{cases} \label{A-1}
\displaystyle m \ddot\theta_i + \dot\theta_i = \nu_i +\frac{\kappa}{N}\sum_{j=1}^N \sin(\theta_j - \theta_i),\quad t > 0,\\
\displaystyle (\theta_i, {\dot \theta}_i) \Big|_{t = 0+} = (\theta_i^0, \omega_i^0), \quad i\in [N].
\end{cases}
\end{equation}
Here, the constants $m$ and $\kappa$ denote the (uniform\footnote{The uniformity means that we are prescribing the same $m$ to each $\theta_i$; one may consider a model where each particle $\theta_i$ has its own inertia $m_i$.} positive) \emph{inertia} and \emph{coupling strength}, respectively, and are both assumed to be nonnegative; the real number $\nu_i$ denotes the natural (intrinsic) frequency of the $i$-th oscillator for $i\in [N]$.

In comparison, the \emph{first-order Kuramoto model,} originally proposed by Kuramoto in \cite{Ku2,Ku}, is the model formally obtained from \eqref{A-1} by taking zero inertia $m=0$:
\begin{align}\label{A-2}
	\begin{cases}
		\displaystyle \dot\theta_i = \nu_i + \frac{\kappa}{N}\sum_{j=1}^N \sin(\theta_j - \theta_i),\quad t > 0,\\
		\displaystyle \theta_i\Big|_{t = 0+} = \theta_i^0,\quad i\in [N].
	\end{cases}
\end{align}

By the classical Cauchy-Lipschitz theory, both Cauchy problems \eqref{A-1} and \eqref{A-2} admit global unique solutions. Moreover, due to the analyticity of the vector fields associated with the systems, the solutions are real analytic in time and with respect to all parameters, as guaranteed by the Cauchy-Kovalevskaya theorem. Therefore, issues such as uniqueness, global existence, and (qualitative) smoothness of solutions will not be a concern in this paper.

The Cauchy problem \eqref{A-2} was proposed as a pioneering mathematical model for {\it synchronization}, which broadly refers to the adjustment of rhythms in coupled oscillatory systems, or the emergence of consensus in phases and frequencies among interacting units. The main goal of studying synchronization models is to rigorously describe \emph{emergent behavior}, which refers to the progressive development of synchronization from initially uncoordinated states; this is frequently observed in nature, for example, in fish swarming, bird flocking, or bacterial aggregation \cite{A-B-F-H-P-P-J,B-B,P-R-K,T-B,V-Z}. Despite its ubiquity, a rigorous mathematical treatment of synchronization was not initiated until about fifty years ago by Winfree \cite{Wi2,Wi1} and Kuramoto \cite{Ku2,Ku} who proposed simplified phase models describing the collective dynamics of weakly coupled limit-cycle oscillators and revealed essential synchronization mechanisms.

The first-order Kuramoto model \eqref{A-2} is often used as a prototype synchronization system, on which various modifications are made to better describe a physical system of interest. The inertial Kuramoto model \eqref{A-1} is a modification of the first-order Kuramoto model \eqref{A-2} and was introduced independently\footnote{A cursory literature search suggest that the model of Bergen and Hill \cite{B-H} in 1981 was inspired by \cite{A-F,Fo,R-P}, each published in 1977, 1975, 1971, respectively, which seem to precede the works of Kuramoto \cite{Ku,Ku2} in 1975 and 1984, respectively. The work of Ermentrout \cite{Er} does not cite Kuramoto \cite{Ku,Ku2} directly nor refers to the work of Bergen and Hill \cite{B-H}, but refers to the work of Strogatz and Mirollo \cite{S-M} which concerns itself with the Kuramoto model \eqref{A-2}. We cautiously suspect that the inertial Kuramoto model arose independently in the power systems community and the mathematical biology community, with the latter being directly influenced by Kuramoto. We do not claim with absolute certainty that \cite{B-H,Er} are the definitive progenitors for the inertial Kuramoto model \eqref{A-1} as we are not experts in mathematical history; these are simply the first sources we could find regarding the inertial Kuramoto model \eqref{A-1}.}
by Bergen and Hill \cite{B-H}  to describe the behavior of electric power grids with generators, and by Ermentrout \cite{Er} to model slow synchronous flashing in the firefly species Pteroptyx malaccae. This coupled oscillator system was further studied in network dynamics \cite{Rodrigues16} and was also used to formulate the swing equation of power network dynamics \cite{C, Kundur}. Due to its second-order nature, the inertial Kuramoto model possesses several novel features absent in the first-order Kuramoto model such as first-order hysteretic transition \cite{T-L-O1,T-L-O2, Olmi14,Barre16}, while maintaining interesting behavior such as bifurcation \cite{C-M22,C-M-M23}.

A central feature of both the Winfree and Kuramoto models is the presence of phase transitions: systems evolve from incoherent (disordered) states to partially synchronized (partially phase-locked) states, and eventually to fully synchronized (completely phase-locked) states, as the coupling strength $\kappa$, quantifying the interaction among oscillators, crosses critical thresholds \cite{A-S,Cr,K-B}. This transition behavior has made these models influential, and their mathematical framework has attracted significant attention from researchers in control theory, neuroscience, and statistical physics \cite{A-B,B-T11,B-T13,D-B14,Ermentrout19,H-K-P-Z,Hoppensteadt12,Rodrigues16,Str}. In the context of power systems, synchronization is interpreted as transient stability and was investigated in \cite{D13,D-B10,D-B12,D-B14,F-N-P,G-Z-L-W,S-U-S-P}.

The asymptotic behavior of the Kuramoto model \eqref{A-2} and its variants has been extensively studied in the literature \cite{Aeyels04,Bronski12,Chopra09,C15,C-H-J-K,D16,D-B11,D-B14,D-H-K,D-X,L-X-Y2,M-S07,Mirollo05,H-K-M-P,H-K-R,H-R,Jadbabaie04,Van93}. These works have established several sufficient conditions for \emph{complete synchronization}, where all oscillators asymptotically align their frequencies. Such results are typically obtained under certain assumptions: for instance, when the coupling strength is sufficiently large, the natural frequencies $\nu_i$ are heterogeneous, and the initial phases are confined to a half-circle; or when the natural frequencies are identical and the initial conditions are allowed to be generic.
Numerical simulations suggest that complete synchronization is still expected for generic initial data even when the natural frequencies $\{\nu_i\}$ are distributed, provided the coupling strength is large enough \cite{H-R}. 
More recently, works such as \cite{H-K-R,H-R} have rigorously derived sufficient conditions for phase-locking in the first-order Kuramoto model \eqref{A-2}, employing a gradient flow perspective and developing refined estimates on the phase diameter and the order parameter.

The complete synchronization problem (or asymptotic phase-locking) has been investigated for the inertial Kuramoto model \eqref{A-1} for restricted initial configurations \cite{Choi15, C-H-M, C-H-Y1, C-L, D-B12, D-B11, H-J-K, L-X, L-X-Y, W-Q}.
Numerical simulations suggest that the inertial Kuramoto model exhibits asymptotic phase-locking (as defined in Definition  \ref{D1.1}) for generic initial data in the large coupling regime. We are led to formulate the following conjecture.

\begin{conjecture}[The pathwise critical coupling strength and critical coupling strength coincide]\label{conj:coincide}
	For fixed natural frequencies $\nu_1,\ldots,\nu_N$, there exists a finite value $\kappa_c = \kappa_c(\nu_1,\cdots,\nu_N)$ such that the following hold.
	\begin{enumerate}
		\item If $\kappa\ge \kappa_c$, then for Lebesgue a.e.~initial data $(\Theta^0,\Omega^0)$, system \eqref{A-1} exhibits asymptotic phase-locking.
		\item If $\kappa< \kappa_c$, then system \eqref{A-1} does not possess phase-locked states.
	\end{enumerate}
\end{conjecture}

Less ambitiously, we may ask the following weakened question.
\begin{question}\label{ques:suff-coupling}
	\emph{(Existence of sufficient coupling strength)}~ For a fixed natural frequency vector ${\mathcal V}=(\nu_1,\ldots,\nu_N)$ and Lebesgue a.e.~initial data $(\Theta^0,\Omega^0)$, does there exist a finite value $\kappa_c = \kappa_c({\mathcal V}, \Theta^0, \Omega^0)$ such that for $\kappa>\kappa_c$,  system \eqref{A-1} exhibits
	asymptotic phase-locking, regardless of the inertia $m>0$?
\end{question}

Our previous work \cite{C-D-H-R25} affirmatively answered this question for small inertia $m>0$, by employing a second-order analogue of the techniques from \cite{H-R} and analyzing Lyapunov functionals. We provided sufficient conditions for asymptotic phase-locking under \emph{generic} initial conditions. The precise statement is as follows, where the \emph{order parameter} of models \eqref{A-1} and \eqref{A-2} is defined as follows:
\[
R(t) \coloneqq \left|\frac{1}{N} \sum_{j=1}^N e^{\mathrm{i}\theta_j(t)} \right|, \quad R^0\coloneqq R(0).
\]

\begin{theorem}[{\cite[Theorem 1.1]{C-D-H-R25}}]
	There is a small universal constant $c>0$ such that
	if the initial configuration and system parameters satisfy
	\begin{align}
		\begin{aligned} \label{eq:R^0}
			&  \max \Bigg \{ \frac{\displaystyle\max_{i,j\in[N]}|\nu_i-\nu_j|}{\kappa},~~\frac{\displaystyle\max_{i,j\in[N]}|\omega_i^0-\omega_j^0|}{\kappa},~~m\kappa \Bigg \} <c (R^0)^2,
		\end{aligned}
	\end{align}
	then the solution to \eqref{A-1} converges to a single traveling wave solution: there exist $\theta_1^\infty,\cdots,\theta_N^\infty\in \mathbb{R}$ such that
	\[
	\lim_{t\to\infty } \left(\theta_i(t) - \frac{1}{N}\sum_{i=1}^N\nu_i t\right) = \theta_i^\infty,
	\quad
	\lim_{t\to\infty} \dot{\theta}_i(t) = \frac{1}{N}\sum_{i=1}^N\nu_i ,\qquad i=1,\cdots,N.
	\]
\end{theorem}
The condition $R^0 > 0$ holds for almost every initial phase configuration $\Theta^0\in\mathbb{R}^N$ with respect to the Lebesgue measure. This is why the condition is said to be {\it generic}.

There, one philosophy was to consider the inertial Kuramoto model \eqref{A-1} as a perturbation of the original Kuramoto model \eqref{A-2} when $m$ is small, and apply the known techniques for proving asymptotic phase-locking for \eqref{A-2}. A natural question is 
\begin{quote}
	\begin{center}
		``In what sense is \eqref{A-1} a perturbation of \eqref{A-2}?" 
	\end{center}
\end{quote}
The goal of this paper is to consider two different perspectives on this matter and to discuss the optimality of such an approximation. The results of our investigation, as well as a by-product synchronization theorem, are described in the next subsection.

\subsection{Showcase of the new results}
In this subsection, we outline the new results obtained in this paper. We first begin by stating the main new result of this paper, namely a synchronization statement given as Theorem \ref{thm:qualitative_mto0} below. It can be considered a variant of \cite[Theorem 1.1]{C-D-H-R25} with better control on the limiting order parameter. 
\begin{theorem}\label{thm:qualitative_mto0}
	Let the initial position data $\theta_1^0,\cdots,\theta_N^0\in \mathbb{R}$ be such that $R^0>0$, and let $\varepsilon>0$. Then, there exist sufficiently small numbers $a,b,c>0$ depending only on $\theta_1^0,\cdots,\theta_N^0$ and $\varepsilon$ such that if the initial velocity data $\omega_1^0,\cdots,\omega_N^0\in \mathbb{R}$ and system parameters $\kappa,m>0$, and $\nu_1,\cdots,\nu_N\in \mathbb{R}$ satisfy
	\[
	\max_{i,j\in[N]}|\nu_i-\nu_j|/\kappa< a,\quad \max_{i,j\in[N]}|\omega_i^0-\omega_j^0|/\kappa< b,\quad m\kappa <c,
	\]
	then the solution $\theta_1(t),\cdots,\theta_N(t)$ to the Cauchy problem \eqref{A-1} achieves \emph{asymptotic phase-locking}: there exist $\theta_1^\infty,\cdots,\theta_N^\infty\in \mathbb{R}$ such that
	\[
	\lim_{t\to\infty } \left(\theta_i(t) - \frac{1}{N}\sum_{i=1}^N\nu_i t\right) = \theta_i^\infty,
	\quad
	\lim_{t\to\infty} \dot{\theta}_i(t) = \frac{1}{N}\sum_{i=1}^N\nu_i ,\qquad i=1,\cdots,N.
	\]
	Furthermore, we have the following lower bound for the limiting order parameter:
	\begin{align*}
		\lim_{t\to\infty}R(t)>
		\begin{cases}
			1-\varepsilon&\mathrm{if~}N=2,\\
			1-\frac 2N-\varepsilon&\mathrm{if~}N=3\mathrm{~or~}\eqref{eq:init-nonidentical}\mathrm{~holds},\\
			R^0-\varepsilon&\mathrm{otherwise},
		\end{cases}
	\end{align*}
	where \eqref{eq:init-nonidentical} is the following condition:
	\begin{equation}\label{eq:init-nonidentical}
		\theta_i^0\not\equiv \theta_j^0 \mod 2\pi,\quad \forall i\neq j\in [N].
	\end{equation}
\end{theorem}
The proof of Theorem \ref{thm:qualitative_mto0} is given in Section \ref{sec:4}.

Theorem \ref{thm:qualitative_mto0} is a partial answer to the first part of the following conjecture posed in \cite{C-D-H-R25} after some numerical investigations. Here, $\operatorname{Var}(\mathcal{V})$ means the variance of the vector $\mathcal{V}=(\nu_1,\cdots,\nu_N)$ (see Section \ref{sec:notations}).
\begin{conjecture}[{\cite[Conjecture 3.2]{C-D-H-R25}}]\label{conj:R}
	\,
	\begin{enumerate}
		\item (Weak form) For any $\varepsilon>0$, there exists $\delta=\delta(\varepsilon)$ such that if $\max_{i,j\in[N]}|\nu_i-\nu_j|/\kappa<\delta$, then for generic initial data $\Theta^0$, we have
		\[
		\liminf_{t\to\infty}R(t)\ge 1-\varepsilon.
		\]
		\item (Strong form) For any $\varepsilon \in (0,\frac 12)$, there exists a sufficiently small constant $c_\varepsilon>0$ such that $\operatorname{Var}(\mathcal{V})/\kappa<c_\varepsilon$, then for generic initial data $\Theta^0$, we have
		\[
		1-\left(\frac 12+\varepsilon\right)\frac{\operatorname{Var}(\mathcal{V})}{\kappa^2} \le \liminf_{t\to\infty} R(t)\le \limsup_{t\to\infty} R(t)\le 1-\left(\frac 12-\varepsilon\right)\frac{\operatorname{Var}(\mathcal{V})}{\kappa^2}.
		\]
		If, in addition, $\max_{i,j\in[N]}|\nu_i-\nu_j|/\kappa<c_\varepsilon$, then $\{\theta_i(t)\}_{i=1}^N$ converges to the unique phase-locked state of \eqref{A-2} confined in the quarter circle \cite{C-H-J-K}.
	\end{enumerate}
\end{conjecture}

We now describe the two perspectives on viewing the inertial Kuramoto model \eqref{A-1} as a perturbation of the first-order Kuramoto model \eqref{A-2}. Our first perspective is a quantitative higher-order Tikhonov statement, namely that the solution to \eqref{A-1} converges to the solution to \eqref{A-2} as $m\to 0+$ in the $C[0,\infty)$ and $C^\infty(0,\infty)$ topologies (but not necessarily in the $C^1[0,\infty)$ topology). More precisely, we have the following statement.

\begin{theorem}\label{thm:tikhonov}
	Fix initial data $\theta_1^0,\cdots,\theta_N^0,\omega_1^0,\cdots,\omega_N^0$, intrinsic velocities $\nu_1,\cdots,\nu_N$, and a coupling strength $\kappa>0$. For each $m>0$, denote by $\theta_1(m,t),\cdots,\theta_N(m,t)$ the solution to the Cauchy problem \eqref{A-1}, and denote by $\theta_1(0,t),\cdots,\theta_N(0,t)$ the solution to the Cauchy problem \eqref{A-2}. Then, the following statements hold.
	\begin{enumerate}
		\item 
		For $i\in [N]$, $\theta_i(m,t)$ converges to $\theta_i(0,t)$ as $m\to 0$ in the Fr\'echet topology $C^0([0,\infty))$. Quantitatively, we have
		\[
		|\theta_i(m,t)-\theta_i(0,t)|\le m \left(2\max_{i\in[N]}|\omega_i^0-\nu_i|+\kappa\right)e^{2\kappa t},\quad i\in [N], ~t\ge 0.
		\]
		\item 
		For $i\in [N]$, $\theta_i(m,t)$ converges to $\theta_i(0,t)$ as $m\to 0$ in the Fr\'echet topology $C^\infty(0,\infty)$. Quantitatively, for $n\ge 1$, we have for $i\in [N]$ and $t\ge 0$,
		\begin{align*}
		&|\theta^{(n)}_i(m,t)-\theta^{(n)}_i(0,t)|\\
		&\le (n+1)! 2^n\left(\kappa+\max_{i\in[N]}|\omega_i^0|+\max_{i\in[N]}|\nu_i|+\frac 1m\right)^n\left(\left(1+\frac tm\right)^ne^{-t/m}(1+m\kappa e^{2\kappa t})+m\kappa e^{2\kappa t}\right)
		\end{align*}

	\end{enumerate}
\end{theorem}

See Proposition \ref{prop:mto0} for a na\"ive implementation of the classical Tikhonov theorem, and Proposition \ref{prop:mto0quant} for a detailed version of Theorem \ref{thm:tikhonov}.

Another sense in which \eqref{A-1} can be approximated to \eqref{A-2} is through a dynamic approximation scheme, as established the author's previous work \cite{C-D-H-R25}:
\begin{lemma}[{\cite[Lemma 2.3]{C-D-H-R25}}]\label{L:approxaut}
	Let $(\Theta,\Omega)$ be a global solution to \eqref{A-1}. For $t\ge 0$ and $i\in [N]$, we have
		\begin{align*}
			&\left|\dot{\theta}_i(t)-\omega_i^0 e^{-t/m}-\nu_i\left(1-e^{-t/m}\right)-\frac \kappa N\sum_{l=1}^N\sin\left(\theta_l(t)-\theta_i(t)\right)\left(1-e^{-t/m}\right)\right|\\
			& \hspace{0.5cm} \le \kappa \left(\max_{j,k\in[N]}|\omega_j^0-\omega_k^0|\right)te^{-t/m}\left(1-e^{-t/m}\right) +m\kappa (\max_{j,k\in[N]}|\nu_j-\nu_k|+2\kappa)\left(1-e^{-t/m}\right)^3.
		\end{align*}
\end{lemma}

\noindent Since obtaining a sharp lower bound on the derivative of the order parameter is essential\footnote{Particularly for the initial layer and condensation steps in the previous result \cite{C-D-H-R25}} one may naturally ask about the optimality of the above estimate. This raises a question \cite[Question 2.1]{C-D-H-R25}: is there a better way to approximate $\dot\theta_i(t)$ with $\theta_i(t)$ and $\omega_i^0$? Is it even possible to \emph{determine} $\dot\theta_i(t)$ from $\theta_i(t)$ and $\omega_i^0$? To this, we give the following answer.

\begin{theorem}\label{thm:determinability}
	Let $\kappa,m,t^*>0$. Then, in system \eqref{A-1}, $\theta_1(t^*),\cdots,\theta_N(t^*),\omega_1^0,\cdots,\omega_N^0$ uniquely determine $\dot\theta_1(t^*),\cdots,\dot\theta_N(t^*)$ if and only if
	\[ \mbox{either} \quad m\kappa \le \frac 14; \quad \mbox{or} \quad m\kappa >\frac 14 \quad \mbox{and} \quad t^*\le \frac{\pi m}{\sqrt{4m\kappa -1}}+\frac{2 m}{\sqrt{4m\kappa-1}}\sin^{-1}\left(\frac{1}{\sqrt{4m\kappa}}\right). \]
	Furthermore, if $t_0<\max\left\{\frac{1}{2\kappa},\sqrt{\frac{m}{\kappa}}\right\}$, there exists a Banach contraction algorithm that determines $\dot\theta_1(t^*),\cdots,\dot\theta_N(t^*)$ from $\theta_1(t^*),\cdots,\theta_N(t^*),\omega_1^0,\cdots,\omega_N^0$.
\end{theorem}
See Proposition \ref{prop:determinability} for the exact determinability statement and Proposition \ref{prop:banach-contraction} for the Banach contraction algorithm.

The rest of this paper is organized as follows. In Section \ref{sec:galilean}, we review basic properties of inertial Kuramoto oscillators obtained in the previous work \cite{C-D-H-R25}, which will be crucially used in the main proof. In Section \ref{sec:mto0}, we prove our quantitative higher order Tikhonov theorem, namely Theorem \ref{thm:tikhonov}. In Section \ref{sec:4}, we provide a proof of the main theorem, namely Theorem \ref{thm:qualitative_mto0}. 
In Section \ref{sec:bound}, we answer the question of determinability of $\dot\theta_i(t)$ from $\theta_i(t)$ and $\omega_i^0$, namely we prove Theorem \ref{thm:determinability}.
Finally, Section \ref{sec:conclusion} is devoted to a brief summary of our main results and their limitations, and possible directions for future investigations. We leave long and technical proofs in the appendix. In particular, Appendix \ref{app:mto0} is devoted to the proof of Proposition \ref{prop:mto0quant}.

\subsection{Gallery of notations and conventions}\label{sec:notations}
Let $(\theta_1,\ldots,\theta_N)$ be a solution to \eqref{A-1} or \eqref{A-2}. We denote $\omega_i(t)\coloneqq \dot\theta_i(t)$. We use capital Greek letters to denote the collection of the corresponding lower Greek letters:
\begin{equation*}
\Theta\coloneqq  (\theta_1, \ldots, \theta_N), \quad \Omega\coloneqq (\omega_1, \ldots, \omega_N), \quad \mathcal{V}\coloneqq (\nu_1, \ldots, \nu_N).
\end{equation*}
We will use $\omega_i$ and $\dot{\theta}_i$ interchangeably and $\Omega$ and $\dot{\Theta}$ interchangeably throughout this article.
We denote ${\bf 1}_{[N]} = (1, \cdots, 1) \in {\mathbb R}^N$. For vectors $X=(x_i)^N_{i=1}\in \mathbb{R}^N$ and $Y=(y_i)^N_{i=1}\in \mathbb{R}^N$, we denote their exterior product $X\wedge Y\in \bigwedge^2 \mathbb{R}^N$. For the standard basis $e_1,\cdots,e_N$ of $\mathbb{R}^N$, we consider $e_i\wedge e_j$, $i,j\in [N]$, $i<j$ to be the standard basis of $\bigwedge^2\mathbb{R}^N$. Thus, in these coordinates, we may write
\[
X\wedge Y= (x_iy_j-x_jy_i)_{i,j\in [N],i<j}.
\]
As usual, we use $\|\cdot\|_p$ to denote the $\ell_p$-norm in $\bbr^{N}$:
\begin{equation*}
\|\Theta\|_{p} \coloneqq \left( \sum_{i=1}^{N} |\theta_i|^p \right)^{\frac{1}{p}}, ~ 1 \leq p < \infty, \quad  \|\Theta\|_{\infty} \coloneqq \max_{i \in[N]} |\theta_i|.
\end{equation*}
We use $\mathcal{D}$ to denote the diameter: for a vector $X= ( x_i )^N_{i=1}\in \mathbb{R}^N$,
\[
\mathcal{D}(X)\coloneqq \max_{i,j\in[N]}|x_i-x_j|.
\]
For example, for the above configurations $\Theta, \Omega$ and ${\mathcal V}$, we denote
\begin{align*}
\begin{aligned}
{\mathcal D}(\Theta) \coloneqq \max_{i,j\in[N]}|\theta_i -\theta_j|,\quad {\mathcal D}(\Omega) \coloneqq \max_{i,j\in[N]}|\omega_i -\omega_j|,\quad {\mathcal D}(\mathcal{V}) \coloneqq \max_{i,j\in[N]}|\nu_i - \nu_j|.
\end{aligned}
\end{align*}
For $X\in \mathbb{R}^N$, we have
\[
\mathcal{D}(X)=\|X\wedge {\bf 1}_{[N]}\|_\infty,
\]
where $\|\cdot\|_\infty$ denotes the $\ell_\infty$ norm endowed on $\bigwedge^2\mathbb{R}^N$ with respect to its standard basis. We observe the triangle inequality
\begin{equation}\label{eq:triangle-ineq}
|\mathcal{D}(A)-\mathcal{D}(B)|\le \mathcal{D}(A-B),\quad A,B\in \mathbb{R}^N.
\end{equation}
We also have the inequality
\[
\frac 12\mathcal{D}(X)\le \|X\|_\infty,\quad X\in \mathbb{R}^N,
\]
which follows from the identity
\begin{equation}\label{eq:diameter-max-diff}
    \|X\|_\infty-\frac 12\mathcal{D}(X)=\frac 12\left|\max_{i\in [N]} x_i+\min_{i\in [N]} x_i\right|,\quad X=(x_1,\cdots,x_N)\in \mathbb{R}^N.
\end{equation}
We also denote the variances as follows:
\begin{align*}
\begin{aligned}
\operatorname{Var}(\mathcal{V}) \coloneqq \frac 1N\sum_{i=1}^N |\nu_i-\nu_c|^2,\quad \operatorname{Var}(\Omega^0) \coloneqq \frac 1N\sum_{i=1}^N |\omega_i^0-\omega_c|^2.
\end{aligned}
\end{align*}
Here $\nu_c$ and $\omega_c$ represent the averages of components in ${\mathcal V}$ and $\Omega$:
\[ \nu_c := \frac{1}{N} \sum_{i=1}^{N} \nu_i, \quad \omega_c := \frac{1}{N} \sum_{i=1}^{N} \omega_i. \]
It is well known that
\begin{align*}
\operatorname{Var}(X)\le \frac{\mathcal{D}(X)^2}{4},\quad X\in \mathbb{R}^N.
\end{align*}
For $\mathcal{A}\subset [N]$ and $X=(x_i)_{i=1}^N\in \mathbb{R}^N$, we define the restricted vector
\[
X_\mathcal{A}\coloneqq (x_i)_{i\in\mathcal{A}}\in \mathbb{R}^{|\mathcal{A}|}.
\]
The concepts of $\mathcal{D}(X_\mathcal{A})$ and $\|X_\mathcal{A}\|_p$ are defined in the same manner:
\[
\mathcal{D}(\Theta_\mathcal{A})\coloneqq \max_{i,j\in \mathcal{A}}|\theta_i-\theta_j|,\quad {\mathcal D}(\Omega_\mathcal{A}) \coloneqq \max_{i,j\in \mathcal{A}}|\omega_i -\omega_j|,\quad {\mathcal D}(\mathcal{V}_\mathcal{A}) \coloneqq \max_{i,j\in \mathcal{A}}|\nu_i - \nu_j|.
\]
and
\begin{equation*}
\|\Theta_\mathcal{A}\|_{p} \coloneqq \Big( \sum_{i\in\mathcal{A}} |\theta_i|^p \Big)^{\frac{1}{p}}, ~ 1 \leq p < \infty, \quad  \|\Theta_\mathcal{A}\|_{\infty} \coloneqq \max_{i\in\mathcal{A}} |\theta_i|.
\end{equation*}
Throughout the paper, we call $m$, $\kappa$, and $\mathcal{V}$ system parameters, $\Theta^0$ and $\Omega^0$ initial data, and all other external parameters free parameters. We sometimes say `{\it parameters}' to refer to any of these variables.

%
%
%
%
\section{Preliminaries}\label{sec:galilean}
\setcounter{equation}{0}
In this section, we study basic properties of the inertial Kuramoto model.

\subsection{Symmetries of the inertial Kuramoto model}\label{subsec:sym}
In this subsection, we recall several symmetries of the inertial Kuramoto model, as systematically organized in \cite{C-D-H-R25}.


\vspace{0.1cm}
\begin{enumerate}
\item  (Galilean symmetry):~ the quantities $\mathcal{D}(\mathcal{V})$, $\mathcal{D}(\Omega^0)$, $\mathcal{D}(\Theta)$ and $R(\Theta)$ are invariant under the transformation: for $\nu$, $\omega$, $\theta\in\mathbb{R}$,
\begin{align}\label{eq:galilean}
\begin{aligned}
&\tilde{\nu}_i\mapsto\nu_i-\nu,\quad \tilde{\theta}_i^0\mapsto \theta_i^0-\theta,\quad \tilde{\omega}_i^0\mapsto \omega_i^0-\omega,\\
&\tilde{\theta}_i(t)\mapsto \theta_i(t)-\theta-m\omega \left(1-e^{-t/m}\right)-\nu \left(t-m+me^{-t/m}\right),\\
&\tilde{\omega}_i(t)\mapsto \omega_i(t)-\omega e^{-t/m}-\nu \left(1-e^{-t/m}\right).
\end{aligned}
\end{align}
\vspace{0.1cm}
\item  (Dilation symmetry):~the `normalized intrinsic frequencies' $\nu_i/\kappa$, `normalized inertia' $m\kappa$, and `normalized initial velocites' $\omega_i^0/\kappa$ are invariant under the time dilation symmetry: for fixed $\alpha>0$,
\begin{align}\label{eq:dilatation}
\begin{aligned}
&\kappa\mapsto \alpha\kappa,\quad \nu_i\mapsto\alpha\nu_i,\quad m\mapsto m/\alpha,\\
&\theta_i^0\mapsto \theta_i^0, \quad \omega_i^0\mapsto \alpha \omega_i^0,\quad \theta_i(t)\mapsto \theta_i(\alpha t),\quad \omega_i(t)\mapsto \alpha\omega_i(\alpha t),
\end{aligned}
\end{align}
\vspace{0.1cm}
\item (Reflection symmetry): the quantities $\mathcal{D}(\mathcal{V})$, $\mathcal{D}(\Omega^0)$, $\mathcal{D}(\Theta)$ and $R(\Theta)$ are invariant under the following transformation:
\begin{equation}\label{eq:reflection}
\nu_i\mapsto -\nu_i,\quad \theta^0_i\mapsto -\theta^0_i,\quad \omega_i^0\mapsto -\omega_i^0,\quad \theta_i(t)\mapsto -\theta_i(t),\quad \omega_i(t)\mapsto -\omega_i(t).
\end{equation}
\item (Particle exchange symmetry): the quantities $\mathcal{D}(\mathcal{V})$, $\mathcal{D}(\Omega^0)$, $\mathcal{D}(\Theta)$ and $R(\Theta)$ are invariant under the following transformation:
\begin{equation}\label{eq:particle-exchange}
  \nu_i\mapsto \nu_{\pi^{-1}(i)},\quad \theta_i^0\mapsto \theta_{\pi^{-1}(i)}^0,\quad \omega_i^0\mapsto \omega_{\pi^{-1}(i)}^0
\end{equation}
for fixed $\pi\in S_N$ (the symmetric group on $N$ elements).
\end{enumerate}
The Galilean symmetry is accompanied by the following conservation law, where we use the subscript $c$ to denote the mean over all particles:
\begin{equation} \label{B-2}
	\theta_c^0 \coloneqq \frac{1}{N} \sum_{i=1}^{N}  \theta_i^0,\quad
	\theta_c \coloneqq \frac{1}{N}\sum_{i=1}^N\theta_i,\quad
	\omega_c^0 \coloneqq \frac{1}{N} \sum_{i=1}^{N} \omega_i^0,\quad
	\omega_c \coloneqq \frac{1}{N}\sum_{i=1}^N\omega_i,\quad
	\nu_c = \frac{1}{N}\sum_{i=1}^N \nu_i.
\end{equation}
\begin{lemma}\label{L2.1}
	Let $(\Theta, \Omega)$ be a global solution to \eqref{A-1}. The mean phase and frequency satisfy
	\begin{align}
		\begin{aligned} \label{eq:galconserve}
			\theta_c(t) &= m\omega_c^0(1-e^{-t/m}) + \nu_c \left(t-m+me^{-t/m}\right) + \theta_c^0, \\
			\omega_c(t) &= \omega_c^0 e^{-t/m} + \nu_c\left(1-e^{-t/m}\right), \quad t \geq 0.
		\end{aligned}
	\end{align}
\end{lemma}
\begin{proof} 
	We sum both sides of \eqref{A-1} over $i$ and apply the definitions of the means given in \eqref{B-2}, yielding
	\[
	{\dot \theta}_c = \omega_c, \quad  {\dot \omega}_c = \frac{\nu_c}{m} - \frac{\omega_c}{m}, \quad t > 0.
	\]
	Integrating these ODEs, we obtain the relations \eqref{eq:galconserve}.
\end{proof}

As pointed out in \cite{C-D-H-R25}, the quantities $m\kappa$, $\mathcal{D}(\mathcal{V})/\kappa$, and $\mathcal{D}(\Omega^0)/\kappa$, that we are requiring to be small in our synchronization framework, are invariant under all four symmetries listed above. 
It is therefore ``natural'' for the assumptions and conclusions of theorems concerning the Kuramoto model \eqref{A-1} to also be invariant under these symmetries. In some cases, if a statement about the Kuramoto model \eqref{A-1} lacks symmetry, one can often find a symmetric counterpart. For instance, if the solution to \eqref{A-1} with initial data $(\Theta^0,\dot{\Theta}^0)$ and parameters $\kappa$, $m$, and $\mathcal{V}$ exhibits asymptotic phase-locking, then so does the solution to \eqref{A-1} with initial data $(\Theta^0,\dot{\tilde{\Theta}}^0)$ and parameters $\tilde{\kappa}$, $\tilde{m}$, and $\tilde{\mathcal{V}}$, provided that $\dot{\tilde{\Theta}}^0/\tilde{\kappa}=\dot{\Theta}^0/\kappa$, $\tilde{m}\tilde{\kappa}=m\kappa$, and $\tilde{\mathcal{V}}/\tilde{\kappa}=\mathcal{V}/\kappa$. 

Another example lies in the notions of ``phase-locked states'' and ``asymptotic phase-locking'' introduced in Definition \ref{D1.1} below. These notions are equivalent to ``equilibria'' and ``convergence to equilibria'' after applying the Galilean transformation \eqref{eq:galilean} with $\nu=\nu_c$ and $\omega=\omega_c$, and they remain invariant under the aforementioned symmetries (refer to Proposition 2.1 in \cite{C-D-H-R25}).
\begin{definition}\label{D1.1}
Let $(\Theta(t), {\dot \Theta}(t))$ be a global solution to the Cauchy problem \eqref{A-1} or \eqref{A-2}.
\begin{enumerate}
\item
$(\Theta(t), {\dot \Theta}(t))$ is a \emph{phase-locked state} if for all $i,j\in [N]$, $\theta_i(t)-\theta_j(t)$ is constant with respect to $t$.
\vspace{0.1cm}
\item
$(\Theta(t), {\dot \Theta}(t))$ exhibits \emph{asymptotic phase-locking} if
\[  \exists~  \lim_{t \to \infty} (\theta_i(t) - \theta_j(t) ), \quad \forall~ i, j \in [N]. \]
\item
$(\Theta(t), {\dot \Theta}(t))$ exhibits \emph{complete (frequency) synchronization} if
\[  \lim_{t \to \infty} \max_{i, j \in [N]} |{\dot \theta}_i(t) - {\dot \theta}_j(t) | = 0. \]
\item
$(\Theta(t), {\dot \Theta}(t))$ exhibits \emph{complete phase synchronization} if
\[  \lim_{t \to \infty} \left(\theta_i(t) - \theta_j(t)\right)  \in 2\pi \mathbb{Z},\quad \forall~ i,j\in [N]. \]
\end{enumerate}
\end{definition}
\begin{remark}\label{rem:pls}
Below, we provide several comments on Definition \ref{D1.1}.
\begin{enumerate}
    \item Asymptotic phase-locking implies complete frequency synchronization; one can easily see this from the Duhamel principle, later presented in \eqref{B-4}. In fact, asymptotic phase-locking is equivalent to the condition that there exist $\theta_1^\infty,\cdots,\theta_N^\infty\in \mathbb{R}$ such that
    \[
    \lim_{t\to\infty } \left(\theta_i(t) - \frac{1}{N}\sum_{i=1}^N\nu_i t\right) = \theta_i^\infty,
    \quad
    \lim_{t\to\infty} \dot{\theta}_i(t) = \frac{1}{N}\sum_{i=1}^N\nu_i ,\qquad i=1,\cdots,N.
    \]
    \vspace{0.1cm}
    \item Complete phase synchronization can happen only if $\nu_i=\nu_j$ for all $i,j\in [N]$. Again, this is due to the Duhamel principle \eqref{B-4}.
\vspace{0.1cm}
 \item Phase-locked states for \eqref{A-1} coincide with phase-locked states for \eqref{A-2}, up to the rotation of the circle, when $\omega^0_i\neq\omega_j^0$ for some $i,j$.
\end{enumerate}
    
\end{remark}
Fixing all other parameters and increasing $\kappa$ to sufficiently large values,  phase-locked states emerge beyond a certain critical threshold \cite{D-B14, V-M}.

\subsection{Duhamel’s principle} \label{sec:2.1}
In this subsection, we discuss Duhamel's principle for \eqref{A-1}. For this, we first rewrite the $N$-dimensional second-order Cauchy problem \eqref{A-1} as the $2N$-dimensional Cauchy problem
\begin{equation}\label{B-1}
\begin{cases}
\displaystyle \dot\theta_i = \omega_i, \quad t >0, \\
\displaystyle \dot\omega_i = \frac{1}{m} \Big( \nu_i - \omega_i + \frac{\kappa}{N}\sum_{j=1}^N \sin(\theta_j -\theta_i) \Big),\\
\displaystyle  (\theta_i, \omega_i) \Big|_{t = 0+} = (\theta^{0}_i, \omega^0_i),\quad i\in[N].
\end{cases}
\end{equation}
Again, system \eqref{B-1} admits a unique global solution and describes the same system as \eqref{A-1}.\footnote{To be precise, the application of the Cauchy-Lipschitz theory to \eqref{A-1} is through the system \eqref{B-1}, so the uniqueness and global existence of solutions is first established for \eqref{B-1} and then transferred to \eqref{A-1}.} One advantage of this point of view is that it allows us to apply the Duhamel principle to $\omega_i$: 
\begin{align}\label{B-4}
\begin{aligned}
\omega_i(t) &
&=  \omega^0_i e^{-t/m} + \nu_i (1 -e^{-t/m}) + \frac{\kappa}{Nm}\sum_{j=1}^N \int_0^t e^{-(t-s)/m}\sin(\theta_j(s) -\theta_i(s)) ds.
\end{aligned}
\end{align}
Since this is a weighted time-delayed version of the first-order model \eqref{A-2}, the effect of inertia can be interpreted as introducing a weighted time-delay in interactions. One immediate consequence of this formulation is the derivation of crude bounds on $\dot\Theta(t)$.
\begin{lemma}[Finite propagation speed {\cite[Lemma 2.2]{C-L}}, {\cite[Lemmas 1{~and~}4]{H-J-K}}, {\cite[Lemma 3.1]{C-D-H-R25}}]\label{L2.2}
Let $(\Theta, \Omega)$ be a global solution to \eqref{B-1}. Then, for $i, j \in [N]$ and $t\ge 0$,  one has
\begin{align*}
\begin{aligned}
& (1)~e^{-t/m}\omega_i^0 +(1-e^{-t/m})(\nu_i - \kappa) \le \omega_i(t) \le e^{-t/m}\omega_i^0 +(1-e^{-t/m})(\nu_i + \kappa). \\
& (2)~ |\omega_i(t)-\omega_j(t)| \le e^{-t/m}|\omega_i^0-\omega_j^0| +(1-e^{-t/m})(|\nu_i-\nu_j| + 2\kappa). \\
& (3)~\mathcal{D}(\Omega(t))\le e^{-t/m}{\mathcal D}(\Omega^0)  + (1-e^{-t/m})\left( {\mathcal D}({\mathcal V}) + 2 \kappa\right). \\
\end{aligned}
\end{align*}    
\end{lemma}

\noindent In Lemma \ref{L:approxaut}, when the authors derived bounds on the first-order derivatives $\omega_i(t)$ in terms of $\dot{\Theta}(0)$, $\Theta(t)$, and $\mathcal{V}$, the Duhamel principle was applied using Lemma \ref{L2.2} once again.

%
%
%
%
\section{Quantitative higher-order Tikhonov perturbation}\label{sec:mto0}
\setcounter{equation}{0}
In this section, we study quantitative estimates for the zero inertia limit of the inertial Kuramoto model. 

In what follows, for initial data $\left(\Theta^0,\Omega^0\right)$, intrinsic frequencies  $\mathcal{V}$, a coupling strength $\kappa$ and  $m>0$, we denote by $\Theta(m,t)$, $t\ge 0$, the solution to the Cauchy problem \eqref{A-1}, and denote by $\Theta(0,t)$, $t\ge 0$, the solution to the Cauchy problem \eqref{A-2}.

We may first ask whether $\Theta(m,t)\to \Theta(0,t)$ as $m\to 0$. This is true on finite time intervals away from zero: it is an immediate consequence of the Tikhonov theorem \cite{T52,V63} that this convergence holds in the $C^0[0,\infty)$ and $C^1(0,\infty)$ Fr\'echet topologies (see Proposition \ref{prop:mto0}). By tedious computation, we will obtain a quantitative and higher-order version of Tikhonov's theorem, namely a $C[0,\infty)\cap C^\infty(0,\infty)$ convergence statement with explicit bounds (see Proposition \ref{prop:mto0quant}). By a comparison argument, we will obtain a qualitative version of the framework \eqref{eq:R^0} without any explicit bounds on the parameters (see Theorem \ref{thm:qualitative_mto0}).
\begin{proposition}[Tikhonov's theorem]\label{prop:mto0}
Let $(\Theta(m,t), \Omega(m,t))$ and $(\Theta(0,t), \Omega(0,t))$ be solutions to the Cauchy problems \eqref{A-1} and \eqref{A-2}, respectively with fixed initial conditions independent of $m$:
\[
(\Theta(m,0),\Omega(m,0))=(\Theta^0,\Omega^0),\quad \Theta(0,0)=\Theta^0.
\]
Then, for all $i\in [N]$, $\theta_i(m,t)\to\theta_i(0,t)$ as $m\to 0$ in the Fr\'echet topologies $C^0[0,\infty)$ and $C^1(0,\infty)$:
\begin{align*}
\begin{aligned}
& (i)~\lim_{m\to 0}\sup_{t\in [0,T]}|\theta_i(m,t)-\theta_i(0,t)|=0,\quad\forall~ T>0, \\
& (ii)~\lim_{m\to 0}\sup_{t\in [T_1,T_2]}\left|\omega_i(m,t)-\omega_i(0,t)\right|=0,\quad\forall~ T_2>T_1>0.
\end{aligned}
\end{align*}
\end{proposition}
\begin{proof}
We set 
\begin{align*}
F_{i}(\Theta(m,t))\coloneqq \nu_i +\frac{\kappa}{N}\sum_{j=1}^N \sin(\theta_j(m,t) - \theta_i(m,t)),\quad F=(F_1,\cdots,F_N).
\end{align*}
Then, the Cauchy problem \eqref{A-1} states that the vector $(\Theta(m,t),\Omega(m,t))$ with $2N$ entries is the unique solution to the ODE:
\[
\begin{cases}
  \displaystyle  \frac {d}{dt}\Theta(m,t)=\Omega(m,t),&\\
  \displaystyle  m\frac {d}{dt}\Omega(m,t)=-\Omega(m,t)+F(\Theta(m,t)),
\end{cases}
\]
with fixed initial conditions:
\[
(\Theta(m,0),\Omega(m,0))=(\Theta^0,\Omega^0).
\]
Likewise, the vector $(\Theta(0,t),\Omega(0,t))$ with $2N$ entries is the unique solution to the ODE
\[
\begin{cases}
\displaystyle    \frac {d}{dt}\Theta(0,t)=\Omega(0,t),&\\
\displaystyle     0=-\Omega(0,t)+F(\Theta(0,t)),
\end{cases}
\]
with initial conditions:
\[
(\Theta(0,0),\Omega(0,0))=(\Theta^0,F(\Theta^0)).
\]
Tikhonov's theorem \cite[Theorem 1.1]{V63} tells us that in this formulation, $\Theta(m,t)\to \Theta(0,t)$ as $m\to 0$ on compact subintervals $[0,T]\subset [0,\infty)$, and $\Omega(m,t)\to \Omega(0,t)$ as $m\to 0$ on compact subintervals $[T',T]\subset (0,\infty)$.
\end{proof}
\begin{remark}\,
\begin{enumerate}
\item Note that Propsition \ref{prop:mto0} does not give quantitative rates of convergence. In order to prove asymptotic phase-locking results, one must have quantitative bounds in order to control the behavior for small times.
\item
Proposition \ref{prop:mto0} cannot give uniform-in-time bounds, i.e., convergence in the uniform topology $L^\infty([0,\infty))$, because it is not true: for example, in the case $N=2$, compare the first-order equation for $\theta = \theta_1 - \theta_2$:
\[
\begin{cases}
\displaystyle \dot\theta=1-0.5\sin\theta, \quad t > 0, \\
\displaystyle \theta^0=0
\end{cases}
\quad \mbox{and} \qquad 
\begin{cases}
\displaystyle 0.5\ddot \theta+\dot\theta=1-0.5\sin\theta, \quad t > 0, \\
\displaystyle \theta^0=\dot\theta^0=0. 
\end{cases}
\]
The difference of the two solutions diverges linearly in time as $t\to\infty$. This is not surprising, since two different dynamical systems with the same initial position data usually do not agree in the long term.

\item
 Proposition \ref{prop:mto0} cannot give uniform convergence on neighborhoods of $t=0$ for the first derivative, i.e., it does not give convergence in the $C^1([0,\infty))$ topology, because it is also not true: indeed, \eqref{A-1} prescribes the arbitrary value $\dot\theta_i(0)=\omega_i$, but \eqref{A-2} mandates that $\dot\theta_i(0)=\nu_i+\sum_j\sin(\theta_j^0-\theta_i^0)$. This suggests that any approach to synchronization that approximates the second-order model \eqref{A-1} by the first-order model \eqref{A-2} should be done in a time regime away from $t=0$.
 \end{enumerate}
 \end{remark}
 
In the following proposition, we present a quantitative version of Proposition~\ref{prop:mto0}. Specifically, we show that
\[ \omega_i(m,t)-\omega_i(0,t)= {\mathcal O}(e^{-t/m})+ {\mathcal O}(m). \]
This estimate provides the rationale behind the ``initial time layer'' of the form $[0,\eta m]$ introduced in \cite{C-D-H-R25}. In that work, analyses such as partial phase-locking and the quasi-monotonicity estimate for the order parameter were carried out after this initial time layer.
A key limitation of the classical Tikhonov theorem is its lack of quantitative bounds, which are necessary for establishing results such as Theorem~\ref{thm:qualitative_mto0} below. By working directly with the ODE systems \eqref{A-1} and \eqref{A-2}, we derive explicit estimates, as detailed in Proposition~\ref{prop:mto0quant}. Furthermore, we demonstrate convergence of higher derivatives, namely convergence in the $C^\infty(0,\infty)$ topology.

\begin{proposition}\label{prop:mto0quant}
Fix initial data $(\Theta^0,\Omega^0)$, intrinsic velocities $(\nu_i)_{i=1}^N$, and a coupling strength $\kappa$. For each $m>0$, denote by $\Theta(m,t)$, $t\ge 0$, the solution to the Cauchy problem \eqref{A-1}, and denote by $\Theta(0,t)$, $t\ge 0$, the solution to the Cauchy problem \eqref{A-2}. Then, the following statements hold. \newline
\begin{enumerate}
    \item 
(Convergence of $\Theta(m, \cdot)$):~$\Theta(m,t)$ converges to $\Theta(0,t)$ as $m\to 0$ in the Fr\'echet topology $C^0([0,\infty))$. Quantitatively, the convergence $\Theta(m,t)\to\Theta(0,t)$ is linear in $m$, with coefficients depending on $\kappa$, $\|\Omega^0-\mathcal{V}\|_\infty$, and $e^{2\kappa t}$:
\begin{align}\label{eq:C0-approx-abs-vanilla}
\begin{aligned}
& \|(\Theta(m,t)-\Theta(0,t))\|_\infty \\
&\hspace{1cm} < m\left|\frac{\max_i(\omega_i^0-\nu_i)+\min_i(\omega_i^0-\nu_i)}2\right|  +\frac 12 m \Big(\mathcal{D} (\Omega^0-\mathcal{V})+2\kappa \Big )e^{2\kappa t},\quad t\ge 0,
\end{aligned}
\end{align}
and the convergence $\Theta(m,t)\wedge {\bf 1}_{[N]}\to\Theta(0,t)\wedge {\bf 1}_{[N]}$ is linear in $m$, with coefficients depending on $\kappa$, $\mathcal{D}(\Omega^0-\mathcal{V})$, and $e^{2\kappa t}$:
\begin{align}\label{eq:C0-approx-rel-vanilla}
\begin{aligned}
\mathcal{D}(\Theta(m,t)-\Theta(0,t))
&< m \Big (\mathcal{D}(\Omega^0-\mathcal{V})+2\kappa \Big )e^{2\kappa t},\quad t\ge 0.
\end{aligned}
\end{align}
\item 
(Convergence of ${\dot \Theta}(m, \cdot)$):
$\dot\Theta(m,t)$ converges to  $\dot\Theta(0,t)$ as $m\to 0$ in the Fr\'echet topology $C^0(0,\infty)$. Quantitatively, the convergence $\dot\Theta(m,t)\to \dot\Theta(0,t)$ is bounded by a linear combination of $m$ and $e^{-t/m}$, with coefficients depending on $\kappa$, $\|\Omega^0-\mathcal{V}\|_\infty$, $\mathcal{D}(\mathcal{V})$ and $e^{2\kappa t}$:
\begin{align}\label{eq:C1-approx-abs-vanilla}
\begin{aligned}
&\|(\dot\Theta(m,t)-\dot\Theta(0,t))\|_\infty \\
&\hspace{1cm} < \Big(\|\Omega^0-\mathcal{V}\|_\infty+\kappa \Big)e^{-t/m}+ m\kappa(\mathcal{D}(\mathcal{V})+2\kappa) +m\kappa \Big (\mathcal{D}(\Omega^0-\mathcal{V})+2\kappa \Big )e^{2\kappa t},\quad t\ge 0,
\end{aligned}
\end{align}
and the convergence $\dot\Theta(m,t)\wedge {\bf 1}_{[N]}\to \dot\Theta(0,t)\wedge {\bf 1}_{[N]}$ is bounded by a linear combination of $m$ and $e^{-t/m}$, with coefficients depending on $\kappa$, $\mathcal{D}(\Omega^0-\mathcal{V})$, $\mathcal{D}(\mathcal{V})$ and $e^{2\kappa t}$:
\begin{align}\label{eq:C1-approx-rel-vanilla}
\begin{aligned}
& \mathcal{D}(\dot\Theta(m,t)-\dot\Theta(0,t)) \\
& \hspace{0.2cm} < \Big (\mathcal{D}(\Omega^0-\mathcal{V} \Big )+2\kappa)e^{-t/m}+ 2m\kappa(\mathcal{D}(\mathcal{V})+2\kappa)  +2m\kappa \Big (\mathcal{D}(\Omega^0-\mathcal{V})+2\kappa \Big )e^{2\kappa t},\quad t\ge 0.
\end{aligned}
\end{align}
\item 
(Convergence of $\Theta^{(n)}(m, \cdot)$ with $n \geq 2$): $\Theta(m,t)$ converges to $\Theta(0,t)$ as $m\to 0$ in the Fr\'echet topology $C^\infty(0,\infty)$. Quantitatively, for $n\ge 2$, we have
\begin{align*}
    &\|\Theta^{(n)}(m,t)-\Theta^{(n)}(0,t)\|_\infty\\
    & \hspace{0.5cm} \le (n-1)!\left(2\kappa+\|\Omega^0\|+\|\mathcal{V}\|+\frac{9}{8m}\right)^{n}\left(1+\frac tm\right)^{n}e^{-t/m}\\
    &  \hspace{0.7cm} + \frac 98m\kappa \cdot  n!e^{2\kappa t}\left(2\kappa+\mathcal{D}(\Omega^0)+\mathcal{D}(\mathcal{V})+\frac{9}{8m}\right)^{n}\left(1+\frac tm\right)^{n}e^{-t/m}\\
& \hspace{0.7cm} + \frac 34m\kappa \cdot  (n+1)!e^{2\kappa t} \Big(2\kappa+\mathcal{D}(\Omega^0)+\mathcal{D}(\mathcal{V}) \Big)^{n}(1-e^{-t/m}),
\end{align*}
and
\begin{align*}
    &\mathcal{D}(\Theta^{(n)}(m,t)-\Theta^{(n)}(0,t))\\
    &  \hspace{0.5cm} \le 2(n-1)!\left(2\kappa+\mathcal{D}(\Omega^0)+\mathcal{D}(\mathcal{V})+\frac{9}{8m}\right)^n\left(1+\frac tm\right)^n e^{-t/m}\\
    &  \hspace{0.7cm} +\frac 32(n+1)!m\kappa e^{2\kappa t}\left(2\kappa+\mathcal{D}(\Omega^0)+\mathcal{D}(\mathcal{V})+\frac{9}{8m}\right)^n\left(1+\frac tm\right)^n e^{-t/m}\\
    & \hspace{0.7cm} +\frac 32(n+1)!m\kappa e^{2\kappa t} \Big (2\kappa+\mathcal{D}(\Omega^0)+\mathcal{D}(\mathcal{V}) \Big)^n(1-e^{-t/m}).
\end{align*}
\end{enumerate}
\end{proposition}
\begin{proof}
    Since the proofs are lengthy and technical, we postpone them to Appendix~\ref{app:mto0}. In the proof, we use an estimate for the integral form of Gr\"onwall-type lemma in Appendix~\ref{app:sturm-picone}.
\end{proof}

\begin{question}[General quantitative Tikhonov statement]
Let $N\ge 1$ be a fixed integer, and let $F:\mathbb{R}^N\to\mathbb{R}^N$ have $C^\infty$ regularity. Let the vector $(\Theta(m,t),\Omega(m,t))$ with $2N$ entries be the unique solution to the ODE
\[
\begin{cases}
	\displaystyle  \frac {d}{dt}\Theta(m,t)=\Omega(m,t),&\\
	\displaystyle  m\frac {d}{dt}\Omega(m,t)=-\Omega(m,t)+F(\Theta(m,t)),
\end{cases}
\]
with fixed initial conditions:
\[
(\Theta(m,0),\Omega(m,0))=(\Theta^0,\Omega^0).
\]
Likewise, let the vector $(\Theta(0,t),\Omega(0,t))$ with $2N$ entries be the unique solution to the ODE
\[
\begin{cases}
	\displaystyle    \frac {d}{dt}\Theta(0,t)=\Omega(0,t),&\\
	\displaystyle     0=-\Omega(0,t)+F(\Theta(0,t)),
\end{cases}
\]
with initial conditions:
\[
(\Theta(0,0),\Omega(0,0))=(\Theta^0,F(\Theta^0)).
\]
Tikhonov's theorem \cite[Theorem 1.1]{V63} tells us that in this formulation, $\Theta(m,\cdot)\to \Theta(0,\cdot)$ as $m\to 0$ in the $C[0,\infty)$ topology, and $\Omega(m,\cdot)\to \Omega(0,\cdot)$ as $m\to 0$ in the $C(0,\infty)$ topology. Do we have the following refined statements?

	\begin{enumerate}
	\item The convergence of $\Theta_i(m,t)$ to $\Theta_i(0,t)$ as $m\to 0$ in the Fr\'echet topology $C^0([0,\infty))$ is quantitatively given as, for some constant $c>0$,
	\[
	\|\Theta(m,t)-\Theta(0,t)\|_\infty\le cm \left(\|\Omega^0\|_\infty+\|F\|_{\infty}\right)e^{c\|F\|_{C^{1}} t},\quad t\ge 0,
	\]
	where $\|\cdot\|_{C^n}$ denotes the $C^n$ norm.
	\item 
	We have the convergence of $\Theta(m,t)$ to $\Theta(0,t)$ as $m\to 0$ in the Fr\'echet topology $C^\infty(0,\infty)$, and quantitatively, for $n\ge 1$, we have for $t\ge 0$,
	\begin{align*}
		&|\Theta^{(n)}(m,t)-\Theta^{(n)}(0,t)|\\
		&\le C\left(\left(1+\frac tm\right)^ne^{-t/m}(1+mC e^{2C t})+mC e^{2C t}\right),
	\end{align*}
	where $C$ depends on $\|F\|_{C^{n+1}}$.

\end{enumerate}
\end{question}

%
%
%
%
\section{Proof of Theorem \ref{thm:qualitative_mto0}} \label{sec:4}
\setcounter{equation}{0}

In this section, we provide the proof of Theorem \ref{thm:qualitative_mto0}. Before we do so, we first recall a result on partial phase-locking.

\subsection{A partial phase-locking result}

In this subsection, we introduce the notion of partial phase-locking, following \cite{C-D-H-R25}, where the definition is motivated from observation of partially ordered behavior among oscillators \cite{B-W}.
\begin{definition}[{\cite[Definition 2.2]{C-D-H-R25}}, {\cite[Definition 4.1]{H-K-R}}, {\cite[Definition 2.2]{H-R}}]\label{def:partial}
	Let $(\Theta(t), {\dot \Theta}(t))$ be a global solution to the Cauchy problem \eqref{A-1}.
	\begin{enumerate}
		\item
		Given $\mathcal{A}\subset [N]$, we say the solution $(\Theta(t), {\dot \Theta}(t))$ exhibits \emph{$\mathcal{A}$-partial phase-locking} if
		\[ 
		\sup_{t\ge 0}\sup_{i,j\in \mathcal{A}}|\theta_i(t)-\theta_j(t)|<\infty.
		\]
		\item
		Given $\lambda\in (0,1]$, we say the solution $(\Theta(t), {\dot \Theta}(t))$ exhibits \emph{$\lambda$-partial phase-locking} if there exists $\mathcal{A}\subset[N]$ with $|\mathcal{A}|\ge \lambda N$ such that the solution $(\Theta(t), {\dot \Theta}(t))$ exhibits $\mathcal{A}$-partial phase-locking.
	\end{enumerate}
\end{definition}

We recall a key ingredient of \cite{C-D-H-R25} for providing a sufficient framework for achieving asymptotic phase-locking, namely the partial phase-locking of a majority cluster. Specifically, it was shown in \cite[Theorem 4.1]{C-D-H-R25} that once a majority cluster forms in a sufficiently small neighborhood, the diameter of that cluster stays uniformly bounded over time. 

The precise statement is given below in Theorem \ref{L4.4}. We remark that this is an abridged form of \cite[Theorem 4.1]{C-D-H-R25} to highlight only the stability property, since this statement is all we need for the proof of Theorem \ref{thm:qualitative_mto0}; the original text of \cite[Theorem 4.1]{C-D-H-R25} also asserts the asymptotic arrangement of oscillators according to the natural frequencies $\nu_i$, and the uniqueness of the maximal majority cluster.

\begin{theorem}[{\cite[Theorem 4.1]{C-D-H-R25}}]\label{L4.4}
	Suppose that the free real parameters $\lambda, \ell>0$ and index set ${\mathcal A} \subset [N]$ satisfy
	\begin{equation*}
		\frac{1}{2} < \lambda \leq 1, \quad \ell \in\left(0,2\cos^{-1} \Big( \frac{1}{\lambda} - 1 \Big) \right), \quad |\mathcal{A}|\ge \lambda N,
	\end{equation*}
	and that the system parameters and the free real parameter $\eta>0$ satisfy the following for the index set $\mathcal{A}\subset [N]$:
	\begin{align}\label{eq:xi-partial}
		\begin{aligned}
			\xi(m,\kappa,\mathcal{V}_\mathcal{A},\Omega^0_\mathcal{A},\eta)&\coloneqq m\mathcal{D}(\mathcal{V}_\mathcal{A})+2m\kappa+\frac{\mathcal{D}(\mathcal{V}_\mathcal{A})}{2\kappa} \\
			&+\mathcal{D}(\Omega^0_\mathcal{A})m\max\{1,\eta\}e^{-\max\{1,\eta\}}+\frac{\mathcal{D}(\Omega^0_\mathcal{A})}{2\kappa}\frac{e^{-\eta}}{1-e^{-\eta}}\\
			&<\frac\lambda 2\sin \ell-(1-\lambda)\sin\frac{\ell}{2}.
		\end{aligned}
	\end{align}
	Let $(\Theta,\Omega)$ be a global solution to \eqref{B-1}. Assume there exists a time $t_1 \geq \eta m$ such that the subensemble $\Theta_\mathcal{A}=(\theta_\ell)_{\ell\in \mathcal{A}}$ satisfies
	\begin{align*}
		\mathcal{D}(\Theta_\mathcal{A}(t_1))\le \ell.
	\end{align*}
	Then, the following assertions hold.
	\begin{enumerate}
		\item (Stability of the majority cluster): One has
		\[
		\sup_{t\ge t_1}\mathcal{D}(\Theta_\mathcal{A}(t))\le \ell, \quad 
		\limsup_{t\rightarrow\infty}\mathcal{D}(\Theta_\mathcal{A}(t)) <\frac{3\pi}{4(2\lambda-1)}\left(2m\mathcal{D}(\mathcal{V}_\mathcal{A})+4m\kappa+\frac{\mathcal{D}(\mathcal{V}_\mathcal{A})}{\kappa}\right).
		\]
		\item (The majority cluster $\mathcal{A}$ confines a larger cluster $\mathcal{B}$)
		Assume that there is an index set $\mathcal{B}\subset [N]$ with $\mathcal{B}\supset \mathcal{A}$ satisfying the following variant of \eqref{eq:xi-partial}:
		\begin{equation*}
			\xi(m,\kappa,\mathcal{V}_\mathcal{B},\Omega^0_\mathcal{B},\infty)=m\mathcal{D}(\mathcal{V}_\mathcal{B})+2m\kappa+\frac{\mathcal{D}(\mathcal{V}_\mathcal{B})}{2\kappa}<\frac\lambda 2\sin \ell-(1-\lambda)\sin\frac{\ell}{2}.
		\end{equation*}
		Then, the ensemble $\Theta_\mathcal{B}$ is partially phase-locked:
		\[
		\sup_{t\ge 0}\mathcal{D}(\Theta_\mathcal{B}(t))<\infty.
		\]
		In particular, if $\mathcal{B}=[N]$, then asymptotic phase-locking occurs.
		
	\end{enumerate}
\end{theorem}

\subsection{Proof of Theorem \ref{thm:qualitative_mto0}: approximation to the identical oscillator case of \eqref{A-2}}

Now we begin the proof of Theorem \ref{thm:qualitative_mto0} using the partial phase-locking statement, namely Theorem \ref{L4.4}, and the quantitative Tikhonov result, namely Proposition \ref{prop:mto0quant}. The idea is to use Proposition \ref{prop:mto0quant} to approximate the inertial Kuramoto model \eqref{A-1} to the first-order Kuramoto model \eqref{A-2}, and then see \eqref{A-2} as a perturbation of the Kuramoto model \eqref{A-2} with identical intrinsic frequencies $\nu_i=\nu_j$ when $\kappa$ is large.

If $N=2$, we may assume $\varepsilon<\frac 12$, and if $N\ge 3$, we may assume 
\[ \varepsilon<\min \Big \{\frac {R^0}{2},\frac 12-\frac 1N \Big \}, \]
by possibly replacing $\varepsilon$ with $\min\{\varepsilon,\frac {R^0}{4},\frac 14-\frac 1{2N}\}$. We denote by $\tau=\kappa t$ the normalized time.\newline

Given the initial data $\Theta^0$ with $R^0>0$, we denote by $\Phi^{\mathrm{ID,1st}}=\Phi^{\mathrm{ID,1st}}(\tau)$ the solution to the first-order Kuramoto model with identical intrinsic frequencies and unit coupling strength:
\begin{align*}
	\begin{cases}
		\displaystyle \dot\phi_i^{\mathrm{ID,1st}}(\tau) = \frac{1}{N}\sum_{j=1}^N \sin(\phi_j^{\mathrm{ID,1st}}(\tau) - \phi_i^{\mathrm{ID,1st}}(\tau)),\quad \tau > 0,\\
		\displaystyle \phi_i^{\mathrm{ID,1st}}(0) = \theta_i^0, \quad ~i\in [N].
	\end{cases}
\end{align*}
By \cite{H-K-R,H-R}, $\Phi^{\mathrm{ID,1st}}$ tends to either the completely synchronized state, or a bipolar state with order parameter $\ge R^0$, i.e., there exists $\lambda\in [\frac {1+R^0}2,1]$, $\phi^\infty\in\mathbb{R}$, and an index set $\mathcal{A}\subset [N]$ such that 
\[
|\mathcal{A}|\ge \lambda N \quad \mbox{and} \quad 
\begin{cases}
	\phi_i^{\mathrm{ID,1st}}\to \phi^\infty \mathrm{~as~}t\to\infty &\mathrm{if~}i\in \mathcal{A},\\
	\phi_i^{\mathrm{ID,1st}}\to \phi^\infty+\pi \mathrm{~as~}t\to\infty   &\mathrm{if~}i\in [N]\setminus\mathcal{A},
\end{cases}
\]
up to $2\pi$-translations.\newline

\noindent If $N=2$, then the only possibility is 
\[ \mathcal{A}=[N] \quad \mbox{and} \quad \lambda=1. \]
If $N=3$, we must have 
\[ |\mathcal{A}|\ge 2 \quad \mbox{and} \quad  \lambda\ge \frac 23. \]
If $N\ge 4$ and \eqref{eq:init-nonidentical} holds, then since the cross-ratios are constants of motion for $\Phi^{\mathrm{
		ID,1st}}$, we must have 
\[ |\mathcal{A}|\ge N-1 \quad \mbox{and} \quad \lambda \ge \frac{N-1}{N}. \]
Otherwise, if $N\ge 4$ but \eqref{eq:init-nonidentical} doesn't hold, all we can say is that
\[ |\mathcal{A}|\ge \frac{1+R^0}{2}N \quad \mbox{and} \quad \lambda \ge \frac{1+R^0}{2}. \]

Given this $\lambda$, we choose $\ell \in\left(0,2\cos^{-1} \Big( \frac{1}{\lambda} - 1 \Big) \right)$ small enough so that
\[ \cos \frac \ell 2>1-\frac \varepsilon \lambda. \]
Then, we have
\begin{align}\label{qualitative_mto0-A2}
	\lambda \cos \frac\ell 2-(1-\lambda)>\lambda(1-\frac\varepsilon\lambda)-(1-\lambda)\ge
	\begin{cases}
		1-\varepsilon&\mathrm{if~}N=2,\\
		1-\frac 2N-\varepsilon&\mathrm{if~}N=3\mathrm{~or~}\eqref{eq:init-nonidentical}\mathrm{~holds},\\
		R^0-\varepsilon&\mathrm{otherwise}.
	\end{cases}
\end{align}
Choose a time $\tau_0\ge 1$ such that
\[
|\phi_i^{ \mathrm{ID,1st}}(\tau_0)- \phi^\infty|<\frac \ell 6,\quad  \forall~ i\in \mathcal{A}.
\]
This implies
\[
\mathcal{D}(\Phi^{ \mathrm{ID,1st}}_\mathcal{A}(\tau_0))<\frac \ell 3.
\]
Note that $\ell=\ell(\Theta^0,\varepsilon)$ and $\tau_0=\tau_0(\Theta^0,\varepsilon)$ can be chosen to be the functions of $\Theta^0$ and $\varepsilon$ only. \newline

Now, we choose $a,b,c>0$ small enough so that
\begin{equation}\label{eq:abc-1}
	a\le \frac \ell 3e^{-2\tau_0}, \quad  c\left(a+b+2\right)\le\frac{\ell}{3} e^{-2\tau_0}, \quad  ac+2c+\frac a2+\frac {bc}2  +\frac{b}{2}
	<\frac\lambda 2\sin\ell-(1-\lambda)\sin\frac \ell 2.
\end{equation}
We note that since $\ell$ was chosen in $\left(0,2\cos^{-1} \Big( \frac{1}{\lambda} - 1 \Big) \right)$, we have 
\[ \frac\lambda 2\sin\ell-(1-\lambda)\sin\frac \ell 2>0, \]
so the last condition is feasible. We denote by  $\Phi^{\mathrm{NID,1st}}$ the solution to the first order Kuramoto model, in normalized time $\tau$, with possibly nonidentical normalized intrinsic velocities $\mathcal{V}/\kappa$:
\begin{align*}
	\begin{cases}
		\displaystyle \dot\phi_i^{\mathrm{NID,1st}}(\tau) = \frac{\nu_i}\kappa + \frac{1}{N}\sum_{j=1}^N \sin(\phi_j^{\mathrm{NID,1st}}(\tau) - \phi_i^{\mathrm{NID,1st}}(\tau)),\quad \tau > 0,\\
		\displaystyle \phi_i^{\mathrm{NID,1st}}(0) = \theta_i^0,\quad i\in [N].
	\end{cases}
\end{align*}
Analogously to \cite[Lemma 4.2]{H-K-R}, we have
\begin{equation}\label{eq:id-nonid-comparison}
	\mathcal{D}(\Phi^{\mathrm{NID,1st}}(\tau)-\Phi^{\mathrm{ID,1st}}(\tau))
	\le \frac{\mathcal{D}(\mathcal{V})}{\kappa}\left(e^{2\tau}-1\right),\quad \tau\ge 0.
\end{equation}
Indeed, we have, for $i,j\in [N]$,
\begin{align*}
	\begin{aligned}
		&\dot\phi_i^{\mathrm{NID,1st}}-\dot\phi_j^{\mathrm{NID,1st}}-\dot\phi_i^{\mathrm{ID,1st}}+\dot\phi_j^{\mathrm{ID,1st}}\\
		& \hspace{0.5cm} =\frac{\nu_i-\nu_j}\kappa+\frac{1}{N}\sum_{k=1}^N \left(\sin(\phi_k^{\mathrm{NID,1st}} - \phi_i^{\mathrm{NID,1st}})-\sin(\phi_k^{\mathrm{ID,1st}} - \phi_i^{\mathrm{ID,1st}})\right)\\
		& \hspace{0.5cm}-\frac{1}{N}\sum_{k=1}^N \left(\sin(\phi_k^{\mathrm{NID,1st}} - \phi_j^{\mathrm{NID,1st}})-\sin(\phi_k^{\mathrm{ID,1st}} - \phi_j^{\mathrm{ID,1st}})\right).
	\end{aligned}
\end{align*}
Now, we take the absolute values of the above relation to get 
\[
\left|\dot\phi_i^{\mathrm{NID,1st}}-\dot\phi_j^{\mathrm{NID,1st}}-\dot\phi_i^{\mathrm{ID,1st}}+\dot\phi_j^{\mathrm{ID,1st}}\right|\le \frac{\mathcal{D}(\mathcal{V})}\kappa+2\mathcal{D}(\Phi^{\mathrm{NID,1st}}(\tau)-\Phi^{\mathrm{ID,1st}}(\tau)),
\]
for $i,j\in [N]$.  Taking any pair $i,j\in [N]$ which maximizes the quantity 
\[ \phi_i^{\mathrm{NID,1st}}-\phi_j^{\mathrm{NID,1st}}-\phi_i^{\mathrm{ID,1st}}+\phi_j^{\mathrm{ID,1st}},\]
we find that
\[
D^+\mathcal{D}(\Phi^{\mathrm{NID,1st}}(\tau)-\Phi^{\mathrm{ID,1st}}(\tau))\le \frac{\mathcal{D}(\mathcal{V})}\kappa+2\mathcal{D}(\Phi^{\mathrm{NID,1st}}(\tau)-\Phi^{\mathrm{ID,1st}}(\tau)),
\]
from which \eqref{eq:id-nonid-comparison} follows by Gr\"onwall's inequality. \newline

\noindent By condition $\eqref{eq:abc-1}_1$, $\frac{\mathcal{D}(\mathcal{V})}{\kappa}\le a\le \frac \ell 3e^{-2\tau_0}$ and \eqref{eq:id-nonid-comparison}, we have
\[
\mathcal{D}(\Phi^{\mathrm{NID,1st}}_\mathcal{A}(\tau_0)-\Phi^{\mathrm{ID,1st}}_\mathcal{A}(\tau_0)) \le \mathcal{D}(\Phi^{\mathrm{NID,1st}}(\tau_0)-\Phi^{\mathrm{ID,1st}}(\tau_0))
\le \frac{\mathcal{D}(\mathcal{V})}{\kappa}\left(e^{2\tau_0}-1\right)< \frac \ell 3,
\]
and
\[
\mathcal{D}( \Phi^{\mathrm{NID,1st}}_\mathcal{A}(\tau_0)) \stackrel{\mathclap{\eqref{eq:triangle-ineq}}}{\le} \mathcal{D}(\Phi^{\mathrm{ID,1st}}_\mathcal{A}(\tau_0)) + \mathcal{D}(\Phi^{\mathrm{NID,1st}}_\mathcal{A}(\tau_0)-\Phi^{\mathrm{ID,1st}}_\mathcal{A}(\tau_0)) < \frac \ell3 +\frac \ell3=\frac {2\ell}3.
\]
Now, we set $\Phi(\tau)=\Theta(\tau/\kappa)$. Then, this solves the second-order Cauchy problem in normalized time:
\begin{align*}
	\begin{cases}
		\displaystyle m\kappa\ddot\phi_i(\tau) +\dot\phi_i(\tau) = \frac{\nu_i}\kappa + \frac{1}{N}\sum_{j=1}^N \sin(\phi_j(\tau) - \phi_i(\tau)),\quad \tau > 0,\\
		\displaystyle \phi_i(0) = \theta_i^0,~\dot\phi_i(0)=\frac{\omega_i^0}{\kappa},\quad i\in [N].
	\end{cases}
\end{align*}
By \eqref{eq:C0-approx-rel-vanilla} of Proposition \ref{prop:mto0quant}, we have
\[
\|(\Phi^{\mathrm{NID,1st}}(\tau)-\Phi(\tau))\wedge {\bf 1}_{[N]}\|_\infty \stackrel{\mathclap{\eqref{eq:C0-approx-rel-vanilla}}}{<} m\kappa\left( \frac{\mathcal{D}(\Omega^0 -\mathcal{V})}{\kappa } +2\right) e^{2\tau} 
\stackrel{\mathclap{\eqref{eq:triangle-ineq}}}{\le}  m\left(\mathcal{D}(\Omega^0)+\mathcal{D}(\mathcal{V})+2\kappa\right)e^{2 \tau},\quad \tau\ge 0.
\]
In particular, by condition $\eqref{eq:abc-1}_2$, we have
\[
m\left(\mathcal{D}(\Omega^0)+\mathcal{D}(\mathcal{V})+2\kappa\right)=m\kappa\left(\frac{\mathcal{D}(\Omega^0)}\kappa+\frac{\mathcal{D}(\mathcal{V})}\kappa+2\right)\le c(a+b+2)\le \frac{\ell e^{-2\tau_0}}{3},
\]
so that
\[
\|(\Phi^{\mathrm{NID,1st}}_\mathcal{A}(\tau_0)-\Phi_\mathcal{A}(\tau_0))\wedge {\bf 1}_{\mathcal{A}}\|_\infty \le \|(\Phi^{\mathrm{NID,1st}}(\tau_0)-\Phi(\tau_0))\wedge {\bf 1}_{[N]}\|_\infty < m\left(\mathcal{D}(\Omega^0)+\mathcal{D}(\mathcal{V})+2\kappa\right)e^{2 \tau_0}\le \frac{\ell}{3},
\]
and so
\[
\mathcal{D}(\Theta_\mathcal{A}(\tau_0/\kappa))=\mathcal{D}(\Phi_\mathcal{A}(\tau_0))
\stackrel{\mathclap{\eqref{eq:triangle-ineq}}}{\le} \mathcal{D}( \Phi^{\mathrm{NID,1st}}_\mathcal{A}(\tau_0)) + \|(\Phi^{\mathrm{NID,1st}}_\mathcal{A}(\tau_0)-\Phi_\mathcal{A}(\tau_0))\wedge {\bf 1}_{[N]}\|_\infty< \frac {2\ell}3 +\frac \ell3=\ell.
\]
Note that, by condition $\eqref{eq:abc-1}_2$, $\ell<\pi$ and $\tau_0\ge 1$,
\[
2m\kappa\le 2c\le c\left(a+b+2\right)\le\frac{\ell e^{-2\tau_0}}{3}\le \frac{\pi e^{-2}}{3}<\frac 12,
\]
so that $m\kappa<\frac 14$. We invoke Theorem \ref{L4.4} at time $t=\tau_0/\kappa$, with $\mathcal{A}$ as it is, $\mathcal{B}=[N]$, and $\eta=1$. The hypotheses of Theorem \ref{L4.4} are satisfied since $t=\tau_0/\kappa\ge 1/\kappa>4m> \eta m$,
\begin{align*}
	\xi(m,\kappa,\mathcal{V}_\mathcal{A},\Omega^0_\mathcal{A},1)&=m\mathcal{D}(\mathcal{V}_\mathcal{A})+2m\kappa+\frac{\mathcal{D}(\mathcal{V}_\mathcal{A})}{2\kappa}+D(\Omega^0_\mathcal{A})m e^{-1}+\frac{\mathcal{D}(\Omega^0_\mathcal{A})}{2\kappa}\frac{e^{-1}}{1-e^{-1}}\\
	&\le m\kappa\cdot\frac{\mathcal{D}(\mathcal{V})}\kappa+2m\kappa+\frac{\mathcal{D}(\mathcal{V})}{2\kappa}+\frac {D(\Omega^0)}{\kappa}\cdot m\kappa +\frac{\mathcal{D}(\Omega^0)}{2\kappa}\\
	&\le ac+2c+\frac a2+bc+\frac b2 \stackrel{\mathclap{\eqref{eq:abc-1}_3}}{<}~\frac\lambda 2\sin\ell-(1-\lambda)\sin\frac \ell 2,
\end{align*}
and
\begin{align*}
	\xi(m,\kappa,\mathcal{V},\Omega^0,\infty)=m\mathcal{D}(\mathcal{V})+2m\kappa+\frac{\mathcal{D}(\mathcal{V})}{2\kappa}\le ac+2c+\frac a2~\stackrel{\mathclap{\eqref{eq:abc-1}_3}}{<}~\frac\lambda 2\sin\ell-(1-\lambda)\sin\frac \ell 2.
\end{align*}
Thus, by Theorem \ref{L4.4}, asymptotic phase-locking occurs, and 
\[ \mathcal{D}(\Theta_\mathcal{A}(t))\le \ell \quad \mbox{for $t\ge \tau_0/\kappa$}. \]
This gives that, denoting $\theta_a \coloneqq (\max_{i\in \mathcal{A}} \theta_i + \min_{i\in\mathcal{A}}\theta_i)/2$,
\begin{align*}
	R(t)&=\left|\frac{1}{N} \sum_{k=1}^N e^{\mathrm{i} \theta_k}\right|= \left|\frac{1}{N} \sum_{k=1}^N e^{\mathrm{i} \theta_k} e^{-\mathrm{i} \theta_a} \right| \ge  \frac{1}{N}\sum_{k=1}^N \cos\left(\theta_k-\theta_a\right) \\
	& = \frac{1}{N}\sum_{k\in\mathcal{A}} \cos\left(\theta_k-\theta_a\right)+\frac{1}{N}\sum_{k\in [N]\backslash\mathcal{A}} \cos\left(\theta_k-\theta_a\right)\\
	&\ge\frac{|\mathcal{A}|}{N}\cos \frac{\mathcal{D}(\Theta_\mathcal{A})}{2}-\frac{N-|\mathcal{A}|}{N} \ge \lambda\cos\frac \ell 2 - (1-\lambda),\quad t\ge \tau_0/\kappa,
\end{align*}
and so by \eqref{qualitative_mto0-A2},
\[
\lim_{t\to\infty}R(t)\ge
\begin{cases}
	1-\varepsilon&\mathrm{if~}N=2,\\
	1-\frac 2N-\varepsilon&\mathrm{if~}N=3\mathrm{~or~}\eqref{eq:init-nonidentical}\mathrm{~holds},\\
	R^0-\varepsilon&\mathrm{otherwise}.
\end{cases}
\]

%
%
%
%
\section{On reconstructing $\dot{\Theta}(t)$ from $\Theta(t)$ and $\dot{\Theta}^0$}\label{sec:bound}
\setcounter{equation}{0} 
Lemma \ref{L:approxaut} is a key tool in \cite{C-D-H-R25}, as it approximates the inertial Kuramoto model \eqref{A-1} to the first-order Kuramoto model \eqref{A-2}, in the sense of partially recovering $\dot\Theta(t)$ from $\Theta(t)$ and $\dot\Theta^0$, enabling the use of Gr\"onwall-type estimates. One may inquire the optimality of Lemma \ref{L:approxaut}, as improving upon Lemma \ref{L:approxaut} may bring improvements upon the results of \cite{C-D-H-R25}. In this section, we describe an improvement upon Lemma \ref{L:approxaut} given as Proposition \ref{prop:banach-contraction}, while we also show a theoretical limit to recovering $\dot\Theta(t)$ from $\Theta(t)$ and $\dot\Theta^0$, namely Proposition \ref{prop:determinability}.

From the point of view of \eqref{B-1}, $\Omega$ and $\Theta$ are essentially independent variables. Indeed, compared to the first-order model \eqref{A-2}, where the derivative $\dot{\Theta}(t)$ at a fixed time $t$ is solely determined by $\Theta(t)$ (forgetting about $\Theta^0$), the second-order model \eqref{A-1} is more difficult to analyze because the derivative $\dot{\Theta}(t)$ at a fixed time $t$ is not determined by $\Theta(t)$ (again, forgetting about $\Theta^0$ and $\dot\Theta^0$), but rather $\dot{\Theta}(t)$ together with $\Theta(t)$ determines the past and future dynamics of $\Theta$. However, this does not rule out the possibility of recovering information about $\dot{\Theta}(t)$ \emph{from $\Theta(t)$ and $\dot{\Theta}^0$}. Lemma \ref{L:approxaut} tells us how to estimate $\dot{\Theta}(t)$ from $\Theta(t)$ and $\dot{\Theta}^0$ (and $t$) up to a small additive error, and effectively reduce the dimension of system \eqref{A-1}-\eqref{B-1} from $2N$ (that of $(\Theta,\Omega)$) to $N$ (that of $\Theta$). Can we do better than Lemma \ref{L:approxaut}? Is it even possible to \emph{determine} $\dot{\Theta}(t)$ from $\Theta(t)$ and $\dot{\Theta}^0$?

In the proof of Lemma \ref{L:approxaut} given in \cite{C-D-H-R25}, the Duhamel principle \eqref{B-4} was essentially invoked twice: once in the proof of Lemma \ref{L2.2}, and again in the proof of Lemma \ref{L:approxaut}, while using the boundedness $|\sin|\le 1$. For short time ranges, we may iterate the Duhamel principle \eqref{B-4} indefinitely.
\begin{proposition}\label{prop:banach-contraction}
Fix $N\in \mathbb{Z}_{>0}$, $\kappa,m,t_0>0$, and $\mathcal{V},\Omega^0,\Theta^*(t_0)\in \mathbb{R}^N$, where we denote $\Theta^*(t_0)=(\theta^*_1(t_0),\cdots,\theta^*_N(t_0))$. Define the map $\mathcal{F}=(f_1,\cdots,f_N):(C[0,t_0])^N\to (C[0,t_0])^N$ as follows: given $\Omega^*=(\omega_1^*,\cdots,\omega_N^*)\in (C[0,t_0])^N$ and $t\in [0,t_0]$, we set
\begin{align*}
f_i(\Omega^*)(t) &=  \omega^0_i e^{-t/m} + \nu_i (1 -e^{-t/m}) \\
&+ \frac{\kappa}{Nm}\sum_{l=1}^N \int_0^t e^{-(t-\sigma)/m}\sin\left(\theta^*_l(t_0) -\theta^*_i(t_0)-\int_\sigma^{t_0} (\omega^*_l(\tau)-\omega^*_i(\tau))d\tau\right) d\sigma.
\end{align*}
Then, the following statements \footnote{In this proposition, an asterisk attached to a variable, as in $\Theta^*(t_0)$ or $\Omega^*$, signifies that it is a dummy variable in place of the actual variables $\Theta$ and $\Omega$ arising from a solution to \eqref{B-1}. The motivation is that we want to find solutions $(\Theta,\Omega)$ to \eqref{B-1}, and $\Omega^*$ is an approximation we have at hand to a bona fide $\Omega$.} hold. 
\begin{enumerate}
    \item Given $\Omega^*\in (C([0,t_0]))^N$, there exists a (necessarily unique) solution $(\Theta,\Omega)$ to \eqref{B-1}, for some choice of initial data $\Theta^0$, with $\Theta(t_0)=\Theta^*(t_0)$, $\Omega(0)=\Omega^0$, and $\Omega(t)=\Omega^*(t)$ for $t\in [0,t_0]$, if and only if it is a fixed point of $\mathcal{F}$: $\mathcal{F}(\Omega^*)=\Omega^*$.
    \vspace{0.2cm}
 \item The map $\mathcal{F}:\ell^N_\infty(C[0,t_0])\to \ell^N_\infty(C[0,t_0])$, where $C[0,t_0]$ is given the supremum norm, is $\min\{2\kappa t_0,\frac{\kappa t_0^2}{m}\}$-Lipschitz.
   \vspace{0.2cm}
\item
 If $t_0<\max\left\{\frac{1}{2\kappa},\sqrt{\frac{m}{\kappa}}\right\}$, there exists a unique solution $(\Theta,\Omega)$ to \eqref{B-1} with $\Theta(t_0)=\Theta^*(t_0)$ and $\Omega(0)=\Omega^0$, and we can find it by iterating $\mathcal{F}$: start with any $\Omega^*\in C[0,t_0]^N$ (such as $\Omega^*(t)=\Omega^0$, $t\in [0,t_0]$), define
 \[
 \Omega=\lim_{n\to\infty}\mathcal{F}^n(\Omega^*),
 \]
 and obtain $\Theta$ on $[0,t_0]$ by integration:
\[
 \Theta(t)=\Theta^*(t_0)-\int_t^{t_0}\Omega(\tau)d\tau,\quad t\in [0,t_0].
\]
Set $\Theta^0\coloneqq \Theta(0)$, and now $(\Theta,\Omega)$ is the solution to \eqref{B-1} with initial data $(\Theta^0,\Omega^0)$.
 \end{enumerate}
\end{proposition}
\begin{proof}
\noindent (i)~A solution $(\Theta,\Omega)$ to \eqref{B-1} with $\Theta(t_0)=\Theta^*(t_0)$ and $\Omega(t)=\Omega^*(t)$ for $t\in [0,t_0]$ must necessarily be unique, for it must satisfy
\begin{equation}\label{eq:t-from-t-star}
\Theta(t)=\Theta(t_0)-\int_t^{t_0}\Omega(s)ds=\Theta^*(t_0)-\int_t^{t_0}\Omega^*(s)ds,\quad t\in [0,t_0],
\end{equation}
and $\Theta(t)$ for $t>t_0$ would be determined by uniqueness of solutions to \eqref{B-1}.

If $(\Theta,\Omega)$ is a solution to \eqref{B-1} with $\Theta(t_0)=\Theta^*(t_0)$, $\Omega(0)=\Omega^0$, and $\Omega(t)=\Omega^*(t)$ for $t\in [0,t_0]$, then it follows from \eqref{eq:t-from-t-star} and the Duhamel principle \eqref{B-4}  that 
\[ \mathcal{F}(\Omega)(t)=\Omega(t),\quad t\in[0,t_0], \quad \mbox{hence}~~\mathcal{F}(\Omega^*)=\Omega^*. \]

Conversely, suppose $\Omega^*\in (C([0,t_0]))^N$ satisfies $\mathcal{F}(\Omega^*)=\Omega^*$. Define $\Theta_{[0,t_0]}\in C[0,t_0]^N$ using \eqref{eq:t-from-t-star}. Then, we have
\[ \Theta_{[0,t_0]}(t_0)=\Theta^*(t_0) \quad \mbox{and} \quad  \dot\Theta_{[0,t_0]}(t)=\Omega^*(t) \quad \mbox{for $t\in [0,t_0]$}. \]
By definition of $\mathcal{F}$ and the fact that $\mathcal{F}(\Omega^*)=\Omega^*$, we have the Duhamel principle: for $t\in [0,t_0]$,
\[
\omega^*_i(t) =  \omega^0_i e^{-t/m} + \nu_i (1 -e^{-t/m}) + \frac{\kappa}{Nm}\sum_{l=1}^N \int_0^t e^{-(t-\sigma)/m}\sin\left({(\theta_{[0,t_0]})}_l(\sigma) -{(\theta_{[0,t_0]})}_i(\sigma)\right) d\sigma,~~ i\in [N].
\]
Substituting $t=0$ gives $\Omega^*(0)=\Omega^0$. On the other hand, multiplying $me^{t/m}$, differentiating in $t$ and dividing by $e^{t/m}$ on both sides of the above equality, gives
\[
m\dot\omega^*_i(t) e^{t/m}+\omega^*_i(t)=\nu_i+\frac \kappa N\sum_{l=1}^N\sin({(\theta_{[0,t_0]})}_l(t)-{(\theta_{[0,t_0]})}_i(t)),\quad i\in [N],~t\in [0,t_0].
\]
Thus, $\Theta_{[0,t_0]}$ and $\Omega^*$ solve \eqref{B-1} on the time interval $[0,t_0]$. By global existence and uniqueness of solutions to \eqref{B-1}, there is a global solution $(\Theta,\Omega)$ that agrees with $(\Theta_{[0,t_0]},\Omega^*)$ on $[0,t_0]$. This solution satisfies $\Omega(0)=\Omega^0$ and $\Theta(t_0)=\Theta^*(t_0)$.

\vspace{0.2cm}

\noindent (ii)~For $\Omega^*,\tilde\Omega^*\in (C([0,t_0]))^N$, $t\in [0,t_0]$, and $i\in [N]$, we have
\begin{align*}
&|f_i(\Omega^*)(t)-f_i(\tilde\Omega^*)(t)|\\
& \hspace{0.5cm} \le \frac{\kappa}{Nm}\sum_{l=1}^N \int_0^t e^{-(t-\sigma)/m}\left| \int_\sigma^{t_0} (\omega^*_l(\tau)-\omega^*_i(\tau))d\tau-\int_\sigma^{t_0} (\tilde\omega^*_l(\tau)-\tilde\omega^*_i(\tau))d\tau\right| d\sigma\\
&  \hspace{0.5cm}  \le \frac{2\kappa}{m}\int_0^t e^{-(t-\sigma)/m}(t_0-\sigma) d\sigma\cdot \max_{i=1,\cdots,N}\|\omega^*_i-\tilde\omega^*_i\|_{C[0,t_0]}\\
& \hspace{0.5cm}  = 2\kappa\left( t_0-t+m-t_0e^{-t/m}-me^{-t/m}\right)\cdot\|\Omega^*-\tilde\Omega^*\|_{\ell^\infty_N(C[0,t_0])}\\
&  \hspace{0.5cm}  \le 2\kappa\left(t_0-m\log\left(1+\frac{t_0}{m}\right) \right)\cdot \|\Omega^*-\tilde\Omega^*\|_{\ell^\infty_N(C[0,t_0])}\\
& \hspace{0.5cm}  \le \min\left\{2\kappa t_0,\frac{\kappa t_0^2}{m}\right\}\cdot \|\Omega^*-\tilde\Omega^*\|_{\ell^\infty_N(C[0,t_0])},
\end{align*}
where in the penultimate inequality, the maximum is attained at $t=m\log\left(1+\frac{t_0}{m}\right)$, while in the last inequality we used $\log(1+x)\ge \max\left\{0,x-\frac{x^2}{2}\right\}$ for $x\ge 0$.

\vspace{0.2cm}

\noindent (iii)~By statement (1), solutions $(\Theta,\Omega)$ to \eqref{B-1} with $\Theta(t_0)=\Theta^*(t_0)$ and $\Omega(0)=\Omega^0$ are in one-to-one correspondence with fixed points of $\mathcal{F}$. However given the bound on $t_0$, the map $\mathcal{F}$ is a contraction with respect to the $\ell_\infty^N( C([0,t_0]))$ norm, and hence, by the Banach contraction principle, it possesses a unique fixed point, which can be found by iteration, as described in statement (3).
\end{proof}

\begin{remark}
We have not aimed to optimize Proposition \ref{prop:banach-contraction}: perhaps a choice of a better norm, or considering phase differences instead of phases themselves, might lead to a longer time period than $\max\left\{\frac{1}{2\kappa},\sqrt{\frac{m}{\kappa}}\right\}$. However, the time period cannot be taken to be infinitely long when $m\kappa>\frac 14$, as demonstrated in Proposition \ref{prop:determinability} below.
\end{remark}
How much can we improve upon Proposition \ref{prop:banach-contraction}? Qualitatively, we can decide exactly the time threshold until which we can determine $\dot{\Theta}(t)$ from $\Theta(t)$ and $\dot{\Theta}^0$. Namely, for each $m$ and $\kappa$, there exists a time\footnote{The fact that the time $T^*(\kappa,m)$  should be of the form $mT^*(m\kappa)$ can be inferred from a dimensional analysis, or equivalently from the time dilation symmetry \eqref{eq:dilatation}.} $T^*(\kappa,m)=mT^*(m\kappa)$ such that if $t\le T^*(\kappa,m)$, then $\Theta(t)$ and $\dot\Theta^0$ decide $\dot\Theta(t)$, while if $t>T^*(\kappa,m)$ then there are examples where $\dot\Theta(t)$ is not uniquely decided by $\Theta(t)$ and $\dot\Theta(t)$. Note that this is a uniqueness result that does not give an explicit algorithm for determining $\dot\Theta(t)$ from $\Theta(t)$ and $\dot\Theta^0$, such as the Banach contraction principle given in Proposition \ref{prop:banach-contraction}.

\begin{proposition}\label{prop:determinability}
Let $\kappa,m,t^*>0$. Then, the following assertions hold.
\begin{enumerate}
\item Suppose that 
\[ \mbox{either} \quad m\kappa \le \frac 14; \quad \mbox{or} \quad m\kappa >\frac 14 \quad \mbox{and} \quad t^*\le \frac{\pi m}{\sqrt{4m\kappa -1}}+\frac{2 m}{\sqrt{4m\kappa-1}}\sin^{-1}\left(\frac{1}{\sqrt{4m\kappa}}\right), \]
and let $N\in \mathbb{Z}_{>0}$ and $\mathcal{V}\in \mathbb{R}^N$ be arbitrary. Let $\Theta(t)$ and $\Phi(t)$ be two solutions to \eqref{A-1} with initial data $(\Theta^0,\dot{\Theta}^0)$ and $(\Phi^0,\dot{\Phi}^0)$, respectively. If $\dot{\Theta}^0=\dot{\Phi}^0$ and $\Theta(t^*)=\Phi(t^*)$, then we have
\[ \dot{\Theta}(t^*)=\dot{\Phi}(t^*). \]
\item Suppose that
\[ m\kappa>\frac 14 \quad \mbox{and} \quad  t^*> \frac{\pi m}{\sqrt{4m\kappa -1}}+\frac{2 m}{\sqrt{4m\kappa-1}}\sin^{-1}\left(\frac{1}{\sqrt{4m\kappa}}\right), \quad N \geq 2. \]
Then there exist a natural frequency vector ${\mathcal V}$ and initial data $(\Theta^0,\dot{\Theta}^0)$ and  $(\Phi^0,\dot{\Phi}^0)$ such that if $\Theta(t)$ is a solution to \eqref{A-1} with initial data $(\Theta^0,\dot{\Theta}^0)$ and $\Phi(t)$ is a solution to \eqref{A-1} with initial data $(\Phi^0,\dot{\Phi}^0)$, then we have
\[ \dot{\Theta}^0=\dot{\Phi}^0, \quad \Theta(t^*)=\Phi(t^*) \quad \mbox{and} \quad \dot{\Theta}(t^*)\neq \dot{\Phi}(t^*). \]
\end{enumerate}
\end{proposition}
\begin{remark}Below, we comment on the results of Proposition \ref{prop:determinability}.
\,
\begin{enumerate}
	\item The quantity $\frac 14$ serves as a critical value for $m\kappa$, marking a transition in the behavior of the system \eqref{A-1}. This transition appears in a Sturm-Picone comparison principle, Lemma \ref{lem:simple-sturm-picone}, where $m\kappa-\frac 14$ serves as the discriminant for an associated second-order linear differential equation. This threshold also appears when investigating the cardinality of collisions, as reported in \cite[Theorem 4.1]{C-D-H}\cite[Theorem D.1]{C-D-H-R25}.
\item Proposition \ref{prop:determinability} says that the critical time $T^*(\kappa,m)$ for determinability of $\dot\Theta(t)$ from $\Theta(t)$ and $\dot\Theta^0$ is
\[
T^*(\kappa,m)=
\begin{cases}
	\infty,&m\kappa\le \frac 14,\\
	\frac{\pi m}{\sqrt{4m\kappa -1}}+\frac{2 m}{\sqrt{4m\kappa-1}}\sin^{-1}\left(\frac{1}{\sqrt{4m\kappa}}\right),&m\kappa>\frac 14.
\end{cases}
\]
Of course, since \eqref{B-1} is a time-autonomous system, for $t_0,t^*\ge 0$, one can say that $\Theta(t_0+t^*)$ and $\dot\Theta(t_0)$ determine $\dot\Theta(t_0+t^*)$ if and only if $t^*\le T^*(\kappa,m)$.

\item Proposition \ref{prop:determinability} suggests that any strategy to prove asymptotic phase-locking that involves estimating $\dot{\Theta}$ from $\Theta(t)$ and $\dot{\Theta}^0$ would likely be restricted to the small inertia regime $m\kappa\le \frac 14$; in the large inertia regime $m\kappa > \frac 14$, we would have to devise a new strategy, for example along the lines of {\bf Conjecture \ref{conj:lyapunov}}.

		\item Although the correspondence $\mathbb{R}^N\times\mathbb{R}^N\to\mathbb{R}^N\times\mathbb{R}^N$, $(\Theta^0,\dot{\Theta}^0)\mapsto (\Theta(t),\dot{\Theta}(t))$ is a diffeomorphism, 
		this does not necessarily mean that for a fixed $\dot\Theta^0$ the restricted map $\mathbb{R}^N\to\mathbb{R}^N$, $\Theta^0\mapsto \Theta(t)$ is bijective. Nevertheless, Proposition \ref{prop:determinability} tells us that this restricted map is a diffeomorphism when $m\kappa \le \frac 14$ or $m\kappa >\frac 14$ and $t\le \frac{\pi m}{\sqrt{4m\kappa -1}}+\frac{2 m}{\sqrt{4m\kappa-1}}\sin^{-1}\left(\frac{1}{\sqrt{4m\kappa}}\right)$.
		\vspace{0.1cm}
		\item From statement (1) of Proposition \ref{prop:determinability}, it appears that the intrinsic velocities $\mathcal{V}$ play no role in the determinability of $\dot{\Theta}(t)$ from $\dot{\Theta}^0$ and $\Theta(t)$. However, in this regard, statement (2) of Proposition \ref{prop:determinability} has the weakness that it specifies a $\mathcal{V}$ (namely, $\mathcal{V}=0$). It would be interesting to strengthen statement (2) of Proposition \ref{prop:determinability} to hold for arbitrary $\mathcal{V}$.
		\vspace{0.1cm}
		\item Of course, knowledge of not only $\dot\Theta^0$ and $\Theta(t)$ but also $\Theta^0$ is sufficient to determine $\dot\Theta(t)$. We may ask whether there is an explicit function of $\dot\Theta^0$, $\Theta(t)$, and $\Theta^0$ that better approximates $\dot\Theta(t)$, e.g., does the knowledge of $R^0$, for example, give explicit functions with better error bounds than those of Lemma \ref{L:approxaut}?
		\vspace{0.1cm}
		\item Proposition \ref{prop:determinability} does not tell us how to compute $\dot{\Theta}(t)$ from $\dot{\Theta}^0$ and $\Theta(t)$ even if it is uniquely determined, whereas Proposition \ref{prop:banach-contraction} gives an explicit algorithm in shorter time intervals $\left[0,\max\left\{\frac{1}{2\kappa},\sqrt{\frac{m}{\kappa}}\right\}\right)$. It would be interesting to find an effective algorithm which calculates $\dot{\Theta}(t)$ from $\dot\Theta^0$ and $\Theta(t)$ in the full range of Proposition \ref{prop:determinability} (1), namely
		\[
		m\kappa \le \frac 14 \mathrm{~and ~}t^*>0\quad \mathrm{or}\quad m\kappa >\frac 14 \mathrm{~and~} 0<t^*\le \frac{\pi m}{\sqrt{4m\kappa -1}}+\frac{2 m}{\sqrt{4m\kappa-1}}\sin^{-1}\left(\frac{1}{\sqrt{4m\kappa}}\right),
		\]
		perhaps by finding a better version of the mapping $\mathcal{F}$. If such an algorithm was descriptive enough to yield bounds improving those of Lemma \ref{L:approxaut}, and if the bounds did not worsen for large time in the case $m\kappa\le \frac 14$, it may even give stronger sufficient frameworks in \cite{C-D-H-R25} for asymptotic phase-locking (which already assume $m\kappa\le \frac 14$), since the estimates used in their proofs fundamentally depend  on Lemma \ref{L:approxaut}.
	\end{enumerate}
\end{remark}

In the proof of Proposition \ref{prop:determinability}, we will use the following Sturm--Picone comparison principle.
\begin{lemma}[Sturm--Picone Comparison Principle]\label{lem:simple-sturm-picone}
	Let $a,b,c>0$ be positive real numbers, and define the extended real-valued number
	\[
	T^* :=T^*(a,b,c)=
	\begin{cases}
		\infty,&4ac\le b^2,\\
		\frac{\pi a}{\sqrt{4ac-b^2}}+\frac{2 a}{\sqrt{4ac-b^2}}\sin^{-1}\left(\frac{b}{2\sqrt{ac}}\right),&4ac>b^2.
	\end{cases}
	\]
	Let $I\subset \mathbb{R}$ be a connected interval, and let $y:I\to \mathbb{R}$ be a continuous function such that for any subinterval $J\subset I$ on which $\left.y\right|_J>0$ pointwise, we have that $y$ is $C^2$ on $J$ and
	\[
	a\ddot{y}(t)+b\dot{y}(t)+cy(t)>0,\quad \forall~t\in J.
	\]
	Then if there is a time $t_0\in I$ such that $y(t_0)>0$ and $\dot{y}(t_0)\ge 0$, we have
	\[
	y(t)>0\quad\mathrm{for}~t\in I\cap [t_0,t_0+T^*].
	\]
\end{lemma}
\begin{proof}
	The proof of Lemma~\ref{lem:simple-sturm-picone} can be found in Appendix B of \cite{C-D-H-R25}, where the authors present and prove a more detailed version of this lemma.
\end{proof}
\begin{proof}[Proof of Proposition \ref{prop:determinability}]
We will prove the first assertion using the Sturm--Picone comparison principle given in Lemma \ref{lem:simple-sturm-picone}. The second assertion follows from constructing an example. \newline

\noindent (i)~Let $\Theta(t)=(\theta_i)_{i\in[N]}$ and $\Phi(t)=(\varphi_i)_{i\in[N]}$ be global solutions to \eqref{A-1}. Then for $i,j\in [N]$,
\begin{align}\label{eq:rel-ODE}
\begin{aligned}
&m(\ddot \theta_i-\ddot \varphi_i-\ddot \theta_j+\ddot \varphi_j)+(\dot\theta_i-\dot\varphi_i-\dot\theta_j+\dot\varphi_j)\\
& \hspace{1cm} =\frac \kappa N \sum_{k=1}^N \left(\sin(\theta_k-\theta_i)-\sin(\varphi_k-\varphi_i)-\sin(\theta_k-\theta_j)+\sin(\varphi_k-\varphi_j)\right).
\end{aligned}
\end{align}
We define the $\ell_2$-mismatch functional\footnote{Of course, the $\ell_2$ norm is taken with respect to the standard basis of $\bigwedge^2\mathbb{R}^N$. We can also use the $\ell_p$-mismatch functional for any $1\le p\le \infty$. If $p$ is an even integer, we get the same result \eqref{eq:2ndODE-L} below. The analysis becomes trickier if $p$ is not an even integer, since the functional may not be differentiable at times whenever $\theta_i-\phi_i-\theta_j+\phi_j=0$ for some distinct indices $i$ and $j$, but by the analyticity of $\Theta(t)$ and $\Phi(t)$, such times are discrete. Perhaps it can be shown that \eqref{eq:2ndODE-L} holds for the $\ell_p$-mismatch functional in a suitable sense with Dini derivatives.}
\begin{equation*}\label{eq:mismatch-functional}
\mathcal{L} = \|(\Theta-\Phi) \wedge {\bf 1}_{[N]}\|_2 = \sqrt{\sum_{i<j}(\theta_i-\varphi_i-\theta_j+\varphi_j)^2}.
\end{equation*}
From the smoothness of $\Theta(t)$ and $\Phi(t)$, it follows that $\mathcal{L}(t)$ is smooth at times $t$ when $\mathcal{L}(t)\neq 0$. Whenever $\mathcal{L}(t)\neq 0$, we may find the orbital derivatives of $\mathcal{L}$ by squaring and differentiating:
\begin{align}
\mathcal{L}^2&= \sum_{i<j}(\theta_i-\varphi_i-\theta_j+\varphi_j)^2,\nonumber \\ 
\mathcal{L}\dot{\mathcal{L}}&= \sum_{i<j}(\theta_i-\varphi_i-\theta_j+\varphi_j)(\dot\theta_i-\dot\varphi_i-\dot\theta_j+\dot\varphi_j),\label{eq:first-derivative} \\
\mathcal{L}\ddot{\mathcal{L}}&= \sum_{i<j}(\theta_i-\varphi_i-\theta_j+\varphi_j)(\ddot\theta_i-\ddot\varphi_i-\ddot\theta_j+\ddot\varphi_j)+\underbrace{\sum_{i<j}(\dot\theta_i-\dot\varphi_i-\dot\theta_j+\dot\varphi_j)^2-\dot{\mathcal{L}}^2}_{\ge 0~(\because \eqref{eq:first-derivative}\land \mathrm{Cauchy-Schwarz})},\label{eq:second-derivative}
\end{align}
and so
\begin{align}\label{eq:second-diff-calc}
\begin{aligned}
&\mathcal{L} \left(m\ddot{ \mathcal{L}}+\dot{\mathcal{L}}+\kappa\mathcal{L}\right) =m\mathcal{L}\ddot{\mathcal{L}}+\mathcal{L}\dot{\mathcal{L}}+\kappa \mathcal{L}^2\\
& \hspace{1cm} \stackrel{\mathclap{\eqref{eq:second-derivative}}}{\geq}~\sum_{i<j}(\theta_i-\varphi_i-\theta_j+\varphi_j)\left(m(\ddot \theta_i-\ddot \varphi_i-\ddot \theta_j+\ddot \varphi_j)+(\dot\theta_i-\dot\varphi_i-\dot\theta_j+\dot\varphi_j)\right)+\kappa \mathcal{L}^2\\
&  \hspace{1cm} \stackrel{\mathclap{\eqref{eq:rel-ODE}}}{=}~\sum_{i<j}(\theta_i-\varphi_i-\theta_j+\varphi_j)\cdot \frac \kappa N \sum_{k=1}^N \Big(\sin(\theta_k-\theta_i)-\sin(\varphi_k-\varphi_i)-\sin(\theta_k-\theta_j)+\sin(\varphi_k-\varphi_j) \Big) \\
& \hspace{1cm} +\kappa \mathcal{L}^2\\
& \hspace{1cm} =\frac \kappa N\sum_{i,j,k}(\theta_i-\varphi_i-\theta_j+\varphi_j) \left(\sin(\theta_k-\theta_i)-\sin(\varphi_k-\varphi_i)\right)+\kappa \mathcal{L}^2 \\
& \hspace{1cm} =: {\mathcal I}_{1} +  \kappa \mathcal{L}^2.
\end{aligned}
\end{align}
Next, we claim that 
\begin{equation} \label{NN-5}
{\mathcal I}_{1} \geq -\kappa\mathcal{L}^2. 
\end{equation}
{\it Proof of \eqref{NN-5}}:  By direct calculation, one has  
\begin{align*}
 {\mathcal I}_{1} &=  \frac \kappa N\sum_{i,j,k}(\theta_i-\varphi_i-\theta_j+\varphi_j) \left(\sin(\theta_k-\theta_i)-\sin(\varphi_k-\varphi_i)\right)\\
    &=\frac \kappa N\sum_{i,j,k}(\theta_k-\varphi_k-\theta_j+\varphi_j) \left(\sin(\theta_i-\theta_k)-\sin(\varphi_i-\varphi_k)\right)\quad (\because \mathrm{switch~}i\mathrm{~and~}k)\\
    &=-\frac \kappa N\sum_{i,j,k}(\theta_k-\varphi_k-\theta_j+\varphi_j) \left(\sin(\theta_k-\theta_i)-\sin(\varphi_k-\varphi_i)\right) \\
    &=\frac \kappa {2N}\sum_{i,j,k}(\theta_i-\varphi_i-\theta_k+\varphi_k) \left(\sin(\theta_k-\theta_i)-\sin(\varphi_k-\varphi_i)\right)\\
    &=\kappa \sum_{i<k}(\theta_i-\varphi_i-\theta_k+\varphi_k) \left(\sin(\theta_k-\theta_i)-\sin(\varphi_k-\varphi_i)\right)\\
    &\ge -\kappa \left(\sum_{i<k}(\theta_i-\varphi_i-\theta_k+\varphi_k)^2\right)^{1/2} \left(\sum_{i<k}\left(\sin(\theta_k-\theta_i)-\sin(\varphi_k-\varphi_i)\right)^2\right)^{1/2}\\
    &\ge -\kappa \mathcal{L} \left(\sum_{i<k}\left(\theta_k-\theta_i-\varphi_k+\varphi_i\right)^2\right)^{1/2}=-\kappa\mathcal{L}^2,
\end{align*}
where we used the average of the first and third lines in the fourth line.  Equality in the final inequality is impossible unless $\mathcal{L}=0$. \newline 

\noindent By \eqref{eq:second-diff-calc} and \eqref{NN-5} on an interval where $\mathcal{L}>0$, we have
\begin{align}\label{eq:2ndODE-L}
m\ddot{ \mathcal{L}}+\dot{\mathcal{L}}+\kappa\mathcal{L}>0.
\end{align}
Thus, the hypothesis of Lemma \ref{lem:simple-sturm-picone} is satisfied with the following setting:
\[ a=m, \quad b=1, \quad c=\kappa, \quad  I=[0,\infty), \quad \mbox{and} \quad y=\mathcal{L}. \]
Since $\Theta(t^*)=\Phi(t^*)$, we have 
\begin{equation} \label{NN-6}
\mathcal{L}(t^*)=0. 
\end{equation}
Next, we claim that 
\[ \mathcal{L}(0)=0. \]
If $\mathcal{L}(0)>0$, then, it follows from $\dot{\Theta}^0=\dot{\Phi}^0$ and \eqref{eq:first-derivative} that $\dot{\mathcal{L}}(0)=0$.
But then, it follows from Lemma \ref{lem:simple-sturm-picone} that $\mathcal{L}(t)>0$ for $t\in [0,\infty)\cap[0,T^*]$, with 
\[
T^*=T^*(a,b,c)=T^*(m,1,\kappa)=
\begin{cases}
    \infty,&m\kappa\le \frac 14,\\
\frac{\pi m}{\sqrt{4m\kappa -1}}+\frac{2 m}{\sqrt{4m\kappa-1}}\sin^{-1}\left(\frac{1}{\sqrt{4m\kappa}}\right), &m\kappa> \frac 14.
\end{cases}
\]
By assumption, $t^*\le T^*$, so that $\mathcal{L}(t^*)>0$ which contradicts to \eqref{NN-6}. Therefore, we have $\mathcal{L}(0)=0$ which means that 
\[ \theta_i^0-\varphi_i^0=\theta_j^0-\varphi_j^0 \quad \mbox{for all $i,j\in [N]$}, \]
i.e. there is a constant $\alpha\in \mathbb{R}$ such that $\theta_i^0=\varphi_i^0+\alpha$ for $i\in [N]$. But we also know $\dot{\Theta}^0=\dot{\Phi}^0$. By the Galilean symmetry \eqref{eq:galilean}, it follows that $\theta_i(t)=\varphi_i(t)+\alpha$ for $i\in [N]$ and $t\ge 0$. From $\Theta(t^*)=\Phi(t^*)$, we have $\alpha=0$, thus $\Theta(t)=\Phi(t)$ for all $t\ge 0$, and, a fortiori, $\dot{\Theta}(t^*)=\dot{\Phi}(t^*)$.

\vspace{0.2cm}

\noindent (ii)~Let $N\ge 2$, $\mathcal{V}=0$, and $\eta\in (0,\pi )$. We decompose $N=N_1+N_2$ where $N_1,N_2\ge 1$, and consider the distinct initial data:
\begin{align*}
\begin{aligned}
& \theta_1^0=\cdots=\theta_{N_1}^0=\eta N_2/N,\quad \theta_{N_1+1}^0=\cdots=\theta_N^0=-\eta N_1/N, \quad \varphi_i^0=-\theta_i^0,\quad i \in [N], \\
& \dot{\theta}_i^0=\dot{\varphi}_i^0=0,\quad i \in [N].
\end{aligned}
\end{align*}
 Now, we consider the solutions $(\Theta,\dot{\Theta})$ and $(\Phi,\dot{\Phi})$ to \eqref{A-1} with initial conditions $\Theta(0)=(\theta_1^0,\cdots,\theta_N^0)$ and $\dot\Theta(0)=(\dot\theta_1^0,\cdots,\dot\theta_N^0)$, and $\Phi(0)=(\varphi_1^0,\cdots,\varphi_N^0)$ and $\dot\Phi(0)=(\dot\varphi_1^0,\cdots,\dot\varphi_N^0)$, respectively. By the particle exchange symmetry \eqref{eq:particle-exchange}, we have 
 \[ \theta_1(t)=\cdots=\theta_{N_1}(t) \quad \mbox{and} \quad \theta_{N_1+1}(t)=\cdots=\theta_N(t) \quad \mbox{for all $t\ge 0$}. \]
 On the other hand, by \eqref{eq:galconserve}, we have
 \[ \frac 1N\sum_{i=1}^N \theta_i(t)=0 \quad \mbox{for all $t \ge 0$}. \]
Therefore, denoting $\theta_{\mathrm{rel}}(t)\coloneqq \theta_1(t)-\theta_N(t)$, $t\ge 0$, we have
\[
\theta_1(t)=\cdots=\theta_{N_1}(t)=\frac{N_2}{N}\theta_{\mathrm{rel}}(t),~\theta_{N_1+1}(t)=\cdots=\theta_{N}(t)=-\frac{N_1}{N}\theta_{\mathrm{rel}}(t),\quad t\ge 0.
\]
On the other hand, by the reflection symmetry \eqref{eq:reflection}, we have
\[
\varphi_i(t)=-\theta_i(t),\quad i \in [N],~t\ge 0.
\]
Therefore, the dynamics is one-dimensional, depending on the single variable $\theta_{\mathrm{rel}}(t)$. Its dynamics is governed by the following Cauchy problem to the second-order ODE:
\[
\begin{cases}
\displaystyle m\ddot \theta_{\mathrm{rel}}+\dot \theta_{\mathrm{rel}}=-\kappa  \sin\left(\theta_{\mathrm{rel}}\right),\quad t>0, \\
\displaystyle \theta_{\mathrm{rel}}(0)=\eta, \quad \dot\theta_{\mathrm{rel}}(0)=0.
\end{cases}
\]
In other words, $\theta_{\mathrm{rel}}$ corresponds to the phase of a damped circular pendulum under a uniform gravitational field. As $\eta\to 0$, $\theta_{\mathrm{rel}}/\eta$ approaches the solution $\theta$ to the second-order linear ordinary differential equation
\[
m\ddot \theta+ \dot\theta =-\kappa \theta,\quad \theta(0)=1,\quad \dot\theta(0)=0.
\]
As $m\kappa>\frac 14$, the solution is
\[
\theta(t)=e^{-t/2m}\left(\cos\left(\frac{\sqrt{4m\kappa -1}}{2m}t\right)+\frac{1}{\sqrt{4m\kappa -1}}\sin\left(\frac{\sqrt{4m\kappa -1}}{2m}t\right)\right)
\]
and has its first zero on $[0,\infty)$ at $t_1=\frac{\pi m}{\sqrt{4m\kappa -1}}+\frac{2 m}{\sqrt{4m\kappa-1}}\sin^{-1}\left(\frac{1}{\sqrt{4m\kappa}}\right)$.

It is well-known from classical mechanics that $\theta_{\mathrm{rel}}$ also has a smallest positive zero $t_1^\eta$, which is simple and which approaches the zero $t_1$ as $\eta\to 0$ and which grows to infinity as $\eta\to\frac{\pi}2$ (we always have $t_1^\eta>t_1$ by the Sturm--Picone comparison principle, Lemma \ref{lem:simple-sturm-picone}). Thus, given $t^*>t_1$, by the intermediate value theorem, there exists $\eta\in (0,\frac \pi 2)$ such that $t^*=t_1^\eta$. Then $\theta_{\mathrm{rel}}(t^*)=0$ and $\dot\theta_{\mathrm{rel}}(t^*)<0$, so
\[
\Theta(t^*)=\Phi(t^*)=0,
\]
but we have
\[
\dot\Theta(t^*) =\dot\theta_{\mathrm{rel}}(t^*)\Big(\underbrace{\frac{N_2}{N},\cdots,\frac{N_2}{N}}_{N_1\mathrm{~times}},\underbrace{-\frac{N_1}{N},\cdots,-\frac{N_1}{N}}_{N_2\mathrm{~times}}\Big) \neq \dot\theta_{\mathrm{rel}}(t^*)\Big(-\underbrace{\frac{N_2}{N},\cdots,-\frac{N_2}{N}}_{N_1\mathrm{~times}},\underbrace{\frac{N_1}{N},\cdots,\frac{N_1}{N}}_{N_2\mathrm{~times}}\Big)=\dot\Phi(t^*).
\]
\end{proof}

%
%
%
%
\section{Concluding remarks} \label{sec:conclusion}
\setcounter{equation}{0}
In this paper, we have revisited the previous work \cite{C-D-H-R25} and investigated both the rationale and the limitations of the philosophy employed therein, which was to view system \eqref{A-1} as a perturbation of the first-order system \eqref{A-2}. There are two senses in which \eqref{A-1} can be viewed as a perturbation of \eqref{A-2}. The first sense is as a perturbation in the Fr\'echet function space $C[0,\infty)\cap C^\infty(0,\infty)$, for which we develop a Tikhonov theory both quantitatively and qualitatively, namely Propositions \ref{prop:mto0} and \ref{prop:mto0quant}, respectively. The second sense is in the determinability of the first derivative $\dot\Theta(t)$ from the phase position $\Theta(t)$ and the initial derivative $\dot\Theta^0$, for which we develop a sharp temporal threshold for determinability in Proposition \ref{prop:determinability}, as well as an explicit Banach contraction principle in Proposition \ref{prop:banach-contraction}.
 
The following are some possible directions for future research.

\subsection{Lyapunov functional approach}
Improvements in the perturbative estimations may enhance the results of \cite{C-D-H-R25} in the small inertia regime $m\kappa\le \frac 14$. However, our results in this paper reveal a fundamental limitation in both the two viewpoints on considering \eqref{A-1} as a perturbation of \eqref{A-2}, namely that they are applicable only in the small inertia regime. In order to fully resolve the complete synchronization problem for the inertial Kuramoto model \eqref{A-1} in all inertial regimes, namely to affirmatively answer Question \ref{ques:suff-coupling} for large inertia, or, more ambitiously, to affirmatively resolve the inertia-oblivious Conjectures \ref{conj:coincide} and \ref{conj:R}, one would need to abandon the perturbative approach and use inertia-independent strategies, most promisingly through the following Lyapunov functional existence conjecture:
\begin{conjecture}\label{conj:lyapunov}\,
	\begin{enumerate}
		\item (Weak form) There is a constant $c\ge \frac 12$ such that if $\kappa>c\mathcal{D}(\mathcal{V})$, then \eqref{A-1} and \eqref{A-2} admit weak Lyapunov functionals.
		\vspace{0.1cm}
		\item (Strong form) Denoting the critical coupling strength\footnote{This is the coupling strength above which phase-locked states exist. It is given as a function of $\mathcal{V}$ \cite{V-M}.} as $\kappa_c(\mathcal{V})$, if $\kappa>\kappa_c(\mathcal{V})$, then \eqref{A-1} and \eqref{A-2} admit weak Lyapunov functionals.
	\end{enumerate}
\end{conjecture}

 Constructing a Lyapunov functional for systems \eqref{A-1} and \eqref{A-2} remains a challenging open problem which we leave to future investigation.

\subsection{Analogous results for the Winfree model}
Another avenue for future investigation would be to establish analogous results for the Winfree model, another prototypical synchronization model mentioned in the introduction. The Winfree model is given as the following Cauchy problem
\begin{align}\label{Wi-1}
	\begin{cases}
		\displaystyle \dot\theta_i = \nu_i - \frac{\kappa}{N}\sin\theta_i\sum_{j=1}^N (1+\cos\theta_j),\quad t > 0,\\
		\displaystyle \theta_i\Big|_{t = 0+} = \theta_i^0,\quad i\in [N],
	\end{cases}
\end{align}
and a natural \emph{inertial Winfree model} would be given as
\begin{equation}
	\begin{cases} \label{Wi-2}
		\displaystyle m \ddot\theta_i + \dot\theta_i = \nu_i - \frac{\kappa}{N}\sin\theta_i\sum_{j=1}^N (1+\cos\theta_j),\quad t > 0,\\
		\displaystyle (\theta_i, {\dot \theta}_i) \Big|_{t = 0+} = (\theta_i^0, \omega_i^0), \quad i\in [N].
	\end{cases}
\end{equation}

\begin{question}
	What are the analogues of Theorem \ref{thm:qualitative_mto0} and Propositions \ref{prop:banach-contraction}, \ref{prop:determinability}, \ref{prop:mto0}, and \ref{prop:mto0quant} for the inertial Winfree model \eqref{Wi-2}?
\end{question}
We remark that the strategy of \cite{H-R} was successfully implemented in \cite{Ry} so as to obtain the following theorem. For the Winfree model, the order parameter is defined as
\[
R=\frac 1N\sum_{j=1}^N(1+\cos\theta_j).
\]

\begin{theorem}[{\cite{Ry}}]
	Fix parameters $\mathcal{V}\in \mathbb{R}^N$ and $\kappa>0$ such that
	\[
	\kappa>2\|\mathcal{V}\|_\infty.
	\]
	Then for Lebesgue almost every initial data $\{\theta_i^0\}_{i=1}^N$, the solution $\{\theta_i(t)\}_{i=1}^N$ to \eqref{Wi-1} satisfies the following properties.
	\begin{enumerate}
		\item (Oscillator death) For all $i=1,\cdots,N$, the limit $\lim_{t\to\infty}\theta_i(t)$ exists and $\lim_{t\to\infty}\dot{\theta}_i(t)=0$.
		\item (Uniform lower bound on the order parameter) We have $R_\infty\coloneqq \lim_{t\to\infty}R(t)\ge \sqrt{\frac 12}$.
	\end{enumerate}
\end{theorem}
It seems promising to investigate whether similar results hold for the inertial Winfree model \eqref{Wi-2}.

\bigskip

\appendix
%
%
%
%
\section{An integral Gr\"onwall inequality}\label{app:sturm-picone}
\setcounter{equation}{0} 
We state and prove the non-negativity of function satisfying the integral form of  Gr\"onwall-type inequality as follows.
\begin{lemma}\label{lem:2nd-Gronwall}
    Let $a,c>0$, and suppose the continuous function $u:[0,\infty)\to\mathbb{R}$ satisfies
    \begin{equation}\label{eq:2nd-Gronwall}
        u(t)\le c\int_0^t u(s)(1-e^{-(t-s)/a})ds,\quad t\ge 0.
    \end{equation}
    Then, we have
   \[  u(t)\le 0 \quad \mbox{for all $t\ge 0$.} \]
\end{lemma}
\begin{proof}
    Define the nonnegative function $U(t):[0,\infty)\to [0,\infty)$ as
    \begin{equation} \label{App-B-2}
    U(t)=\max\left\{\max_{s\in [0,t]}u(s),0\right\}.
    \end{equation}
    Then, it suffices to check that 
    \[  U\equiv 0. \]
\noindent $\bullet$~Step A: we claim that 
\[ U(0) = 0. \]
It follows from \eqref{eq:2nd-Gronwall} at $t=0$ that 
\begin{equation} \label{App-B-3}
  u(0) \leq 0.
\end{equation}   
Then, we use \eqref{App-B-2} and \eqref{App-B-3} to see 
\begin{equation} \label{App-B-3-1}
U(0) = \max\left\{ u(0),~0\right\} = 0.  
\end{equation}
\noindent $\bullet$~Step B: It follows from the definition of $U$ that $U$ is continuous, monotonically nondecreasing, and 
\begin{equation} \label{App-B-4}
U\ge u.
\end{equation}
Now, we use $0\le 1-e^{-(t-s)/a}\le 1$, \eqref{eq:2nd-Gronwall} and  \eqref{App-B-4} to get 
    \[
    u(t)\stackrel{\mathclap{\eqref{eq:2nd-Gronwall}}}{\le} c\int_0^t u(s)(1-e^{-(t-s)/a})ds\stackrel{\mathclap{\eqref{App-B-4}}}{\le} c\int_0^t U(s)(1-e^{-(t-s)/a})ds\le c\int_0^t U(s)ds.
    \]
    For each $t\ge 0$, we may find $s^*\in [0,t]$ such that 
    \[ u(s^*)=\max_{s\in [0,t]}u(s). \]
    By the above inequality,
    \[
    \max_{s\in [0,t]}u(s)=u(s^*)\le c\int_0^{s^*}U(s)ds\le c\int_0^tU(s)ds.
    \]
    On the other hand, since $0\le c\int_0^tU(s)ds$,  one has by definition \eqref{App-B-2} of $U$,
    \[
    U(t)\le c\int_0^tU(s)ds,\quad t\ge 0.
    \]
    By the standard Gr\"onwall lemma and \eqref{App-B-3-1}, we have 
 \[ U(t)\le U(0)e^{ct}=0, \quad \mbox{hence}~~U\equiv 0. \]
\end{proof}

%
%
%
%
\section{Proof of the quantitative higher-order Tikhonov perturbation, Proposition \ref{prop:mto0quant}}\label{app:mto0}
\setcounter{equation}{0} 
We begin by noting that since
\begin{equation}\label{eq:mto0-zeroth}
\sqrt{1+8m\kappa}> 1,\quad \frac{-1+\sqrt{1+8m\kappa}}{2m}< 2\kappa,\quad \mathrm{and}\quad 1-e^{-t/m}<1,
\end{equation}
in order to show \eqref{eq:C0-approx-abs-vanilla} it is enough to show the stronger statement
\begin{align}\label{eq:C0-approx-abs}
\begin{aligned}
\|(\Theta(m,t)-\Theta(0,t))\|_\infty&\le m\left|\frac{\max_i(\omega_i^0-\nu_i)+\min_i(\omega_i^0-\nu_i)}2\right|(1-e^{-t/m})\\
&\quad +\frac{m(\mathcal{D}(\Omega^0-\mathcal{V})+2\kappa)}{2\sqrt{1+8m\kappa}}e^{-t/2m}\left(e^{\frac{\sqrt{1+8m\kappa}}{2m}t}-e^{-\frac{\sqrt{1+8m\kappa}}{2m}t}\right),
\end{aligned}
\end{align}
and in order to show \eqref{eq:C0-approx-rel-vanilla} it is enough to show the stronger statement
\begin{align}\label{eq:C0-approx-rel}
\begin{aligned}
\mathcal{D}(\Theta(m,t)-\Theta(0,t))&\le \frac{m(\mathcal{D}(\Omega^0-\mathcal{V})+2\kappa)}{\sqrt{1+8m\kappa}}e^{-t/2m}\left(e^{\frac{\sqrt{1+8m\kappa}}{2m}t}-e^{-\frac{\sqrt{1+8m\kappa}}{2m}t}\right).
\end{aligned}
\end{align}
Also, in order to show \eqref{eq:C1-approx-abs-vanilla}, it is enough to show the stronger statement
\begin{align}\label{eq:C1-approx-abs}
\begin{aligned}
\|(\dot\Theta(m,t)-\dot\Theta(0,t))\|_\infty&\le (\|\Omega^0-\mathcal{V}\|_\infty+\kappa)e^{-t/m}+ m\kappa(\mathcal{D}(\mathcal{V})+2\kappa)(1-e^{-t/m})\\
&\quad +\frac{2m\kappa(\mathcal{D}(\Omega^0-\mathcal{V})+2\kappa)}{\sqrt{1+8m\kappa}(1+\sqrt{1+8m\kappa})}\left(e^{\frac{-1+\sqrt{1+8m\kappa}}{2m}t}-1\right),
\end{aligned}
\end{align}
and in order to show \eqref{eq:C1-approx-rel-vanilla} it is enough to show the stronger statement
\begin{align}\label{eq:C1-approx-rel}
\begin{aligned}
\mathcal{D}(\dot\Theta(m,t)-\dot\Theta(0,t))&\le (\mathcal{D}(\Omega^0-\mathcal{V})+2\kappa)e^{-t/m}+ 2m\kappa(\mathcal{D}(\mathcal{V})+2\kappa)(1-e^{-t/m})\\
&\quad +\frac{4m\kappa(\mathcal{D}(\Omega^0-\mathcal{V})+2\kappa)}{\sqrt{1+8m\kappa}(1+\sqrt{1+8m\kappa})}e^{\frac{-1+\sqrt{1+8m\kappa}}{2m}t}.
\end{aligned}
\end{align}

Inequalities \eqref{eq:C0-approx-abs-vanilla}, \eqref{eq:C0-approx-rel-vanilla}, \eqref{eq:C1-approx-abs-vanilla}, and \eqref{eq:C1-approx-rel-vanilla} were written in that way to better visualize the convergence as $m\to 0$. Not much was lost in the simplified forms \eqref{eq:C0-approx-abs-vanilla}, \eqref{eq:C0-approx-rel-vanilla}, \eqref{eq:C1-approx-abs-vanilla}, and \eqref{eq:C1-approx-rel-vanilla} since the approximation of \eqref{eq:mto0-zeroth} is sharp up to  the zeroth order in the limit $m\to 0$. We now prove each statement of Proposition \ref{prop:mto0quant} in order. \newline

\subsection{Proof of the first statement} We must prove \eqref{eq:C0-approx-abs} and \eqref{eq:C0-approx-rel}. We derive the Duhamel principle \eqref{B-4} to see that for $i\in [N]$ and $t\ge 0$,
    \begin{align}\label{eq:0th-order-duhamel-inertia}
    \begin{aligned}
    & \theta_i(m,t)-\theta_i^0  = \int_0^t \omega_i(m,s)ds\\
    & \hspace{0.5cm} \stackrel{\mathclap{\eqref{B-4}}}{=}\int_0^t\left[\omega^0_i e^{-s/m} + \nu_i (1 -e^{-s/m}) + \frac{\kappa}{Nm}\sum_{l=1}^N \int_0^s e^{-(s-\tau)/m}\sin(\theta_l(m,\tau) -\theta_i(m,\tau))d\tau \right]ds\\
    & \hspace{0.5cm}  =m\omega_i^0(1-e^{-t/m})+\nu_i(t-m+me^{-t/m})\\
    & \hspace{0.7cm} +\frac\kappa N\sum_{l=1}^N\int_0^t\sin(\theta_l(m,s)-\theta_i(m,s))(1-e^{-(t-s)/m}) ds,
    \end{aligned}
    \end{align}
where in the final equality we used Fubini's theorem. On the other hand, integrating the first-order ODE of \eqref{A-2} directly gives
\begin{equation}\label{eq:0th-order-duhamel-sans-inertia}
    \theta_i(0,t)-\theta_i^0=\int_0^t\omega_i(0,s)ds=\nu_i t+\frac\kappa N\sum_{l=1}^N \int_0^t \sin(\theta_l(0,s)-\theta_i(0,s))ds,\quad i\in [N]
,~t\ge 0.
\end{equation}
We subtract \eqref{eq:0th-order-duhamel-sans-inertia} from \eqref{eq:0th-order-duhamel-inertia} that for $i\in [N]$ and $t\ge 0$,
\begin{align}\label{eq:0th-order-comparison}
\begin{aligned}
    &\theta_i(m,t)-\theta_i(0,t) =m(\omega_i^0-\nu_i)(1-e^{-t/m})\\
    &\quad+\frac\kappa N\sum_{l=1}^N\int_0^t(\sin(\theta_l(m,s)-\theta_i(m,s))-\sin(\theta_l(0,s)-\theta_i(0,s)))(1-e^{-(t-s)/m}) ds\\
    &\quad -\frac\kappa N\sum_{l=1}^N\int_0^t \sin(\theta_l(0,s)-\theta_i(0,s))e^{-(t-s)/m} ds.
\end{aligned}
\end{align}
Again, we subtract \eqref{eq:0th-order-comparison} for $j\in [N]$ from \eqref{eq:0th-order-comparison} for $i\in [N]$ to derive
\begin{align*}
    &\theta_i(m,t)-\theta_j(m,t)-\theta_i(0,t)+\theta_j(0,t)\\
    & \hspace{0.5cm} =m(\omega_i^0-\omega_j^0-\nu_i+\nu_j)(1-e^{-t/m})\\
    & \hspace{0.7cm} +\frac\kappa N\sum_{l=1}^N\int_0^t(\sin(\theta_l(m,s)-\theta_i(m,s))-\sin(\theta_l(0,s)-\theta_i(0,s)))(1-e^{-(t-s)/m}) ds\\
    &  \hspace{0.7cm} -\frac\kappa N\sum_{l=1}^N\int_0^t(\sin(\theta_l(m,s)-\theta_j(m,s))-\sin(\theta_l(0,s)-\theta_j(0,s)))(1-e^{-(t-s)/m}) ds\\
    &  \hspace{0.7cm} -\frac\kappa N\sum_{l=1}^N\int_0^t \sin(\theta_l(0,s)-\theta_i(0,s))e^{-(t-s)/m} ds\\
    &  \hspace{0.7cm} +\frac\kappa N\sum_{l=1}^N\int_0^t \sin(\theta_l(0,s)-\theta_j(0,s))e^{-(t-s)/m} ds.
\end{align*}
Now, we use 
\[
|\sin(\theta_l(m,s)-\theta_i(m,s))-\sin(\theta_l(0,s)-\theta_i(0,s))|\le \mathcal{D}(\Theta(m,s)-\Theta(0,s))
\]
and
\[
|\sin(\theta_l(0,s)-\theta_i(0,s))|\le 1
\]
to find that for $i,j\in [N]$ and $t\ge 0$,
\begin{align*}
\begin{aligned} 
 &|\theta_i(m,t)-\theta_j(m,t)-\theta_i(0,t)+\theta_j(0,t)|\\
 & \hspace{0.5cm} \le m(\mathcal{D}(\Omega^0-\mathcal{V})+2\kappa)(1-e^{-t/m}) +2\kappa \int_0^t \mathcal{D}(\Theta(m,s)-\Theta(0,s))(1-e^{-(t-s)/m}) ds.
\end{aligned}
\end{align*}
We take the maximum the above relation over all $i,j\in [N]$ to find 
\begin{align*}
\begin{aligned}
 &\mathcal{D}(\Theta(m,t)-\Theta(0,t))\\
 & \hspace{0.5cm} \le m(\mathcal{D}(\Omega^0-\mathcal{V})+2\kappa)(1-e^{-t/m})  +2\kappa \int_0^t \mathcal{D}(\Theta(m,s)-\Theta(0,s))(1-e^{-(t-s)/m}) ds.
\end{aligned}
\end{align*}
We claim that this implies \eqref{eq:C0-approx-rel}. Indeed, denoting the right-hand side of \eqref{eq:C0-approx-rel} by
\[
v(t)\coloneqq\frac{m(\mathcal{D}(\Omega^0-\mathcal{V})+2\kappa)}{\sqrt{1+8m\kappa}}e^{-t/2m}\left(e^{\frac{\sqrt{1+8m\kappa}}{2m}t}-e^{-\frac{\sqrt{1+8m\kappa}}{2m}t}\right),\quad t\ge 0,
\]
one may calculate that this satisfies
\begin{equation}\label{eq:2nd-Gronwall-solution}
v(t)=m(\mathcal{D}(\Omega^0-\mathcal{V})+2\kappa)(1-e^{-t/m})+2\kappa \int_0^t v(s)(1-e^{-(t-s)/m}) ds.
\end{equation}
Therefore, we set
\[
u(t)\coloneqq \mathcal{D}(\Theta(m,t)-\Theta(0,t))-v(t),\quad t\ge 0
\]
to find 
\[
u(t)\le 2\kappa \int_0^t u(s)(1-e^{-(t-s)/m}) ds.
\]
By Lemma \ref{lem:2nd-Gronwall}, we have $u(t)\le 0$ and 
\begin{equation}\label{eq:2nd-Gronwall-comparison}
    \mathcal{D}(\Theta(m,t)-\Theta(0,t))\le v(t),\quad t\ge 0.
\end{equation}
This is equivalent to \eqref{eq:C0-approx-rel}. We note from \eqref{eq:0th-order-comparison} that
\begin{align*}
    &|\theta_i(m,t)-\theta_i(0,t)|\\
    & \hspace{0.5cm} \le m|\omega_i^0-\nu_i|(1-e^{-t/m})
    +\kappa \int_0^t \mathcal{D}(\Theta(m,s)-\Theta(0,s))(1-e^{-(t-s)/m}) ds
    +\kappa \int_0^t e^{-(t-s)/m} ds\\
    &  \hspace{0.5cm} \stackrel{\mathclap{\eqref{eq:2nd-Gronwall-comparison}}}{\le} m(\|\Omega^0-\mathcal{V}\|_\infty+\kappa)(1-e^{-t/m})+\kappa \int_0^t v(s)(1-e^{-(t-s)/m}) ds\\
    & \hspace{0.5cm} \stackrel{\mathclap{\eqref{eq:2nd-Gronwall-solution}}}{=}~m(\|\Omega^0-\mathcal{V}\|_\infty+\kappa)(1-e^{-t/m})-m(\frac 12\mathcal{D}(\Omega^0-\mathcal{V})+\kappa)(1-e^{-t/m})+\frac 12 v(t)\\
    & \hspace{0.5cm} \stackrel{\mathclap{\eqref{eq:diameter-max-diff}}}{=}~ m\left|\frac{\max_j(\omega_j^0-\nu_j)+\min_j(\omega_j^0-\nu_j)}2\right|(1-e^{-t/m})+\frac 12 v(t), \quad i\in [N],~t\ge 0.
\end{align*}
This gives \eqref{eq:C0-approx-abs}. \newline

\subsection{Proof of the second statement} Below, we are ready to prove \eqref{eq:C1-approx-abs} and \eqref{eq:C1-approx-rel}.  \newline

\noindent $\bullet$ Case A (Verification of \eqref{eq:C1-approx-abs}):~First, we note that
\begin{align}\label{eq:1st-speed-rel}
\begin{aligned}
|\dot\theta_i(0,t)-\dot\theta_j(0,t)|&=\left|\nu_i-\nu_j+\frac{\kappa}{N}\sum_{l=1}^N(\sin(\theta_l(0,t)-\theta_i(0,t))-\sin(\theta_l(0,t)-\theta_j(0,t)))\right|\\
&\le \mathcal{D}(\mathcal{V})+2\kappa,\quad i,j\in [N],~t\ge 0.
\end{aligned}
\end{align}
Next, we differentiate the ODE of \eqref{A-2} to see
\begin{align}\label{eq:1st-acc}
\begin{aligned}
|\ddot\theta_i(0,t)|&=\left|\frac{\kappa}{N}\sum_{l=1}^N\cos(\theta_l(0,t)-\theta_i(0,t))(\dot\theta_j(0,t)-\dot\theta_i(0,t))\right|\le \kappa \max_{j\in [N]}|\dot\theta_j(0,t)-\dot\theta_i(0,t)|\\
&\le\kappa(\mathcal{D}(\mathcal{V})+2\kappa),\quad i\in [N],~t\ge 0.
\end{aligned}
\end{align}
It follows from \eqref{A-1} and \eqref{A-2} that
\begin{align*}
    &m(\ddot\theta_i(m,t)-\ddot\theta_i(0,t))+(\dot\theta_i(m,t)-\dot\theta_i(0,t))\\
    & \hspace{0.5cm} =\frac{\kappa}{N}\sum_{l=1}^N (\sin(\theta_l(m,t)-\theta_i(m,t))-\sin(\theta_l(0,t)-\theta_i(0,t)))-m\ddot\theta_i(0,t).
\end{align*}
By invoking Duhamel's principle as in \eqref{B-4} and its two preceding equations, we obtain
\begin{align}\label{eq:speed-duhamel-comparison}
\begin{aligned}
&\dot\theta_i(m,t)-\dot\theta_i(0,t)  \\
& \hspace{0.5cm} =(\omega_i^0-\dot\theta_i(0,0))e^{-t/m} \\
& \hspace{0.5cm} +\frac{\kappa}{Nm}\sum_{l=1}^N\int_0^t e^{-(t-s)/m}(\sin(\theta_l(m,s)-\theta_i(m,s))-\sin(\theta_l(0,s)-\theta_i(0,s)))ds\\
& \hspace{0.5cm} +\frac 1m \int_0^t e^{-(t-s)/m}(-m\ddot\theta_i(0,s))ds \\
& \hspace{0.5cm} =: {\mathcal I}_{11} e^{-t/m} + {\mathcal I}_{12} + {\mathcal I}_{13}, \quad   i\in [N],~t\ge 0.
\end{aligned}
\end{align}
Next, we estimate ${\mathcal I}_{1i},~i =1,2,3$ as follows: \newline

\noindent $\diamond$ Case A.1 (Estimate of ${\mathcal I}_{11}$):~By direct calculation, we have
\begin{equation}\label{eq:speed-duhamel-comparison-1}
    |\omega_i^0-\dot\theta_i(0,0)|=\left|\omega_i^0-\nu_i-\frac{\kappa}{N}\sum_{l=1}^N\sin(\theta^0_l-\theta_i^0)\right|\le \|\Omega^0-\mathcal{V}\|_\infty+\kappa.
\end{equation}
\noindent $\diamond$ Case A.2  (Estimate of ${\mathcal I}_{12}$): Note that 
\begin{align}\label{eq:speed-duhamel-comparison-2}
\begin{aligned}
    &\frac{\kappa}{Nm}\sum_{l=1}^N\int_0^t e^{-(t-s)/m}(\sin(\theta_l(m,s)-\theta_i(m,s))-\sin(\theta_l(0,s)-\theta_i(0,s)))ds\\
    & \hspace{0.5cm} \le \frac\kappa m \int_0^t e^{-(t-s)/m}\mathcal{D}(\Theta(m,s)-\Theta(0,s))ds\\
    & \hspace{0.5cm} \stackrel{\mathclap{\eqref{eq:C0-approx-rel}}}{\le}\frac\kappa m \int_0^t e^{-(t-s)/m}\frac{m(\mathcal{D}(\Omega^0-\mathcal{V})+2\kappa)}{\sqrt{1+8m\kappa}}e^{\frac{-1+\sqrt{1+8m\kappa}}{2m}s}ds\\
    &  \hspace{0.5cm} \le \frac{2m\kappa(\mathcal{D}(\Omega^0-\mathcal{V})+2\kappa)}{\sqrt{1+8m\kappa}(1+\sqrt{1+8m\kappa})}\left(e^{\frac{-1+\sqrt{1+8m\kappa}}{2m}t}-e^{-t/m}\right).
\end{aligned}
\end{align}
\noindent $\diamond$ Case A.3  (Estimate of ${\mathcal I}_{13}$): Note that 
\begin{align}\label{eq:speed-duhamel-comparison-3}
\begin{aligned}
    & \left|\frac 1m \int_0^t e^{-(t-s)/m}(-m\ddot\theta_i(0,s))ds\right| \le  \int_0^t e^{-(t-s)/m}|\ddot\theta_i(0,s))|ds \\
    & \hspace{1cm} \stackrel{\mathclap{\eqref{eq:1st-acc}}}{\le}\kappa (\mathcal{D}(\mathcal{V})+2\kappa)\int_0^t e^{-(t-s)/m}ds =m\kappa (\mathcal{D}(\mathcal{V})+2\kappa)(1-e^{-t/m}).
\end{aligned}
\end{align}
Finally, we substitute \eqref{eq:speed-duhamel-comparison-1}, \eqref{eq:speed-duhamel-comparison-2}, and \eqref{eq:speed-duhamel-comparison-3} into \eqref{eq:speed-duhamel-comparison} to find  \eqref{eq:C1-approx-abs}. \newline

\noindent $\bullet$ Case B (Verification of  \eqref{eq:C1-approx-rel}):~Given $i,j\in [N]$, we subtract \eqref{eq:speed-duhamel-comparison} for $j$ from \eqref{eq:speed-duhamel-comparison} for $i$ to derive
\begin{align}\label{eq:speed-duhamel-comparison-rel}
\begin{aligned}
    &\dot\theta_i(m,t)-\dot\theta_i(0,t)-\dot\theta_j(m,t)+\dot\theta_j(0,t)\\
    & \hspace{0.5cm} =(\omega_i^0-\omega_j^0-\dot\theta_i(0,0)+\dot\theta_j(0,0)) e^{-t/m}\\
    &  \hspace{0.7cm} +\frac{\kappa}{Nm}\sum_{l=1}^N\int_0^t e^{-(t-s)/m}(\sin(\theta_l(m,s)-\theta_i(m,s))-\sin(\theta_l(0,s)-\theta_i(0,s)))ds\\
    &  \hspace{0.7cm} -\frac{\kappa}{Nm}\sum_{l=1}^N\int_0^t e^{-(t-s)/m}(\sin(\theta_l(m,s)-\theta_j(m,s))-\sin(\theta_l(0,s)-\theta_j(0,s)))ds\\
    &  \hspace{0.7cm} +\frac 1m \int_0^t e^{-(t-s)/m}(-m\ddot\theta_i(0,s))ds +\frac 1m \int_0^t e^{-(t-s)/m}m\ddot\theta_j(0,s)ds \\
    & \hspace{0.7cm} = {\mathcal I}_{21}  e^{-t/m}  + {\mathcal I}_{22} + {\mathcal I}_{23} + {\mathcal I}_{24}. 
\end{aligned}
\end{align}
We estimate the first term $\mathcal{I}_{21}$ as
\begin{align}\label{eq:speed-duhamel-comparison-rel-1}
\begin{aligned}
    &|\omega_i^0-\omega_j^0-\dot\theta_i(0,0)+\dot\theta_j(0,0)|\\
    &=\left|\omega_i^0-\omega_j^0-\nu_i+\nu_j-\frac{\kappa}{N}\sum_{l=1}^N\sin(\theta^0_l-\theta_i^0)+\frac{\kappa}{N}\sum_{l=1}^N\sin(\theta^0_l-\theta_j^0)\right|\\
    &\le \mathcal{D}(\Omega^0-\mathcal{V})+2\kappa.
\end{aligned}
\end{align}
For other estimates ${\mathcal I}_{2\ell},~\ell=2,3,4$, we substitute inequalities \eqref{eq:speed-duhamel-comparison-rel-1}, \eqref{eq:speed-duhamel-comparison-2} for $i$, \eqref{eq:speed-duhamel-comparison-2} for $j$, \eqref{eq:speed-duhamel-comparison-3} for $i$, and \eqref{eq:speed-duhamel-comparison-3} for $j$ into \eqref{eq:speed-duhamel-comparison-rel} to derive the desired estimate  \eqref{eq:C1-approx-rel}.

\subsection{Proof of the third statement}
We begin with a simple lemma, which will be useful in applying Faa di Bruno's formula.
\begin{lemma}\label{lem:FdB-probmass}
For $n \in {\mathbb N}$ and $\alpha \in {\mathbb R}$, the following identity holds:
    \[
    \sum_{\substack{m_1,\ldots,m_n\in\mathbb{Z}_{\ge 0}\\ \sum_{\ell=1}^n \ell m_\ell =n}} \prod_{l=1}^n\frac{\alpha^{m_l}}{m_l!l^{m_l}} =\binom{n+\alpha-1}{n},
    \]
  where 
  \[ 
 \binom{n+\alpha-1}{n}=\frac{\alpha(\alpha+1)\cdots (\alpha+n-1)}{n!}.
\]  
     \end{lemma}
\begin{proof}
    For a real-analytic function $f:I\to \mathbb{R}$ defined in a neighborhood $I$ of $0$, we denote by $f(x)[x^n]$ the coefficient of $x^n$ in the power series expansion of $f$ at $x=0$. Then, we have
    \begin{align*}
        \binom{n+\alpha-1}{n}&=\left(\frac{1}{1-x}\right)^\alpha [x^n]=\exp\left(-\alpha\log(1-x)\right)[x^n]\\
        &=\exp\left(\alpha x+\frac{\alpha x^2}{2}+\cdots +\frac{\alpha x^n}{n}+\cdots\right)[x^n]\\
        &=\exp\left(\alpha x+\frac{\alpha x^2}{2}+\cdots+ \frac{\alpha x^n}{n}\right)[x^n]\\
        &=\left(\exp\left(\alpha x\right)\exp\left(\frac{\alpha x^2}{2}\right)\cdots\exp\left(\frac{\alpha x^n}{n}\right)\right)[x^n]\\
        &=\left(\left(\sum_{m_1\ge 0}\frac{\alpha^{m_1}x^{m_1}}{m_1!}\right)\left(\sum_{m_2\ge 0}\frac{\alpha^{m_2}x^{2m_2}}{m_2!2^{m_2}}\right)\cdots \left(\sum_{m_n\ge 0}\frac{\alpha^{m_n}x^{nm_n}}{m_n!n^{m_n}}\right)\right)[x^n]\\
        &=\sum_{\substack{m_1,\ldots,m_n\in\mathbb{Z}_{\ge 0}\\ \sum_{\ell=1}^n \ell m_\ell =n}} \prod_{l=1}^n\frac{\alpha^{m_l}}{m_l!l^{m_l}}.
    \end{align*}
\end{proof}
\begin{remark} 
Note that for $\alpha=1$ and $\alpha=2$, we have
    \[
    \sum_{\substack{m_1,\ldots,m_n\in\mathbb{Z}_{\ge 0}\\ \sum_{\ell=1}^n \ell m_\ell =n}} \prod_{l=1}^n\frac{1}{m_l!l^{m_l}} =1,\quad \sum_{\substack{m_1,\ldots,m_n\in\mathbb{Z}_{\ge 0}\\ \sum_{\ell=1}^n \ell m_\ell =n}} \prod_{l=1}^n\frac{2^{m_l}}{m_l!l^{m_l}} =n+1.
    \]

\end{remark}
To obtain the convergence statements, we will need bounds on the $n$-th time derivatives of $\theta_i(m,t)$ and $\theta_i(0,t)$. We first have the following.
\begin{lemma}\label{lem:1st-deriv-bound}
    For each integer $n> 1,  i\in [N],~t\ge 0$, the following relations hold. 
    \begin{align}
    \begin{aligned} \label{eq:1st-deriv-abs-bound}
 &  (i)~ |\theta_i^{(n)}(0,t)|\le 
    \kappa (n-1)! (\mathcal{D}(\mathcal{V})+2\kappa)^{n-1}.  \\
 & (ii)~|\theta_i^{(n)}(0,t)-\theta_j^{(n)}(0,t)|\le (n-1)! (\mathcal{D}(\mathcal{V})+2\kappa)^{n}.
    \end{aligned}
    \end{align}
\end{lemma}
\begin{proof} We use the induction for the verification of the relations \eqref{eq:1st-deriv-abs-bound}. \newline

\noindent $\bullet$~Step A (Initial step with $n = 1, 2$): 
    The case $n=1$ for $\eqref{eq:1st-deriv-abs-bound}_2$ is true by \eqref{eq:1st-speed-rel}. For each $n>1$, the statement $\eqref{eq:1st-deriv-abs-bound}_1$ implies $\eqref{eq:1st-deriv-abs-bound}_2$, since
    \[
    |\theta_i^{(n)}(0,t)-\theta_j^{(n)}(0,t)|\le 2\kappa (n-1)! (\mathcal{D}(\mathcal{V})+2\kappa)^{n-1}<(n-1)! (\mathcal{D}(\mathcal{V})+2\kappa)^{n}.
    \]
    Also, the case $n=2$ for $\eqref{eq:1st-deriv-abs-bound}_1$ is true by \eqref{eq:1st-acc}. Hence the case $n=2$ for $\eqref{eq:1st-deriv-abs-bound}_2$ is also true.
    
\vspace{0.2cm}

\noindent $\bullet$~Step B (Inductive step):~Suppose $\eqref{eq:1st-deriv-abs-bound}_1$ is true for all $2\le n\le k$, and $\eqref{eq:1st-deriv-abs-bound}_2$ is true for all $1\le n\le k$, where $k\ge 2$ is an integer. Now, differentiating \eqref{A-2} $k$ times, we obtain by Faa di Bruno's formula \cite{F} that
    \[
\theta^{(k+1)}_i(0,t)=\frac{\kappa}{N}\sum_{j=1}^N\sum_{\substack{m_1,\cdots,m_k\in\mathbb{Z}_{\ge 0}\\ \sum_{\ell=1}^k \ell m_\ell =k}} k! \sin^{(\sum_{l=1}^nm_l)}(\theta_j(0,t)-\theta_i(0,t)) \prod_{l=1}^k \frac{1}{m_l!}\left(\frac{ \theta_j^{(l)}(0,t) -\theta_i^{(l)}(0,t)}{l!}\right)^{m_l},
\]
and using $|\sin^{(\sum_l m_l)}|\le 1$ and $\eqref{eq:1st-deriv-abs-bound}_2$, we have
\begin{align*}
|\theta^{(k+1)}_i(0,t)|&\le \frac{k!\kappa}{N}\sum_{j=1}^N\sum_{\substack{m_1,\cdots,m_k\in\mathbb{Z}_{\ge 0}\\ \sum_{\ell=1}^k \ell m_\ell =k}} \prod_{l=1}^k \frac{1}{m_l!}\left(\frac{ |\theta_j^{(l)}(0,t) -\theta_i^{(l)}(0,t)|}{l!}\right)^{m_l}\\
&\stackrel{\mathclap{\eqref{eq:1st-deriv-abs-bound}_2}}{\le}~ k!\kappa\sum_{\substack{m_1,\cdots,m_k\in\mathbb{Z}_{\ge 0}\\ \sum_{\ell=1}^k \ell m_\ell =k}} \prod_{l=1}^k \frac{1}{m_l!}\left(\frac{(\mathcal{D}(\mathcal{V})+2\kappa)^l}{l}\right)^{m_l}\\
&=k!\kappa (\mathcal{D}(\mathcal{V})+2\kappa)^k\sum_{\substack{m_1,\ldots,m_k\in\mathbb{Z}_{\ge 0}\\ \sum_{\ell=1}^k \ell m_\ell =k}} \prod_{l=1}^k\frac{1}{m_l!l^{m_l}}\\
&\stackrel{\mathclap{\mathrm{Lemma~}\ref{lem:FdB-probmass}}}{=}\qquad k!\kappa (\mathcal{D}(\mathcal{V})+2\kappa)^k,
\end{align*}
closing the induction.
\end{proof}

\begin{lemma}\label{lem:2nd-deriv-initial-bound}
    For $m>0$, $i,j\in [N]$, and each integer $n\ge 1$, we have
\begin{align}
\begin{aligned} \label{eq:2nd-deriv-initial-power-bound}
& (i)~|\theta^{(n)}_i(m,0)-\theta^{(n)}_j(m,0)|\le 2(n-1)!\left(\kappa+\frac{\mathcal{D}(\Omega^0)+\mathcal{D}(\mathcal{V})}2+ \frac 9{16m}\right)^{n}. \\
& (ii)~|\theta^{(n)}_i(m,0)|\le (n-1)!\left(\kappa+\|\Omega^0\|_\infty+\|\mathcal{V}\|_\infty+ \frac 9{16m}\right)^{n}.
\end{aligned}
\end{align}
\end{lemma}
\begin{proof}  We verify the assertions in \eqref{eq:2nd-deriv-initial-power-bound} using the method of induction. \newline

\noindent  (i)~We set 
    \[
    x=\frac 1{m\kappa},\quad y=\frac{\mathcal{D}(\Omega^0)+\mathcal{D}(\mathcal{V})}{\kappa},\quad \lambda=1+\frac 9{16}x+\frac 12y.
    \]
    Then $\eqref{eq:2nd-deriv-initial-power-bound}_1$ is equivalent to
    \begin{equation}\label{eq:2nd-deriv-initial-power-bound-equiv}
        |\theta^{(n)}_i(m,0)-\theta^{(n)}_j(m,0)|\le 2(n-1)!\lambda^n\kappa^n,\quad n\ge 1.
    \end{equation}
  \noindent $\bullet$~Step A.1 (Initial step $(n = 1,2))$: ~For the case $n=1$, it follows from the initial condition that 
    \[
    |\dot\theta_i(m,0)-\dot\theta_j(m,0)|=|\omega_i^0-\omega_j^0|\le \mathcal{D}(\Omega^0)\le \kappa y\le 2\lambda\kappa.
    \]
    For the case $n=2$, we use  ODE of \eqref{A-1} to see
    \[
        m\ddot\theta_i(m,0)-m\ddot\theta_j(m,0)=-\omega_i^0+\omega_j^0+\nu_i-\nu_j+\frac\kappa N\sum_{l=1}^N \Big(\sin(\theta_l^0-\theta_i^0)-\sin(\theta_l^0-\theta_j^0) \Big ).
    \]
 This implies 
\begin{align*}
|\ddot\theta_i(m,0)-\ddot\theta_j(m,0)|&\le \frac 1m\left(2\kappa+\mathcal{D}(\Omega^0)+\mathcal{D}(\mathcal{V})\right)
        =2\kappa^2 x(1+\frac y2)\le 2\lambda^2\kappa^2.
\end{align*}
\vspace{0.2cm}

\noindent $\bullet$~Step A.2 (inductive step):~By induction, we assume that \eqref{eq:2nd-deriv-initial-power-bound-equiv} is true for $1\le n\le k+1$ and $k \geq 1$. We differentiate \eqref{A-1} $k$ times to obtain by Faa di Bruno's formula:
    \begin{align}\label{eq:2nd-FdB-zero}
    \begin{aligned}
&m\theta^{(k+2)}_i(m,0)-m\theta^{(k+2)}_j(m,0)+\theta^{(k+1)}_i(m,0)-\theta^{(k+1)}_j(m,0)\\
& \hspace{0.2cm} =\frac{\kappa}{N}\sum_{p=1}^N\sum_{\substack{m_1,\cdots,m_k\in\mathbb{Z}_{\ge 0}\\ \sum_{\ell=1}^k \ell m_\ell =k}} k! \sin^{(\sum_{l=1}^km_l)}(\theta_p(m,0)-\theta_i(m,0)) \prod_{l=1}^k \frac{1}{m_l!}\left(\frac{ \theta_p^{(l)}(m,0) -\theta_i^{(l)}(m,0)}{l!}\right)^{m_l}\\
&\hspace{0.4cm}~-\frac{\kappa}{N}\sum_{p=1}^N\sum_{\substack{m_1,\cdots,m_k\in\mathbb{Z}_{\ge 0}\\ \sum_{\ell=1}^k \ell m_\ell =k}} k! \sin^{(\sum_{l=1}^km_l)}(\theta_p(m,0)-\theta_j(m,0)) \prod_{l=1}^k \frac{1}{m_l!}\left(\frac{ \theta_p^{(l)}(m,0) -\theta_j^{(l)}(m,0)}{l!}\right)^{m_l}.
\end{aligned}
\end{align}
Now, we estimate using the induction hypothesis \eqref{eq:2nd-deriv-initial-power-bound-equiv} and $\sum_{l=1}^klm_l=k$ to find 
\[
    \prod_{l=1}^k \frac{1}{m_l!}\left(\frac{ |\theta_p^{(l)}(m,0) -\theta_i^{(l)}(m,0)|}{l!}\right)^{m_l}\le \prod_{l=1}^k \frac{1}{m_l!}\left(\frac{ 2\lambda^l\kappa^l}{l}\right)^{m_l}=\lambda^k\kappa^k\prod_{l=1}^k \frac{1}{m_l!}\frac{ 2^{m_l}}{l^{m_l}}.
\]
Likewise, we have
\[
    \prod_{l=1}^k \frac{1}{m_l!}\left(\frac{ |\theta_p^{(l)}(m,0) -\theta_j^{(l)}(m,0)|}{l!}\right)^{m_l}\le \lambda^k\kappa^k\prod_{l=1}^k \frac{1}{m_l!}\frac{ 2^{m_l}}{l^{m_l}}.
\]
Putting this back into \eqref{eq:2nd-FdB-zero}, along with the induction hypothesis \eqref{eq:2nd-deriv-initial-power-bound-equiv} for $k+1$ and $|\sin^{(\sum_l m_l) (\cdot)}|\le 1$, we have
\begin{align*}
&|\theta^{(k+2)}_i(m,0)-\theta^{(k+2)}_j(m,0)|\\
& \hspace{0.5cm} \le \frac 1m|\theta^{(k+1)}_i(m,0)-\theta^{(k+1)}_j(m,0)|+\frac{2k!\kappa}{mN}\sum_{p=1}^N\sum_{\substack{m_1,\cdots,m_k\in\mathbb{Z}_{\ge 0}\\ \sum_{\ell=1}^k \ell m_\ell =k}}  \lambda^k\kappa^k\prod_{l=1}^k \frac{1}{m_l!}\frac{ 2^{m_l}}{l^{m_l}}\\
&  \hspace{0.5cm} \stackrel{\mathclap{\mathrm{Lemma~}\ref{lem:FdB-probmass}}}{\le} \qquad\frac 2m k! \lambda^{k+1}\kappa^{k+1}+\frac2m (k+1)!\lambda^k\kappa^{k+1}= 2 k!  x\lambda^{k+1}\kappa^{k+2}+2 (k+1)! x\lambda^k\kappa^{k+2}\\
&  \hspace{0.5cm} =2(k+1)! x \lambda^k\kappa^{k+2}(\frac{\lambda}{k+1}+1)\le 2(k+1)! \lambda^{k}\kappa^{k+2} x(\frac \lambda 2 +1)\\
&  \hspace{0.5cm} \le 2(k+1)! \lambda^{k+2}\kappa^{k+2},
\end{align*}
where in the last inequality we used
\begin{align*}
x(\frac \lambda 2 +1)&=x(\frac 32+\frac{9x}{32}+\frac y4)=\frac {3x}2+\frac{xy}{4}+\frac{9x^2}{32}\le \frac {3x}2+\frac{xy}{4}+\frac{9x^2}{32}+(1-\frac{3x}{16})^2\\
&\le 1+\frac {9x}8+y+\frac {81x^2}{256}+\frac {9xy}{16}+\frac{y^2}{4}= (1+\frac {9x}{16}+\frac y2)^2=\lambda^2.
\end{align*}
This closes the induction and completes the proof of \eqref{eq:2nd-deriv-initial-power-bound-equiv} and equivalently $\eqref{eq:2nd-deriv-initial-power-bound}_1$. \newline

\noindent (ii)~Again, we verify $\eqref{eq:2nd-deriv-initial-power-bound}_2$ by induction. For this, we set 
    \[
    {\tilde y} :=\frac{2\|\Omega^0\|_\infty+2\|\mathcal{V}\|_\infty}{\kappa}\ge y, \quad  {\tilde \lambda} =1+\frac 9{16}x+\frac 12 {\tilde y} \ge \lambda,
    \]
where the inequality follows from $\mathcal{D}(A)\le 2\|A\|_\infty$.  \newline

\noindent Then $\eqref{eq:2nd-deriv-initial-power-bound}_2$ is equivalent to
    \begin{equation}\label{eq:2nd-deriv-abs-initial-power-bound-equiv}
        |\theta^{(n)}_i(m,0)-\theta^{(n)}_j(m,0)|\le (n-1)! {\tilde \lambda}^n\kappa^n,\quad n\ge 1.
    \end{equation}
\noindent $\bullet$~Step B.1 (Initial step $(n = 1, 2))$:~For the case $n=1$, it follows from the initial condition that 
    \[
    |\dot\theta_i(m,0)|=|\omega_i^0|\le \|\Omega^0\|_\infty \le \frac 12\kappa {\tilde y} \le {\tilde \lambda} \kappa.
    \]
For the case $n=2$, we use ODE of \eqref{A-1} to find 
    \[
        m\ddot\theta_i(m,0)=-\omega_i^0+\nu_i+\frac\kappa N\sum_{l=1}^N\sin(\theta_l(m,0)-\theta_i(m,0)).
    \]
 This yields
\begin{align*}
        |\ddot\theta_i(m,0)|&\le \frac 1m\left(\kappa+\|\Omega^0\|_\infty+\|\mathcal{V}\|_\infty)\right) =\kappa^2 x \Big (1+\frac {{\tilde y}}2 \Big )\le {\tilde \lambda}^2\kappa^2.
\end{align*}
\vspace{0.2cm}

\noindent $\bullet$~Step B.2 (Inductive step:~By induction, assume \eqref{eq:2nd-deriv-abs-initial-power-bound-equiv} is true for $1\le n\le k+1$ and $k\ge 1$. We differentiate \eqref{A-1} $k$ times to obtain by Faa di Bruno's formula:
\begin{align*}
    \begin{aligned}
&m\theta^{(k+2)}_i(m,0)+\theta^{(k+1)}_i(m,0)\\
& \hspace{0.5cm} =\frac{\kappa}{N}\sum_{p=1}^N\sum_{\substack{m_1,\cdots,m_k\in\mathbb{Z}_{\ge 0}\\ \sum_{\ell=1}^k \ell m_\ell =k}} k! \sin^{(\sum_{l=1}^km_l)}(\theta_p(m,0)-\theta_i(m,0)) \prod_{l=1}^k \frac{1}{m_l!}\left(\frac{ \theta_p^{(l)}(m,0) -\theta_i^{(l)}(m,0)}{l!}\right)^{m_l}.
\end{aligned}
\end{align*}
Then, we use this and \eqref{eq:2nd-deriv-abs-initial-power-bound-equiv} to close the induction:
\begin{align*}
\begin{aligned}
|\theta^{(k+2)}_i(m,0)| &\le \frac 1m|\theta^{(k+1)}_i(m,0)|+\frac{k!\kappa}{mN}\sum_{p=1}^N\sum_{\substack{m_1,\cdots,m_k\in\mathbb{Z}_{\ge 0}\\ \sum_{\ell=1}^k \ell m_\ell =k}}  {\tilde \lambda}^k\kappa^k\prod_{l=1}^k \frac{1}{m_l!}\frac{ 2^{m_l}}{l^{m_l}}\\
&\stackrel{\mathclap{\mathrm{Lemma~}\ref{lem:FdB-probmass}}}{\le}\qquad \frac 1m k! {\tilde \lambda}^{k+1}\kappa^{k+1}+\frac1m (k+1)!{\tilde \lambda}^k\kappa^{k+1}=  k!  x{\tilde \lambda}^{k+1}\kappa^{k+2}+ (k+1)! x{\tilde \lambda}^k\kappa^{k+2}\\
&=(k+1)! x {\tilde \lambda}^k\kappa^{k+2} \Big (\frac{{\tilde \lambda}}{k+1}+1 \Big )\le (k+1)! {\tilde \lambda}^{k}\kappa^{k+2} x \Big (\frac {{\tilde \lambda}} 2 +1 \Big)\\
&\le (k+1)! {\tilde \lambda}^{k+2}\kappa^{k+2}.
\end{aligned}
\end{align*}
\end{proof}
\begin{remark} In the sequel, we give several comments on the results of Lemma \eqref{lem:2nd-deriv-initial-bound}. 
\begin{enumerate}
    \item A slight inspection of the proof reveals that we have the following stronger estimate with better dependency on $m$:
    \begin{align*}
    &|\theta^{(n)}_i(m,0)-\theta^{(n)}_j(m,0)|\\
    & \hspace{0.5cm} \le
        2(n-1)!\left(\kappa+\frac{\mathcal{D}(\Omega^0)+\mathcal{D}(\mathcal{V})}2\right)\left(\frac{\mathcal{D}(\Omega^0)+\mathcal{D}(\mathcal{V})}2+\frac{ 1}{m}\right) \left(\kappa+\frac{\mathcal{D}(\Omega^0)+\mathcal{D}(\mathcal{V})}2+\frac 9{16m} \right)^{n-2}.
    \end{align*}
    Indeed, we define
\[
    \mu=\frac{(1+\frac y2)(x+\frac y2)}{(1+\frac 9{16}x+\frac y2)^2}=\frac{x+\frac y2+\frac{xy}{2}+\frac{y^2}{4}}{(1+\frac 9{16}x+\frac y2)^2}<1,
    \]
    and set the induction hypothesis:
    \[
    |\theta^{(n)}_i(m,0)-\theta^{(n)}_j(m,0)|\le 2(n-1)!\mu\lambda^n\kappa^n.
    \]
    One may verify this for $n=1$ and $n=2$. In Induction step, the only material modification is in the following estimate using $\sum_{l=1}^k m_l\ge 1$:
    \[
    \prod_{l=1}^k \frac{1}{m_l!}\left(\frac{ |\theta_p^{(l)}(m,0) -\theta_i^{(l)}(m,0)|}{l!}\right)^{m_l}\le \prod_{l=1}^k \frac{1}{m_l!}\left(\frac{ 2\mu\lambda^l\kappa^l}{l}\right)^{m_l}\le \mu\lambda^k\kappa^k\prod_{l=1}^k \frac{1}{m_l!}\frac{ 2^{m_l}}{l^{m_l}}.
    \]
    However, this estimate is also quite not sharp, since we are throwing away potentially very high powers of $\mu$. The estimates in the last two Lemmas are not intended to be optimal. \newline
    \vspace{0.2cm}
    \item A similar argument can be made to show that for all $t\ge 0$,
    \[
    |\theta^{(n)}_i(m,t)-\theta^{(n)}_j(m,t)|\le
        2(n-1)!\left(\kappa+\frac{\mathcal{D}(\Omega^0)+\mathcal{D}(\mathcal{V})}2+\frac 1m\right)^{n}.
    \]
    The proof proceeds the same, but with the weaker base case for $n=1$. Again, with the $\mu$-trick, one may obtain the stronger bound
    \begin{align*}
    &|\theta^{(n)}_i(m,t)-\theta^{(n)}_j(m,t)|\\
    & \hspace{0.5cm} \le
        2(n-1)!\left(\kappa+\frac{\mathcal{D}(\Omega^0)+\mathcal{D}(\mathcal{V})}2\right)\left(\kappa+\frac{\mathcal{D}(\Omega^0)+\mathcal{D}(\mathcal{V})}2+\frac 2m\right) \left(\kappa+\frac{\mathcal{D}(\Omega^0)+\mathcal{D}(\mathcal{V})}2+\frac 1m\right)^{n-2}.
    \end{align*}
    Although this type of bound, with dominant term $\frac 1{m^{n-1}}$ as $m\to 0$, is true for all $t\ge 0$, it will be most useful at $t=0$, since for larger $t$ the term $\frac 1{m^{n-1}}$ will decay with a factor of $e^{-t/m}$ by the Duhamel formula, as will be seen later in Lemma \ref{lem:2nd-deriv-bound}. \newline
\vspace{0.2cm}
\item
Lemmas \ref{lem:1st-deriv-bound} and \ref{lem:2nd-deriv-bound} tell us that $\theta_i(0,t)$, $\theta_i(m,t)$ have uniform lower bounds on their radius of convergence and hence admit holomorphic extensions as functions of $t$, namely $\theta_i(0,t)$ may be considered a holomorphic function on $\{t:\Re(t)>0,~|\Im(t)|<\frac{1}{2\kappa+\mathcal{D}(\mathcal{V})}\}$,\footnote{A careful complex analytic examination tells us that we may take the slightly larger domain $|\Im (t)|<\frac{1}{\mathcal{D}(\mathcal{V})}\log\left(1+\frac{\mathcal{D}(\mathcal{V})}{2\kappa}\right)$ if $\mathcal{D}(\mathcal{V})>0$.} and $\theta_i(m,t)$ may be considered a holomorphic function on $\{t:\Re(t)>0,~|\Im(t)|<\frac{2}{2\kappa+\mathcal{D}(\mathcal{V})+\mathcal{D}(\Omega^0)+(9/8m)}\}$. Although the latter bound is not optimal, it is generically necessary for $\theta^{(n)}_i(m,t)$ to grow like $m^{1-n}$, so the maximal strip to which $\theta_i(m,t)$ may be extended would have height on the order of $O(m)$ as $m\to 0$. We find this discrepancy curious: after all, $\theta_i(m,t)$ should converge to $\theta_i(0,t)$, but why does the domain of $\theta_i(m,t)$ contract onto the positive real line? \newline

By \cite[Section 2, p.26]{V63}, $\theta_i(m,t)$ admits a Maclaurin expansion in terms of $m$ for a fixed positive integer $M$:
\[
\theta_i(m,t)=\theta_i(0,t)+\sum_{n=1}^M \frac{m^n}{n!} \theta^n_{i}(t)+R^{M+1}_i(m,t),
\]
with $R^{M+1}_i(m,t)=O(m^{M+1})$ as $m\to 0$ uniformly on compact subsets of the $t$-interval $(0,\infty)$, but we do not necessarily have $\lim_{M\to\infty}R^{M+1}_i(m,t)=0$. Here, $\theta^n_i(t)$ is given as the solution of the ODE obtained by differentiating $n$-times in $m$ the ODE of \eqref{A-1} and then substituting $m=0$:
\begin{align*}
&n\frac{\partial^2}{\partial t^2}\theta^{n-1}_i(t)+\frac{\partial}{\partial t}\theta^n_i(t)\\
& \hspace{0.4cm} =\frac{\kappa n!}{N}\sum_{p=1}^N\sum_{\substack{m_1,\cdots,m_{n-1}\in\mathbb{Z}_{\ge 0}\\ \sum_{\ell=1}^{n-1} \ell m_\ell =n}}  \sin^{(\sum_{\ell=1}^{n-1}m_\ell)}(\theta_p(0,t)-\theta_i(0,t)) \prod_{l=1}^{n-1} \frac{1}{m_l!}\left(\frac{ \theta_p^{l}(0,t) -\theta_i^{l}(0,t)}{l!}\right)^{m_l}\\
& \hspace{0.4cm} +\frac{\kappa }{N}\sum_{p=1}^N\cos(\theta_p(0,t)-\theta_i(0,t)) \left(\theta_p^{n}(t) -\theta_i^{n}(t)\right),
\end{align*}
with initial condition $\theta_i^n(0)=0$. Inductively, this is a linear ODE in $\theta_i^n(t)$, and hence $\theta_i^n(t)$ has the same $t$-domain as $\theta_i(0,t)$. It seems that the proof methods of Lemmas \ref{lem:1st-deriv-bound}, \ref{lem:2nd-deriv-initial-bound}, and \ref{lem:2nd-deriv-bound} should give
\[
\left|\frac{\partial^l}{\partial t^l}\theta^n_i(t)\right|\le c(2\kappa+\mathcal{D}(\mathcal{V}))^{n+l}e^{2\kappa t},\quad t\ge 0,
\]
for a constant $c>0$, but we have not checked the details. If so, the infinite series
\[
\tilde\theta_i(m,t)=\theta_i(0,t)+\sum_{n=1}^\infty \frac{m^n}{n!} \theta^n_{i}(t)
\]
would converge for $m<\frac{1}{2\kappa+\mathcal{D}(\mathcal{V})}$, and may even be well-defined on a $t$-half strip:
\[ \{t:\Re(t)>0,~|\Im(t)|<c'(m)\}, \]
where $c'(m)<\frac{1}{2\kappa+\mathcal{D}(\mathcal{V})}$ depends on $m$, with $\lim_{m\to 0} c'(m)=\frac{1}{2\kappa+\mathcal{D}(\mathcal{V})}$. This $\tilde\theta_i$ is the solution to \eqref{A-1} with
\[
\tilde\omega_i^0=\nu_i+\frac\kappa N\sum_{j=1}^N\sin(\theta_j^0-\theta_i^0).
\]

We may as well write
\[
\theta_i(m,t)=\tilde\theta_i(m,t)+R(m,t).
\]
Our analysis suggests the domain of $\theta_i(m,t)$ shrinks to the positive real line as $m\to0$ because the same is true for $R(m,t)$. Considering the decay term $e^{-t/m}$ in the estimates for $\theta_i^{(n)}(m,t)$, we ask whether $R(m,t)$ admits an analytic expansion in terms of $e^{-\alpha t/m}$ and $e^{-\bar{\alpha} t/m}$, with $\Re\alpha=1$. If so, this would explain the $O(m)$-width of the holomorphic strip domain of $\theta_i(m,t)$. Answering the questions posed in this remark seems an interesting future research direction.
\end{enumerate}
\end{remark}
\vspace{0.2cm}

Next, we will obtain better bounds on $|\theta_i^{{n}}(m,t)-\theta_j^{(n)}(m,t)|$, with terms that are polynomial in $\frac 1m$ decaying in time as $e^{-t/m}$. But first, we need a lemma.
\begin{lemma}\label{lem:weighted-powers}
    Let $p\in[0,1]$, $a,b\ge 0$, $n,m\in \mathbb{Z}_{\ge 0}$. Then, we have
    \[
    (pa^n+(1-p)b^n)(pa^m+(1-p)b^m)\le (pa^{n+m}+(1-p)b^{n+m}).
    \]
\end{lemma}
\begin{proof} We subtract the left-hand side from the right-hand side to find 
    \[
    RHS-LHS=p(1-p)(a^n-b^n)(a^m-b^m)\ge 0.
    \]
\end{proof}

\begin{lemma}\label{lem:2nd-deriv-bound}
For $m>0$, $i,j\in [N]$, $t\ge 0$, and each integer $n\ge 1$, we have
    \begin{align*}
    &|\theta_i^{(n)}(m,t)-\theta_j^{(n)}(m,t)|\\
    & \hspace{0.5cm} \le (n-1)!\left(\Big(2\kappa+\mathcal{D}(\Omega^0)+\mathcal{D}(\mathcal{V})+\frac{9}{8m}\Big)^n\Big(1+\frac tm\Big)^n e^{-t/m}+(2\kappa+\mathcal{D}(\mathcal{V}))^n(1-e^{-t/m})\right).
    \end{align*}
\end{lemma}
\begin{proof} We prove by induction. The base case $n=1$ follows from Lemma \ref{L2.2}. Assume this is true for $1\le n\le k$. Then by Lemma \ref{lem:weighted-powers},
\begin{align*}
&\prod_{l=1}^k \frac{1}{m_l!}\left(\frac{ |\theta_p^{(l)}(m,t) -\theta_i^{(l)}(m,t)|}{l!}\right)^{m_l}\\
& \hspace{0.2cm} \le \left((2\kappa+\mathcal{D}(\Omega^0)+\mathcal{D}(\mathcal{V})+\frac{9}{8m})^k(1+\frac tm)^k e^{-t/m}+(2\kappa+\mathcal{D}(\mathcal{V}))^k(1-e^{-t/m})\right)\prod_{l=1}^k\frac{1}{m_l!l^{m_l}}.
\end{align*}
Then, it follows from Faa di Bruno's formula that 
\begin{align*}
&m\theta^{(k+2)}_i(m,t)-m\theta^{(k+2)}_j(m,t)+\theta^{(k+1)}_i(m,t)-\theta^{(k+1)}_j(m,t)\\
& \hspace{0.5cm} =\frac{\kappa}{N}\sum_{p=1}^N\sum_{\substack{m_1,\cdots,m_k\in\mathbb{Z}_{\ge 0}\\ \sum_{\ell=1}^k \ell m_\ell =k}} k! \sin^{(\sum_{l=1}^km_l)}(\theta_p(m,t)-\theta_i(m,t)) \prod_{l=1}^k \frac{1}{m_l!}\left(\frac{ \theta_p^{(l)}(m,t) -\theta_i^{(l)}(m,t)}{l!}\right)^{m_l}\\
&\hspace{0.7cm} -\frac{\kappa}{N}\sum_{p=1}^N\sum_{\substack{m_1,\cdots,m_k\in\mathbb{Z}_{\ge 0}\\ \sum_{\ell=1}^k \ell m_\ell =k}} k! \sin^{(\sum_{l=1}^km_l)}(\theta_p(m,t)-\theta_j(m,t)) \prod_{l=1}^k \frac{1}{m_l!}\left(\frac{ \theta_p^{(l)}(m,t) -\theta_j^{(l)}(m,t)}{l!}\right)^{m_l}.
\end{align*}
We use Lemma \ref{lem:FdB-probmass} to see
\begin{align*}
    &|m\theta^{(k+2)}_i(m,t)-m\theta^{(k+2)}_j(m,t)+\theta^{(k+1)}_i(m,t)-\theta^{(k+1)}_j(m,t)|\\
    & \hspace{0.5cm}  \le 2\kappa k!\left((2\kappa+\mathcal{D}(\Omega^0)+\mathcal{D}(\mathcal{V})+\frac{9}{8m})^k(1+\frac tm)^k e^{-t/m}+(2\kappa+\mathcal{D}(\mathcal{V}))^k(1-e^{-t/m})\right)\\
    &  \hspace{0.5cm}  \le k!\left((2\kappa+\mathcal{D}(\Omega^0)+\mathcal{D}(\mathcal{V})+\frac{9}{8m})^{k+1}(1+\frac tm)^k e^{-t/m}+(2\kappa+\mathcal{D}(\mathcal{V}))^{k+1}\right).
\end{align*}
Finally, we invoke the Duhamel principle and Lemma \ref{lem:2nd-deriv-initial-bound} to see
\begin{align*}
    &|\theta^{(k+1)}_i(m,t)-\theta^{(k+1)}_j(m,t)|\\
    &  \hspace{0.5cm}  \le |\theta^{(k+1)}_i(m,0)-\theta^{(k+1)}_j(m,0)|e^{-t/m}\\
    &  \hspace{0.7cm} +\frac 1m \int_0^tk!\left((2\kappa+\mathcal{D}(\Omega^0)+\mathcal{D}(\mathcal{V})+\frac{9}{8m})^{k+1}(1+\frac sm)^k e^{-s/m}+(2\kappa+\mathcal{D}(\mathcal{V}))^{k+1}\right)e^{-(t-s)/m}ds\\
    &  \hspace{0.5cm} \le 2^{-k}k!(2\kappa+\mathcal{D}(\Omega^0)+\mathcal{D}(\mathcal{V})+\frac{9}{8m})^{k+1}e^{-t/m}\\
    &  \hspace{0.7cm}  +k!(2\kappa+\mathcal{D}(\Omega^0)+\mathcal{D}(\mathcal{V})+\frac{9}{8m})^{k+1}e^{-t/m}\cdot \frac{1}{k+1}\left((1+\frac tm)^{k+1}-1\right)\\
    &  \hspace{0.7cm} + k!(2\kappa+\mathcal{D}(\mathcal{V}))^{k+1}(1-e^{-t/m})\\
    &  \hspace{0.5cm} \le k! \Big (2\kappa+\mathcal{D}(\Omega^0)+\mathcal{D}(\mathcal{V})+\frac{9}{8m} \Big)^{k+1} \Big (1+\frac tm \Big)^{k+1}e^{-t/m}+k!(2\kappa+\mathcal{D}(\mathcal{V}))^{k+1}(1-e^{-t/m}),
\end{align*}
which close the induction. 
\end{proof}
We combine Lemma \ref{lem:1st-deriv-bound} and Lemma \ref{lem:2nd-deriv-bound} to see that for all $i,j\in [N]$, $n\in \mathbb{Z}_{>0}$, $m>0$, and $t\ge 0$,
    \begin{align}\label{sum-lem:deriv-bound}
    \begin{aligned}
    &|\theta_i^{(n)}(0,t)-\theta_j^{(n)}(0,t)|,~|\theta_i^{(n)}(m,t)-\theta_j^{(n)}(m,t)|\\
    & \hspace{0.2cm} \le (n-1)!\left( \Big (2\kappa+\mathcal{D}(\Omega^0)+\mathcal{D}(\mathcal{V})+\frac{9}{8m} \Big )^n \Big ( 1+\frac tm \Big)^n e^{-t/m}+(2\kappa+\mathcal{D}(\mathcal{V}))^n(1-e^{-t/m})\right)\\
    & \hspace{0.2cm} \le  (n-1)!\Big(2\kappa+\mathcal{D}(\Omega^0)+\mathcal{D}(\mathcal{V})+\frac{9}{8m} \Big )^n \Big (1+\frac tm \Big )^n.
    \end{aligned}
    \end{align}
    We are to show that
    \begin{align*}
    &|\theta_i^{(n)}(0,t)-\theta_j^{(n)}(0,t)-\theta_i^{(n)}(m,t)+\theta_j^{(n)}(m,t)|\\
    & \hspace{0.5cm} \le 2(n-1)!(2\kappa+\mathcal{D}(\Omega^0)+\mathcal{D}(\mathcal{V})+\frac{9}{8m})^n(1+\frac tm)^n e^{-t/m}+\frac 32(n+1)!m\kappa e^{2\kappa t}\\
    & \hspace{0.5cm} \times  \left((2\kappa+\mathcal{D}(\Omega^0)+\mathcal{D}(\mathcal{V})+\frac{9}{8m})^n(1+\frac tm)^n e^{-t/m}+(2\kappa+\mathcal{D}(\Omega^0)+\mathcal{D}(\mathcal{V}))^n(1-e^{-t/m})\right).
    \end{align*}
    By Faa di Bruno's formula, we have
\begin{align}\label{eq:FdB-gigantic}
    \begin{aligned}
&m\theta^{(k+2)}_i(m,t)-m\theta^{(k+2)}_j(m,t) +\theta^{(k+1)}_i(m,t)-\theta^{(k+1)}_j(m,t)-\theta^{(k+1)}_i(0,t)+\theta^{(k+1)}_j(0,t)\\
&=\frac{\kappa}{N}\sum_{p=1}^N\sum_{\substack{m_1,\cdots,m_k\in\mathbb{Z}_{\ge 0}\\ \sum_{\ell=1}^k \ell m_\ell =k}} k! \sin^{(\sum_{l=1}^km_l)}(\theta_p(m,t)-\theta_i(m,t)) \prod_{l=1}^k \frac{1}{m_l!}\left(\frac{ \theta_p^{(l)}(m,t) -\theta_i^{(l)}(m,t)}{l!}\right)^{m_l}\\
&-\frac{\kappa}{N}\sum_{p=1}^N\sum_{\substack{m_1,\cdots,m_k\in\mathbb{Z}_{\ge 0}\\ \sum_{\ell=1}^k \ell m_\ell =k}} k! \sin^{(\sum_{l=1}^km_l)}(\theta_p(0,t)-\theta_i(0,t)) \prod_{l=1}^k \frac{1}{m_l!}\left(\frac{ \theta_p^{(l)}(0,t) -\theta_i^{(l)}(0,t)}{l!}\right)^{m_l}\\
&-\frac{\kappa}{N}\sum_{p=1}^N\sum_{\substack{m_1,\cdots,m_k\in\mathbb{Z}_{\ge 0}\\ \sum_{\ell=1}^k \ell m_\ell =k}} k! \sin^{(\sum_{l=1}^km_l)}(\theta_p(m,t)-\theta_j(m,t)) \prod_{l=1}^k \frac{1}{m_l!}\left(\frac{ \theta_p^{(l)}(m,t) -\theta_j^{(l)}(m,t)}{l!}\right)^{m_l}\\
&+\frac{\kappa}{N}\sum_{p=1}^N\sum_{\substack{m_1,\cdots,m_k\in\mathbb{Z}_{\ge 0}\\ \sum_{\ell=1}^k \ell m_\ell =k}} k! \sin^{(\sum_{l=1}^km_l)}(\theta_p(0,t)-\theta_j(0,t)) \prod_{l=1}^k \frac{1}{m_l!}\left(\frac{ \theta_p^{(l)}(0,t) -\theta_j^{(l)}(0,t)}{l!}\right)^{m_l}.
\end{aligned}
\end{align}
\begin{lemma}\label{lem:product-diff}
If $a_\alpha ,b_\alpha ,c_\alpha ,d_\alpha ,\varepsilon_\alpha ,\delta_\alpha \ge 0$, $\alpha \in [n]$, satisfy
\[
\max\{a_\alpha ,b_\alpha \}\le \min\{c_\alpha ,d_\alpha \}, \quad |a_\alpha -b_\alpha |\le \varepsilon_\alpha  c_\alpha  +\delta_\alpha  d_\alpha ,
\]
then we have
\[
\left|a_1\cdots a_n-b_1\cdots b_n\right|\le \left(\sum_{i=1}^n\varepsilon_i\right)c_1\cdots c_n+\left(\sum_{i=1}^n\delta_i\right)d_1\cdots d_n.
\]
\end{lemma}
\begin{proof} By direct calculation, we have
    \begin{align*}
        &\left|a_1\cdots a_n-b_1\cdots b_n\right|\\
        & \hspace{0.5cm} \le \sum_{\alpha =1}^n a_1\cdots a_{\alpha -1}|a_\alpha-b_\alpha|b_{\alpha+1}\cdots b_n\\
        & \hspace{0.5cm} \le \sum_{\alpha=1}^n a_1\cdots a_{\alpha-1}(\varepsilon_\alpha c_\alpha)b_{\alpha+1}\cdots b_n+\sum_{\alpha=1}^n a_1\cdots a_{\alpha-1}(\delta_\alpha d_\alpha)b_{\alpha+1}\cdots b_n\\
        &\hspace{0.5cm} \le \left(\sum_{\alpha=1}^n\varepsilon_\alpha\right)c_1\cdots c_n+\left(\sum_{\alpha=1}^n\delta_\alpha\right)d_1\cdots d_n.
    \end{align*}
\end{proof}
We may set 
\[
n=\sum_{l=1}^k m_l, \quad \left(a_\alpha\right)_{\alpha=1}^n=\left(\frac{\theta_p^{(l)}(m,t) -\theta_i^{(l)}(m,t)}{(l-1)!}\right),\quad \left(b_\alpha\right)_{\alpha=1}^n=\left(\frac{\theta_p^{(l)}(0,t) -\theta_i^{(l)}(0,t)}{(l-1)!}\right),
\]
where each index $l$ is repeated $m_l$ times. By \eqref{sum-lem:deriv-bound}, we have
\begin{align*}
\begin{aligned}
& \left(c_\alpha\right)_{\alpha=1}^n=\left((2\kappa+\mathcal{D}(\Omega^0)+\mathcal{D}(\mathcal{V})+\frac{9}{8m})^l(1+\frac tm)^l\right), \\
& \left(d_\alpha\right)_{\alpha=1}^n=\left((2\kappa+\mathcal{D}(\Omega^0)+\mathcal{D}(\mathcal{V})+\frac{9}{8m})^l(1+\frac tm)^l e^{-t/m}+(2\kappa+\mathcal{D}(\mathcal{V}))^l(1-e^{-t/m})\right), \\
& \left(\varepsilon_\alpha\right)_{\alpha=1}^n=\left(2e^{-t/m}\right),\quad \left(\delta_\alpha\right)_{\alpha=1}^n=\left(\frac 32l(k+1)m\kappa  e^{2\kappa t}\right).
\end{aligned}
\end{align*}
The index $l$ repeated $m_l$ times, satisfy the hypotheses of Lemma \ref{lem:product-diff} so that
\begin{align*}
&\left| \prod_{l=1}^k \left( \frac{\theta_p^{(l)}(m,t) - \theta_i^{(l)}(m,t)}{(l-1)!} \right)^{m_l}
-\prod_{l=1}^k \left( \frac{\theta_p^{(l)}(0,t) - \theta_i^{(l)}(0,t)}{(l-1)!} \right)^{m_l} \right| \\
& \hspace{0.5cm}  =\left|a_1\cdots a_n-b_1\cdots b_n\right|\\
&  \hspace{0.5cm} \le \left(\sum_{i=1}^n\varepsilon_i\right)c_1\cdots c_n+\left(\sum_{i=1}^n\delta_i\right)d_1\cdots d_n\\
&  \hspace{0.5cm} \le 2k c_1\cdots c_ne^{-t/m}+\frac 32k(k+1)m\kappa e^{2\kappa t} d_1\cdots d_n\\
& \hspace{0.5cm} \le 2k  \Big (2\kappa+\mathcal{D}(\Omega^0)+\mathcal{D}(\mathcal{V})+\frac{9}{8m} \Big )^k \Big (1+\frac tm \Big)^ke^{-t/m}  +\frac 32 k(k+1)m\kappa e^{2\kappa t} \\
&  \hspace{0.7cm} \times \left( \Big(2\kappa+\mathcal{D}(\Omega^0)+\mathcal{D}(\mathcal{V})+\frac{9}{8m} \Big )^k \Big (1+\frac tm \Big )^k e^{-t/m}+(2\kappa+\mathcal{D}(\mathcal{V}))^k(1-e^{-t/m})\right).
\end{align*}
We estimate 
\begin{align*}
    &\prod_{l=1}^k \frac{1}{m_l!}\left(\frac{ \theta_p^{(l)}(m,t) -\theta_i^{(l)}(m,t)}{l!}\right)^{m_l}-\prod_{l=1}^k \frac{1}{m_l!}\left(\frac{ \theta_p^{(l)}(0,t) -\theta_i^{(l)}(0,t)}{l!}\right)^{m_l}\\
    & \hspace{0.5cm}   \le 2 k \Big (2\kappa+\mathcal{D}(\Omega^0)+\mathcal{D}(\mathcal{V})+\frac{9}{8m} \Big)^k \Big (1+\frac tm \Big)^k e^{-t/m}\prod_{l=1}^k\frac{1}{m_l! l^{m_l}}\\
    & \hspace{0.7cm} +\frac 32k(k+1)m\kappa e^{2\kappa t} \Big(2\kappa+\mathcal{D}(\Omega^0)+\mathcal{D}(\mathcal{V})+\frac{9}{8m} \Big)^k \Big (1+\frac tm \Big)^k e^{-t/m}\prod_{l=1}^k\frac{1}{m_l! l^{m_l}}\\
    & \hspace{0.7cm} +\frac 32k(k+1)m\kappa e^{2\kappa t} \Big (2\kappa+\mathcal{D}(\Omega^0)+\mathcal{D}(\mathcal{V}) \Big)^k(1-e^{-t/m})\prod_{l=1}^k\frac{1}{m_l! l^{m_l}}
\end{align*}
to find 
\begin{align}\label{eq:FdB-gigantic-1}
\begin{aligned}
&\Bigg|\frac{\kappa}{N}\sum_{p=1}^N\sum_{\substack{m_1,\cdots,m_k\in\mathbb{Z}_{\ge 0}\\ \sum_{\ell=1}^k \ell m_\ell =k}} k! \sin^{(\sum_{l=1}^km_l)}(\theta_p(m,t)-\theta_i(m,t))\\
& \hspace{0.7cm} \times \left(\prod_{l=1}^k \frac{1}{m_l!}\left(\frac{ \theta_p^{(l)}(m,t) -\theta_i^{(l)}(m,t)}{l!}\right)^{m_l}-
\prod_{l=1}^k \frac{1}{m_l!}\left(\frac{ \theta_p^{(l)}(0,t) -\theta_i^{(l)}(0,t)}{l!}\right)^{m_l}\right)\Bigg|\\
& \hspace{0.5cm}  \le 2\kappa k!k(2\kappa+\mathcal{D}(\Omega^0)+\mathcal{D}(\mathcal{V})+\frac{9}{8m})^k(1+\frac tm)^k e^{-t/m}\\
& \hspace{0.7cm} + \frac 32 (k+1)!km\kappa^2e^{2\kappa t} \Big(2\kappa+\mathcal{D}(\Omega^0)+\mathcal{D}(\mathcal{V}) \Big)^k(1-e^{-t/m}))\\
& \hspace{0.7cm}  + \frac 32 (k+1)!km\kappa^2e^{2\kappa t} \Big (2\kappa+\mathcal{D}(\Omega^0)+\mathcal{D}(\mathcal{V}) \Big)^k(1-e^{-t/m})\\
& \hspace{0.5cm}  \le (k+1)! \Big(2\kappa+\mathcal{D}(\Omega^0)+\mathcal{D}(\mathcal{V})+\frac{9}{8m} \Big)^{k+1}(1+\frac tm)^k e^{-t/m}\\
& \hspace{0.7cm}  +\frac 34 (k+1)!k m\kappa e^{2\kappa t} \Big(2\kappa+\mathcal{D}(\Omega^0)+\mathcal{D}(\mathcal{V})+\frac{9}{8m} \Big)^{k+1} \Big (1+\frac tm \Big)^k e^{-t/m}\\
&\hspace{0.7cm}  +\frac 34 (k+1)!k m\kappa e^{2\kappa t}(2\kappa+\mathcal{D}(\Omega^0)+\mathcal{D}(\mathcal{V}))^{k+1}(1-e^{-t/m}).
\end{aligned}
\end{align}
We also estimate
\begin{align}\label{eq:FdB-gigantic-2}
\begin{aligned}
    &\frac{\kappa}{N}\sum_{p=1}^N\sum_{\substack{m_1,\cdots,m_k\in\mathbb{Z}_{\ge 0}\\ \sum_{\ell=1}^k \ell m_\ell =k}} k! \left|\sin^{(\sum_{l=1}^km_l)}(\theta_p(m,t)-\theta_i(m,t)) - \sin^{(\sum_{l=1}^km_l)}(\theta_p(0,t)-\theta_i(0,t))\right|\\
    &\hspace{0.5cm} \times  \prod_{l=1}^k \frac{1}{m_l!}\left(\frac{| \theta_p^{(l)}(0,t) -\theta_i^{(l)}(0,t)|}{l!}\right)^{m_l} \overset{\mathclap{\eqref{eq:C0-approx-rel},\eqref{eq:1st-deriv-abs-bound}}}{\le} m\kappa e^{2\kappa t}k!(2\kappa+\mathcal{D}(\Omega^0)+\mathcal{D}(\mathcal{V}))^{k+1},
\end{aligned}
\end{align}
and
\begin{equation}\label{eq:FdB-gigantic-3}
    \left|m\theta^{(k+2)}_i(0,t)-m\theta^{(k+2)}_j(0,t)\right| ~\overset{\mathclap{\eqref{eq:1st-deriv-abs-bound}}}{\le} (k+1)!2m\kappa (2\kappa+\mathcal{D}(\Omega^0)+\mathcal{D}(\mathcal{V}))^{k+1}.
\end{equation}
We harvest estimates \eqref{eq:FdB-gigantic-1}, \eqref{eq:FdB-gigantic-2}, and \eqref{eq:FdB-gigantic-3} into \eqref{eq:FdB-gigantic} to get 
\begin{align*}
 &\Big|m\theta^{(k+2)}_i(m,t)-m\theta^{(k+2)}_j(m,t)-m\theta^{(k+2)}_i(0,t)+m\theta^{(k+2)}_j(0,t)\\
 & \hspace{3cm} +\theta^{(k+1)}_i(m,t)-\theta^{(k+1)}_j(m,t)-\theta^{(k+1)}_i(0,t)+\theta^{(k+1)}_j(0,t)\Big|\\
& \hspace{0.5cm} \le 2(k+1)!(2\kappa+\mathcal{D}(\Omega^0)+\mathcal{D}(\mathcal{V})+\frac{9}{8m})^{k+1}(1+\frac tm)^k e^{-t/m}\\
&  \hspace{0.7cm}  +\frac 32(k+1)!km\kappa e^{2\kappa t}(2\kappa+\mathcal{D}(\Omega^0)+\mathcal{D}(\mathcal{V})+\frac{9}{8m})^{k+1}(1+\frac tm)^k e^{-t/m}\\
&  \hspace{0.7cm}  +\frac 32(k+1)!km\kappa e^{2\kappa t}(2\kappa+\mathcal{D}(\Omega^0)+\mathcal{D}(\mathcal{V}))^{k+1}(1-e^{-t/m})\\
&  \hspace{0.7cm}  +2m\kappa e^{2\kappa t}k!(2\kappa+\mathcal{D}(\Omega^0)+\mathcal{D}(\mathcal{V}))^{k+1} + (k+1)!2m\kappa (2\kappa+\mathcal{D}(\Omega^0)+\mathcal{D}(\mathcal{V}))^{k+1}\\
&  \hspace{0.5cm}  \le 2(k+1)!(2\kappa+\mathcal{D}(\Omega^0)+\mathcal{D}(\mathcal{V})+\frac{9}{8m})^{k+1}(1+\frac tm)^{k+1} e^{-t/m} +\frac 32(k+2)! m\kappa e^{2\kappa t} \\
&  \hspace{0.7cm}  \times \left((2\kappa+\mathcal{D}(\Omega^0)+\mathcal{D}(\mathcal{V})+\frac{9}{8m})^{k+1}(1+\frac tm)^k e^{-t/m}+(2\kappa+\mathcal{D}(\Omega^0)+\mathcal{D}(\mathcal{V}))^{k+1}\right),
\end{align*}
where we used
\begin{align*}
\begin{aligned}
& \frac 32(k+1)!k+2k!+2(k+1)!  \\
& \hspace{0.5cm} \le \frac 32(k+1)!k+(k+1)!+2(k+1)! =\frac 32(k+1)!k+3(k+1)!=\frac 32 (k+2)!.
\end{aligned}
\end{align*}
However, the initial conditions satisfy
\begin{align*}
    \left|\theta_i^{(k+1)}(m,0)-\theta_j^{(k+1)}(m,0)-\theta_i^{(k+1)}(0,0)+\theta_j^{(k+1)}(0,0)\right|\overset{\eqref{sum-lem:deriv-bound}}{\le} 2k!(2\kappa+\mathcal{D}(\Omega^0)+\mathcal{D}(\mathcal{V})+\frac{9}{8m})^{k+1}.
\end{align*}
At last, we apply the Duhamel principle with two preceding estimates to close the induction:
\begin{align*}
\begin{aligned}
    &\left|\theta^{(k+1)}_i(m,t)-\theta^{(k+1)}_j(m,t)-\theta^{(k+1)}_i(0,t)+\theta^{(k+1)}_j(0,t)\right|\\
    &  \hspace{0.5cm}  \le 2k!(2\kappa+\mathcal{D}(\Omega^0)+\mathcal{D}(\mathcal{V})+\frac{9}{8m})^{k+1}e^{-t/m}\\
    &  \hspace{0.7cm}  +\frac{e^{-t/m}}{m}\int_0^t 2(k+1)!(2\kappa+\mathcal{D}(\Omega^0)+\mathcal{D}(\mathcal{V})+\frac{9}{8m})^{k+1}(1+\frac sm)^kds\\
    &  \hspace{0.7cm}  +\frac{e^{-t/m}}{m}\int_0^t \frac 32 m\kappa \cdot  (k+2)!e^{2\kappa t}(2\kappa+\mathcal{D}(\Omega^0)+\mathcal{D}(\mathcal{V})+\frac{9}{8m})^{k+1}(1+\frac sm)^k ds\\
&  \hspace{0.7cm} +\frac{e^{-t/m}}{m}\int_0^t \frac 32 m\kappa \cdot  (k+2)!e^{2\kappa t}(2\kappa+\mathcal{D}(\Omega^0)+\mathcal{D}(\mathcal{V}))^{k+1}e^{s/m}ds\\
  &  \hspace{0.5cm}  \le 2k!(2\kappa+\mathcal{D}(\Omega^0)+\mathcal{D}(\mathcal{V})+\frac{9}{8m})^{k+1}(1+\frac tm)^{k+1}e^{-t/m}\\
    &  \hspace{0.7cm} + \frac 32 m\kappa \cdot  (k+2)!e^{2\kappa t}(2\kappa+\mathcal{D}(\Omega^0)+\mathcal{D}(\mathcal{V})+\frac{9}{8m})^{k+1}(1+\frac tm)^{k+1}e^{-t/m}\\
&  \hspace{0.7cm} + \frac 32 m\kappa \cdot  (k+2)!e^{2\kappa t}(2\kappa+\mathcal{D}(\Omega^0)+\mathcal{D}(\mathcal{V}))^{k+1}(1-e^{-t/m}).
\end{aligned}
\end{align*}
A calculation similar to \eqref{eq:FdB-gigantic}, \eqref{eq:FdB-gigantic-1}, \eqref{eq:FdB-gigantic-2}, and \eqref{eq:FdB-gigantic-3} gives
\begin{align*}
    &\left|m\theta^{(k+2)}_i(m,t)-m\theta^{(k+2)}_i(0,t)+\theta^{(k+1)}_i(m,t)-\theta^{(k+1)}_i(0,t)\right|\\
&  \hspace{0.5cm} \le (k+1)!(2\kappa+\mathcal{D}(\Omega^0)+\mathcal{D}(\mathcal{V})+\frac{9}{8m})^{k+1}(1+\frac tm)^{k+1} e^{-t/m} +\frac 34 m\kappa \cdot  (k+2)!e^{2\kappa t} \\
&  \hspace{0.7cm}  \times \left((2\kappa+\mathcal{D}(\Omega^0)+\mathcal{D}(\mathcal{V})+\frac{9}{8m})^{k+1}(1+\frac tm)^k e^{-t/m}+(2\kappa+\mathcal{D}(\Omega^0)+\mathcal{D}(\mathcal{V}))^{k+1}\right).
\end{align*}
Compared to \eqref{eq:FdB-gigantic}, there are only half as many terms. Coupled with the initial conditions \eqref{eq:1st-deriv-abs-bound} and $\eqref{eq:2nd-deriv-initial-power-bound}_2$, which give
\[
|\theta_i^{(k+1)}(0,0)|,~|\theta^{(k+1)}_i(m,0)|\le k!\left(2\kappa+2\|\Omega^0\|+2\|\mathcal{V}\|+ \frac 9{8m}\right)^{k+1}.
\]
By the Duhamel principle, we have
\begin{align*}
    &\left|\theta^{(k+1)}_i(m,t)-\theta^{(k+1)}_i(0,t)\right|\\
    &  \hspace{0.5cm}  \le k!(2\kappa+\|\Omega^0\|+\|\mathcal{V}\|+\frac{9}{8m})^{k+1}e^{-t/m}\\
    &  \hspace{0.7cm}  +\frac{e^{-t/m}}{m}\int_0^t (k+1)!(2\kappa+\mathcal{D}(\Omega^0)+\mathcal{D}(\mathcal{V})+\frac{9}{8m})^{k+1}(1+\frac sm)^kds\\
    &  \hspace{0.7cm}  +\frac{e^{-t/m}}{m}\int_0^t \frac 34 m\kappa \cdot  (k+2)!e^{2\kappa t}(2\kappa+\mathcal{D}(\Omega^0)+\mathcal{D}(\mathcal{V})+\frac{9}{8m})^{k+1}(1+\frac sm)^k ds\\
&  \hspace{0.7cm} +\frac{e^{-t/m}}{m}\int_0^t \frac 34 m\kappa \cdot  (k+2)!e^{2\kappa t}(2\kappa+\mathcal{D}(\Omega^0)+\mathcal{D}(\mathcal{V}))^{k+1}e^{s/m}ds\\
    &  \hspace{0.5cm}  \le k!(2\kappa+\|\Omega^0\|+\|\mathcal{V}\|+\frac{9}{8m})^{k+1}(1+\frac tm)^{k+1}e^{-t/m}\\
    &  \hspace{0.7cm}  + \frac 98m\kappa \cdot  (k+1)!e^{2\kappa t}(2\kappa+\mathcal{D}(\Omega^0)+\mathcal{D}(\mathcal{V})+\frac{9}{8m})^{k+1}(1+\frac tm)^{k+1}e^{-t/m}\\
&  \hspace{0.7cm} + \frac 34m\kappa \cdot  (k+2)!e^{2\kappa t}(2\kappa+\mathcal{D}(\Omega^0)+\mathcal{D}(\mathcal{V}))^{k+1}(1-e^{-t/m}),
\end{align*}
as stated. This completes the proof of Proposition \ref{prop:mto0quant}.

%
%
%
%

\end{document}